\documentclass{amsart}
\usepackage{blindtext}
\usepackage[french,english]{babel}
\usepackage{prodint}

\makeatletter
\renewcommand\part{%
   \if@noskipsec \leavevmode \fi
   \par
   \addvspace{4ex}%
   \@afterindentfalse
   \secdef\@part\@spart}

\def\@part[#1]#2{%
    \ifnum \c@secnumdepth >\m@ne
      \refstepcounter{part}%
      \addcontentsline{toc}{part}{\thepart\hspace{1em}#1}%
    \else
      \addcontentsline{toc}{part}{#1}%
    \fi
    {\parindent \z@ \raggedright
     \interlinepenalty \@M
     \normalfont
     \ifnum \c@secnumdepth >\m@ne
    \bfseries \partname\nobreakspace\thepart
       \par\nobreak
     \fi
      \bfseries #1%
     \par}%
    \nobreak
    \vskip 3ex
    \@afterheading}
\def\@spart#1{%
    {\parindent \z@ \raggedright
     \interlinepenalty \@M
     \normalfont
     \huge \bfseries #1\par}%
     \nobreak
     \vskip 3ex
     \@afterheading}
\makeatother

\usepackage{bm}

 \usepackage{upgreek}
 \usepackage{pifont}
   \usepackage{dsfont}

\usepackage{amsopn, amsthm, amsgen, amscd,amsmath, amssymb}
\DeclareMathAlphabet{\mathbbm}{U}{bbm}{m}{n}
\usepackage{color}
\usepackage{tikz}
  \usepackage{graphicx}
\usepackage{epstopdf}
\usepackage[small,nohug,heads = LaTeX]{diagrams} \diagramstyle[labelstyle=\scriptstyle]

\definecolor{CadetBlue}{cmyk}{0.62, 0.57, 0.23, 0 }
\definecolor{black}{cmyk}{1, 0.5, 0, 0 }
\definecolor{RedViolet}{cmyk}{0.07, 0.9, 0, 0.34 }
\definecolor{SeaGreen}{cmyk}{0.69, 0, 0.5, 0}

\DeclareMathAlphabet{\mathpzc}{OT1}{pzc}{m}{it}

\newcommand{\R}{\mathbb R}
\newcommand{\C}{\mathbb C}
\newcommand{\D}{\mathbb D}

\newcommand{\F}{\mathbb F}

\newcommand{\K}{\mathbb K}
\newcommand{\N}{\mathbb N}
\newcommand{\PR}{\mathbb P}
\newcommand{\Q}{\mathbb Q}
\newcommand{\Z}{\mathbb Z}

\newcommand{\T}{\mathbb T}
\newcommand{\I}{\mathbb I}

\newcommand{\SI}{\mathbb S}

\newtheorem{theo}{Theorem}
\numberwithin{theo}{section}
\newtheorem{lemm}{Lemma}
\numberwithin{lemm}{section}
\newtheorem{prop}{Proposition}
\numberwithin{prop}{section}
\newtheorem{coro}{Corollary}
\numberwithin{coro}{section}

\newtheorem*{maintheo}{Main Theorem}

\newtheorem*{conj1}{Main Conjecture 1: Hilbert Class Fields}
\newtheorem*{conj2}{Main Conjecture 2: Ray Class Fields}

\numberwithin{rem}{section}

\theoremstyle{definition}

\numberwithin{defi}{section}

\numberwithin{axio}{section}

\theoremstyle{remark}

\numberwithin{nota}{section}
\newtheorem{exam}{Example}
\numberwithin{exam}{section}

\newtheorem{case}{Case}
\newtheorem{case2}{Case}
\newtheorem{case3}{Case}

\newtheorem{note}{Note}
\numberwithin{note}{section}

\numberwithin{aside}{section}
\newtheorem{rema}{Remark}
\numberwithin{rema}{section}

\title[Quantum Drinfeld Modules]{Quantum Drinfeld Modules and Ray Class Fields of Real Quadratic Global Function Fields}
\author{L. Demangos}
\address{Xi'an Jiaotong - Liverpool University, Department of Mathematical Sciences, Mathematics Building Block B, 111 Ren'ai Road, Suzhou Dushu Lake Science
and Education Innovation District, Suzhou Industrial Park, Suzhou, Peoples Republic of China, 215123}
\email{Luca.Demangos@xjtlu.edu.cn}
\author{T.M. Gendron}
\address{Instituto de Matem\'{a}ticas -- Unidad Cuernavaca, Universidad
Nacional Aut\'{o}noma de M\'{e}xico, Av. Universidad S/N, C.P. 62210
Cuernavaca, Morelos, M\'{e}xico}
\email{tim@matcuer.unam.mx}
\subjclass[2010]{Primary 11R37, 11R80, 11R58, 11F03; Secondary 11K60}
\date{\today}
\keywords{quantum Drinfeld module, ray class field, function field arithmetic}
\begin{document}
\vspace{2cm}

 \maketitle

 \begin{abstract} This is the second in a series of two papers presenting
a solution to Hilbert's 12th problem for real quadratic function fields in positive characteristic, in the sense of proving an analog
of the Theorem of Weber-Fueter.  We also
offer a conjectural treatment of the number field case using quasicrystal counterparts \cite{GLL} of the constructions used in function fields.

  \end{abstract}
  

\section*{Introduction}

In \cite{DGIII}, an explicit description was given of the Hilbert class field $H_{\mathcal{O}_{K}}$ associated to the integral closure $\mathcal{O}_{K}$ of $A=\F_{q}[T]$ 
in a real quadratic extension  $K$ of $k=\F_{q}(T)$.  By real, we mean that $K\subset k_{\infty}$ = the analytic completion of $k$ with respect to the valuation associated to $\infty\in\PR^{1}$. This explicit description of $H_{\mathcal{O}_{K}}$  uses the
values of a multi-valued, modular invariant function
 \[ j^{\rm qt}: k_{\infty}/{\rm GL}_{2}(A)\multimap k_{\infty}\] called the {\it quantum modular invariant}.
  More specifically, for a quadratic unit $f\in K$, it was shown in  that $j^{\rm qt}(f)$ is a finite set.  Then 
 we have the explicit description
\begin{align}\label{HCF3rdPaper} H_{\mathcal{O}_{K}} =  K\left( \prod_{\upalpha\in j^{\rm qt}(f) } \upalpha \right).\end{align}

Now assume that ${\rm char}(\F_{q})$ is odd.  The purpose of this paper is to prove a generation theorem for ray class fields involving a quantum notion of Drinfeld module for quadratic extensions  Thus, together with
(\ref{HCF3rdPaper}), the results of this paper give the analog of the Theorem of Weber-Fueter \cite{Sil}, the latter theorem being the main inspiration for 
 Hilbert's 12th Problem  \cite{Schapp}.

We give now a more detailed description of the main theorem.   Let $f\in\mathcal{O}^{\times}_{K}$ be non-constant having norm $=1$, and consider the order
\[ \mathcal{O}_{f} = \F_{q}[T,f,f^{-1}]  \subset\mathcal{O}_{K}.\]
 For any ideal $\mathfrak{M}\subset\mathcal{O}_{K}$, denote
 $\mathfrak{M}_{f} = \mathfrak{M}\cap  \mathcal{O}_{f}$ and consider the family 
 \[ \mathcal{M}_{f} = \left\{ \mathfrak{M}_{f}\subset \mathcal{O}_{f}\; : \;\; f\equiv 1\mod \mathfrak{M}\subset\mathcal{O}_{K}\right\}.\]
 The (narrow) ray class field $K_{\mathfrak{M}_{f}} $ associated to $\mathfrak{M}_{f}\in  \mathcal{M} $ is called a {\it unit ray class field}.  
 If we denote by $K^{\rm ab}_{\infty_{1}, \infty_{2}}$ the maximal abelian extension of $K$ split over the
 places $\infty_{1},\infty_{2}$ of $K$ lying over $\infty$, then (see Theorem \ref{cofinalitytheo} of \S  \ref{RayClassFieldSection})
 \[ K^{\rm ab}_{\infty_{1}, \infty_{2}} = \bigcup_{\mathfrak{M}_{f}\in \mathcal{M}} K_{\mathfrak{M}_{f}}.\]

The explicit description of the $K_{\mathfrak{M}_{f}} $ is given by the  {\it quantum exponential
function}
\[ \exp  ^{\rm qt} 
:\C_{\infty}\multimap \C_{\infty}\]
associated to $f$ 
(here $\C_{\infty}$ is the analog of the complex numbers in function field arithmetic).  
It naturally leads to an associated {\it quantum Drinfeld module}, denoted $\uprho^{\rm qt}  $, and for which the corresponding notion of modular 
invariant is given by $j^{\rm qt}(f)$.  A point of a quantum Drinfeld module is multi-valued, of the shape 
$ z^{\rm qt}=\{ z_{0},\dots ,z_{d-1}\}$, where $d-1$ is the genus of the curve associated to $K$.  See \S \ref{QDMSection}.  

To generate $K_{\mathfrak{M}_{f}}$, we use the group  
$  \uprho^{\rm qt}  [\mathfrak{M}_{f}]
$ of {\it quantum
$\mathfrak{M}_{f}$ torsion points} of $\uprho^{\rm qt}  $,
defined at the end of  \S  \ref{RayClassFieldSection}, c.f.\ (\ref{QTDef}).
 Then, if 
we denote by ${\sf Tr}(z^{\rm qt}) = \sum z_{i}$  the sum of the elements of $z^{\rm qt}$, we have 
\begin{maintheo}  Let ${\rm id}\not=  \varphi \in {\rm Gal}(K/\F_{q}(T))$, lifted to $K_{\mathfrak{M}_{f}}$.  Then
\[ K_{\mathfrak{M}_{f}}=H_{\mathcal{O}_{f}} \big( {\sf Tr}(\uprho^{\rm qt}  [\mathfrak{M}_{f}] ),    {\sf Tr}(\uprho^{\rm qt}  [\mathfrak{M}_{f}])^{  \varphi} 
\big)
.\]
\end{maintheo}
The proof of this result is spread over \S\S \ref{ellipticsection}, \ref{genus2section} and \ref{generalcase}; see Theorems \ref{ellipticrayclasstheorem}, \ref{genus2rayclasstheorem} and \ref{generalrayclasstheorem}
therein.   
In \S \ref{CharZeroConj}, we present a conjectural adaptation of the results of this paper and \cite{DGIII} to the number field case. 
This paper relies heavily on \cite{DGVI}, which presents Hayes Theory for orders; this material was originally included as a set of appendices
in the arXiv version of the present paper \cite{DGArXiv}.
\vspace{5mm}

\noindent {\bf Acknowledgements:}  
We have benefited
from conversations with Ernst-Ulrich Gekeler,   Ignazio Longhi, Federico Pellarin and Richard Pink.  We would also like to thank the referees for many useful suggestions.

 \section{Quantum Drinfeld Modules}\label{QDMSection}

In this section we introduce the multi-valued notions of quantum exponential and its associated quantum Drinfeld module as renormalized limits of objects
defined using diophantine approximations, in the spirit of the definition of the quantum modular invariant \cite{DGIII}.  
The underlying philosophy -- diophantine approximations give rise to quantum analogs of classical objects in number theory --  serves as a useful guide to the development  and understanding of the theory: in particular, it explains why the quantum objects proposed here are multi-valued.

We recall basic notation established in \cite{DGIII}: $\F_{q}$ is the field with $q=p^{n}$ elements, $p$ a prime, $k=\F_{q}(T)$, $A=\F_{q}[T]$.  Throughout the paper we will also assume that $p\not=2$.
    We recall that $k$ is the function field of $\PR^{1}$; we denote by $k_{\infty}$ the completion of $k$ with respect to the place $\infty\in\PR^{1}$ and
    by $\C_{\infty}$ the completion of the algebraic closure $\overline{k_{\infty}}$.
  
  Let $f\in k_{\infty}$ be a quadratic unit whose minimal polynomial has the shape 
    \begin{align}\label{minpolyoff} X^{2}-{\tt a}X-{\tt b}, \end{align}
    where ${\tt a}\in A$ is assumed to be a {\it monic} polynomial in $T$ of degree $d$ and ${\tt b}\in\F_{q}^{\times}$.  Note that for all $n\in\N$, $f^{n}$ also satisfies
    a minimal polynomial with linear term ${\tt a}_{n}$  monic.  Indeed, the monicity condition on ${\tt a}$ implies that the coefficient of the first non zero term of the Laurent series expansion
    of $f$ in $T^{-1}$ is 1, which implies the same for $f^{n}$, which, in turn, implies that ${\tt a}_{n}$ is monic.

    We denote by $K=k(f)$ the associated quadratic extension and
by $\mathcal{O}_{K}$ the integral closure of $A$ in $K$.   
 The extension $K$ is real in the sense that it is contained in $k_{\infty}$, or equivalently -- if we denote by $\Upsigma_{K}\rightarrow\PR^{1}$ the 
 degree $2$ morphism of curves inducing $K/k$ -- 
  there are two places $\infty_{1},\infty_{2}\in\Upsigma_{K}$ lying over $\infty\in\PR^{1}$.  We denote
    by $A_{\infty_{1}}$, $A_{\infty_{2}}\subset \mathcal{O}_{K}$ the Dedekind domains of functions regular outside of $\infty_{1},\infty_{2}$.

Fix a {\it fundamental} unit  $f_{0}$, also assumed to satisfy a polynomial of the shape (\ref{minpolyoff}), where its coefficients are denoted ${\tt a}_{0}, {\tt b}_{0}$ and
${\tt a}_{0}$ is monic of degree $d_{0}$.   
  We have the explicit
    descriptions \cite{DGIII}:
    \begin{align}\label{expdesc}  A_{\infty_{1}}= \F_{q}[f_{0}, f_{0}T,\dots ,f_{0} T^{d_{0}-1}],\quad A_{\infty_{2}}= \F_{q}[f_{0}^{-1}, f_{0}^{-1}T,\dots ,f_{0}^{-1}T^{d_{0}-1 }].  \end{align}
    As in \cite{DGIII} we choose the place $\infty_{1}$ and work with $A_{\infty_{1}}$;
    the discussion for the choice of place $\infty_{2}$ is identical.  For a general unit of the form
$  f:=f_{0}^{n}$, 
the ring
   \begin{align}\label{forder} A_{f}:=\F_{q}[f, fT,\dots ,fT^{d-1}]\subset A_{\infty_{1}}, \quad d:=nd_{0}, \end{align}
  is an order in $A_{\infty_{1}}=A_{f_{0}} $.
    See \S 1 of \cite{DGVI} for facts regarding
      orders in function fields.  
 

Consider the ideals 
\begin{align}\label{aidef} \mathfrak{a}_{i} =\mathfrak{a}_{f,i} =(f, fT,\dots ,fT^{i})\subset A_{f}, \quad i=0,1,\dots ,d-1 .\end{align}

\begin{prop}\label{PowerProp} For all $i$, $\mathfrak{a}_{i}= \mathfrak{a}_{d-1}^{d-i}$.
      \end{prop}
      
 \begin{proof} See Proposition 5.1 of \cite{DGVI}.\end{proof}




In \cite{DGIII}, the modular invariant $j^{\rm qt}(f)$ of the quantum torus
$ \T_{f} =\C_{\infty}/ \langle 1, f \rangle$
 was introduced as a multi-valued and discontinuous function of $k_{\infty}$. Its multi-valued
and discontinuous character arose through its definition as a limit of functions $j_{\upvarepsilon}$ defined using the set $ \Uplambda_{\upvarepsilon}(f)=\{\uplambda \in A,\| \uplambda f\|<\upvarepsilon\}$
of $\upvarepsilon$ diophantine approximations of $f$;
 here $\| x\|$ = the distance of $x$ to the nearest element of $A$.
It was then proved, through a renormalization procedure, that  \[ j^{\rm qt}(f) = \left\{  j(\mathfrak{a}_{i})\right\}, \quad i=0,\dots ,d-1,\]
where $ j(\mathfrak{a}_{i})$ is the modular invariant of an ideal $\mathfrak{a}_{i}\subset A_{\infty_{1}}$.  In what follows, it is shown that a similar procedure leads to a quantum notion of
exponential
   \[ \exp^{\rm qt}  :\C_{\infty}\multimap \C_{\infty},\] whose multivalues will  generate ray class fields over the Hilbert class field $H_{\mathcal{O}_{K}}$.
   

 Define 
${\tt Q}_{n}\in A$ by 
$ {\tt Q}_{0}=1, {\tt Q}_{1}={\tt a},\dots , {\tt Q}_{n+1}={\tt a}{\tt Q}_{n} +{\tt b}{\tt Q}_{n-1}$,
${\tt a},{\tt b}$ are as in (\ref{minpolyoff}).  
If $f^{\ast}$ denotes the Galois conjugate of $f$, 
we may assume $|f|>|f^{\ast}|$, and then
 $|f|=|{\tt a}|=q^{d}$ and $|f^{\ast}|=q^{-d}$.  Let ${\tt D}={\tt a}^{2}+4{\tt b}$ be the discriminant.
Using Binet's formula
\begin{align}\label{Binetform}{\tt Q}_{n} =\frac{ f^{n+1}-(f^{\ast})^{n+1}}{\sqrt{{\tt D}}},\quad n=0,1,\dots ,\end{align}
one may show that
$\| {\tt Q}_{n}f\| =  q^{ -(n+1) d } $,
from which it follows that the set \begin{align*}
\mathcal{B}=\{ T^{d-1}{\tt Q}_{0}, \dots , T{\tt Q}_{0}, {\tt Q}_{0} ; T^{d-1}{\tt Q}_{1}, \dots , T{\tt Q}_{1}, {\tt Q}_{1};\dots  \}\end{align*}
forms an $\F_{q}$ basis of $A$.
Write 
$\mathcal{B}(i) = \{ T^{d-1}{\tt Q}_{i},\dots , {\tt Q}_{i}\}$
 for the $i$th block of $\mathcal{B}$ and for $0\leq \tilde{d}\leq d-1$, denote
$ \mathcal{B}(i)_{\tilde{d}} = \{T^{\tilde{d}}{\tt Q}_{i},\dots , {\tt Q}_{i}\}$.
Then (see \cite{DGIII}, Lemma 1) 
\begin{align}\label{explicitformoflambda} 
 \Uplambda_{q^{-Nd-l}}(f )={\rm span}_{\F_{q}} ( \mathcal{B}(N)_{d-1-l},\mathcal{B}(N+1),\dots   ) .
\end{align}

   The $\boldsymbol\upvarepsilon$ {\bf {\em exponential function}} is the additive function
   \[ \exp_{\Uplambda_{\upvarepsilon}(f)}(z) = z\prod_{0\not=\uplambda\in\Uplambda_{\upvarepsilon}(f)} \left(1- \frac{z}{\uplambda}\right). \]
Unlike the functions $j_{\upvarepsilon}$,  the $\upvarepsilon$ exponential functions have trivial limit:   $ \lim_{\upvarepsilon\rightarrow 0}\exp_{\Uplambda_{\upvarepsilon}(f)}(z)$ $=z $,
   since, by (\ref{explicitformoflambda}), $\bigcap \Uplambda_{\upvarepsilon}(f)=0$ and 
$\lim_{\upvarepsilon\rightarrow 0}\inf \{ |\uplambda| \,  |\;\uplambda\in \Uplambda_{\upvarepsilon}(f)-0\}  \rightarrow\infty $.
      On the other hand, there is a natural normalization of  $\Uplambda_{\upvarepsilon}(f)$ by a transcendental factor $\upxi_{\upvarepsilon}\in \C_{\infty}$,
      defined below in (\ref{formulaxi}), so that if we denote 
     $ \breve{\Uplambda}_{\upvarepsilon}(f):=  \upxi_{\upvarepsilon} \Uplambda_{\upvarepsilon}(f) $ and \[ e_{\upvarepsilon}(z):=
      \exp_{\breve{\Uplambda}_{\upvarepsilon}(f)}(z) = z\prod_{0\not=\uplambda\in\breve{\Uplambda}_{\upvarepsilon}(f)} \left(1- \frac{z}{\uplambda}\right)  = \upxi_{\upvarepsilon}\exp_{\Uplambda_{\upvarepsilon}(f)}( \upxi_{\upvarepsilon}^{-1} z) ,
      \]
      then the $\upvarepsilon\rightarrow 0$ limit of the $e_{\upvarepsilon}(z)$ 
      will produce a nontrivial multi-valued function.
      The period $\upxi_{\upvarepsilon}$ is defined according to the analog of a classical procedure, which may be found, for example, in the statement of Theorem 7.10.10 of \cite{Goss}. Before proceeding further, we will need to 
      fix a  uniformizer $\uppi\in k_{\infty}$.   Since $f$ is of degree $d$, we may write  $f/T^{d}=1+u$,  whence $   f^{-1/d} =T\sum_{k\geq 0}\binom{1/d}{k}  u^{k} \in k_{\infty}$.
    Define
     \begin{align}\label{defnuniform}  \uppi := \left\{ \begin{array}{ll}
     f^{-1/d}   &\text{if $(d,p)=1$} \\
     (fT)^{-1/(d+1)} & \text{otherwise}.
     \end{array}\right. \end{align}
     Note that in case $(d,p)\not=1$, $(d+1,p)=1$ and so the root $ (fT)^{1/(d+1)}$ also exists in $k_{\infty}$.
     In either case, $|\uppi |=q^{-1}$, so $\uppi$ is a uniformizer for $k_{\infty}$.

     Now following the statement of Theorem 7.10.10 of \cite{Goss},
     the transcendental element $\upxi_{\upvarepsilon}\in \C_{\infty}$ is defined (up to multiplication by $(q-1)$th roots of unity) by
     \begin{align}\label{formulaxi} \upxi_{\upvarepsilon}^{q-1} :=\upeta_{\upvarepsilon}^{-1}\times \uppi^{-t_{\upvarepsilon}}\times u^{q-1}_{\upvarepsilon},\end{align}
  where the three factors on the right hand side above  are as follows. 
  
     \vspace{3mm}
  
 \noindent {\bf 1.}  Fix the sign homomorphism ${\rm sgn}:k_{\infty}^{\times}\rightarrow \F_{q}^{\times}$
 given by ${\rm sgn}(x)=$ coefficient of the leading term in the Laurent series expansion of $x$ in $T^{-1}$.
Then (\cite{Goss},  page 222)
   \[ \upeta_{\upvarepsilon}  :=\lim_{M\rightarrow\infty}\prod_{\substack{ 0\not=\uplambda\in \Uplambda_{\upvarepsilon}(f), \\ \deg_{T}(\uplambda)= M}} {\rm sgn}(\uplambda) .\]
   
   \vspace{3mm}
 
 \noindent {\bf 2.} 
Let
 \[ Z_{\upvarepsilon}(u) := \sum_{0\not=\uplambda\in\Uplambda_{\upvarepsilon}(f)} u^{\deg_{T} (\uplambda )} \]
 where $u=q^{-s}$ and $s\in\C$.
Then the exponent of $\uppi$ is defined  \[ t_{\upvarepsilon}:= (q-1)Z'_{\upvarepsilon}(1).\] 
 
 \vspace{3mm}
 
\noindent {\bf 3.}  Finally, we define
\[ u_{\upvarepsilon} :=\lim_{M\rightarrow\infty}\prod_{\substack{0\not =\uplambda\in \Uplambda_{\upvarepsilon}(f), \\ \deg_{T}(\uplambda)\leq M}} \langle\uplambda\rangle,
\quad  \langle\uplambda \rangle := \frac{\uplambda\cdot\uppi^{\deg_{T} (\uplambda) }}{{\rm sgn}(\uplambda)} = \text{ $1$-unit part of $\uplambda$.}\]

\vspace{3mm}


Recall the ideals $\mathfrak{a}_{i} $  introduced
 in (\ref{aidef}).
When  $f=f_{0}=$ a fundamental unit,  they form \cite{DGIII} 
a cyclic subgroup $ Z\subset {\sf Cl}_{A_{\infty_{1}}} =$ the ideal class group of $A_{\infty_{1}}$.  We may now also define a period $\upxi_{i}$ for each $\mathfrak{a}_{i}$ as is done in \cite{Goss}, \S 7.10.  Denote 
by \[ e_{i}(z) = \text{ the exponential function with lattice $\Uplambda_{i}:=\upxi_{i}\mathfrak{a}_{i}\subset \C_{\infty}$},\]  which
 is the exponential Hayes module $\uprho_{i}$ with lattice $\Uplambda_{i}$.  See \cite{Hayes}, \cite{Goss}, \cite{Thak}.

When $f$ is a general unit, the statements in the previous paragraph remain true.
In particular, the ideals $\mathfrak{a}_{i}$ now form a cyclic subgroup of the ideal class group of the order, ${\sf Cl}_{A_{f}}$, see Proposition 5.1 of \cite{DGVI}.  The normalized exponentials $e_{i}(z)$ are now associated to {\it order Hayes
modules} associated to $A_{f}$, \cite{Hayes0}.
The {\bf {\em quantum exponential function}}
associated to $f$ is defined
\[  \exp^{\rm qt}  :\C_{\infty}\multimap \C_{\infty},\quad  \exp^{\rm qt}  (z) :=  
\lim_{\upvarepsilon\rightarrow 0} e_{\upvarepsilon}(z) .\]
  
 \begin{theo}\label{qtexplimit}  $ \exp^{\rm qt}  (z) =\{ e_{i}(z) |\;i =0,\dots ,d-1 \} $.
  \end{theo}

The proof of Theorem \ref{qtexplimit} is a normalized version of  the proof of Theorem 4 in \cite{DGIII}.
In particular, we will show that for $l\in \{ 0,\dots ,d-1\}$  and $\upvarepsilon_{N,l}=q^{-dN-l}$, 
$ \lim_{N\rightarrow\infty}e_{\upvarepsilon_{N,l}}(z) $ $=$ $ e_{d-1-l}(z)$.
The first step will be to relate the individual factors appearing in $\upxi_{\upvarepsilon_{N,l}}$ with their counterparts appearing
in the period $\upxi_{d-1-l}$ corresponding to the ideal $\mathfrak{a}_{d-1-l}$, defined up to $(q-1)$th roots of unity by
$ \upxi^{q-1}_{d-1-l} $ $ =$ $ \upeta^{-1}_{d-1-l}\times \uppi^{-t_{d-1-l}}\times u^{q-1}_{d-1-l}$. 
See again Theorem 7.10.10 of \cite{Goss}.

\begin{lemm}\label{signlemma} For $\upvarepsilon=\upvarepsilon_{N,l}$,  $\upeta_{\upvarepsilon}=-1=\upeta_{d-1-l}$.
\end{lemm}

\begin{proof}   Since $\prod_{c\in \F_{q}^{\times}}c=-1$, then by (\ref{explicitformoflambda}),  for all $M$
 \[ \prod_{\substack{ 0\not=\uplambda\in \Uplambda_{\upvarepsilon}(f), \\ \deg(\uplambda)= dN+M}} {\rm sgn}(\uplambda)= 
 \prod_{c\in \F_{q}^{\times}} \left(\prod_{\substack{ 0\not=\uplambda\in \Uplambda_{\upvarepsilon}(f), \\ \deg(\uplambda)= dN+M \\ {\rm sgn}( \uplambda ) =c}} c \right)=
 (-1)^{q^{M}}=-1,\]
giving $ \upeta_{\upvarepsilon}=-1$.  The second equality is proved similarly. 
 \end{proof}

\begin{lemm}\label{tlemma} For $\upvarepsilon=\upvarepsilon_{N,l}$, $ t_{\upvarepsilon} = t_{d-1-l}-d(N-1)(q-1)$.
\end{lemm}
 
\begin{proof}  By (\ref{explicitformoflambda}), 
there are $q-1$ elements of degree $dN$, $q(q-1)$ of degree $dN+1$,  etc., which gives (for $|u|<1$)
 \begin{align*} 
 Z_{\upvarepsilon}(u)  
 & = (q-1)\left(u^{dN} +qu^{dN+1}+\cdots + q^{d-1-l} u^{dN+d-1-l} +\frac{q^{d-l} u^{d(N+1)}}{1-qu} \right) . 
 \end{align*}
 For the ideal $\mathfrak{a}_{d-1-l}$, one defines (see \cite{Goss}, Definition 7.8.2) 
 \[ Z_{d-1-l}(u) =  \sum_{0\not=\upalpha\in\mathfrak{a}_{d-1-l}} u^{\deg_{T} (\upalpha )}.\]
 Using the description $\mathfrak{a}_{d-1-l} =(f,fT,\dots , fT^{d-1-l})$,  $ Z_{d-1-l}(u)$ can be written
 \begin{align}\label{formforZd-1-l}  
(q-1)\left(u^{d} +\cdots + q^{d-1-l} u^{2d-1-l} +\frac{q^{d-l} u^{2d}}{1-qu} \right)  & = u^{-(N-1)d}Z_{\upvarepsilon}(u) .
 \end{align}
 Since $t_{d-1-l}=(q-1)Z_{d-1-l}'(1)$,
then evaluating at $u=1$ the derivative of the equation $ Z_{\upvarepsilon}(u) =u^{(N-1)d}Z_{d-1-l}(u)$, and using the calculation 
\[ Z_{d-1-l}(1)=(q-1)\left( 1 + \cdots +  q^{d-1-l}   + \frac{q^{d-l}}{1-q} \right) =   -1,\]   we obtain
$ t_{\upvarepsilon} = t_{\mathfrak{a}_{d-1-l}}-d(N-1)(q-1)$.
\end{proof}



\begin{lemm}\label{ulemma}  Let $u_{d-1-l}$ be the $1$-unit part of $\upxi_{d-1-l}$.  Then for $\upvarepsilon=\upvarepsilon_{N,l}$,
\begin{align}\label{uepsilonform} u_{\upvarepsilon} = \left\{  \begin{array}{ll}
\frac{\sqrt{{\tt D}}}{f}u_{d-1-l} + O(q^{-2d(N+1)}) & \text{if $(d,p)=1$} \\ 
\\
\frac{\sqrt{{\tt D}}}{f}u_{d-1-l} \langle T\rangle^{N-1}+ O(q^{-2d(N+1)}) & \text{otherwise.}
\end{array} \right. \end{align}
\end{lemm}

\begin{proof}  First assume $(d,p)=1$ so that $\uppi=f^{-1/d}$.  Since the leading coefficient of the $T^{-1}$ expansion of ${\tt Q}_{N}$ is always $1$, 
${\rm sgn}({\tt Q}_{N})=1$.  Then by Binet's formula (\ref{Binetform})
\[ \langle {\tt Q}_{N}\rangle ={\tt Q}_{N}\uppi^{dN} =  f^{-N}\cdot \frac{f^{N+1}-(f^{\ast})^{N+1}}{\sqrt{{\tt D}}} = \frac{f}{\sqrt{{\tt D}}}  + O(q^{-d(2N+2)}) .  \]
Let us write 
\[u_{\upvarepsilon}(M) := \prod_{\substack{ 0\not=\uplambda\in\Uplambda_{\upvarepsilon}(f)\\ \deg(\uplambda)\leq M}} \langle\uplambda\rangle ,\quad   u_{d-1-l}(M):=
 \prod_{\substack{ 0\not=\upalpha\in\mathfrak{a}_{d-1-l} \\ \deg(\upalpha)\leq M}} \langle\upalpha\rangle.\] 
Then
\begin{align}\label{udNform}   u_{\upvarepsilon}(dN) = \prod_{c\in\F_{q}^{\times}} \langle c {\tt Q}_{N}\rangle =  \langle {\tt Q}_{N}\rangle^{q-1} = \left(  \frac{f}{\sqrt{{\tt D}}} \right)^{q-1}   + O(q^{-d(2N+2)}) ,\end{align}
where we use the fact that $|f/\sqrt{{\tt D}}|=1$.  The elements of degree $dN+1$ are of the form $c'({\tt Q}_{N}T+c{\tt Q}_{N})$ where $c'\in \F_{q}^{\times}$ and $c\in \F_{q}$.  Using that $\langle cx\rangle = \langle x\rangle$ for all $x\in k_{\infty}$ and $c\in\F_{q}^{\times}$, the identity
$ \langle  {\tt Q}_{N}T+c{\tt Q}_{N} \rangle $ $=$ $ (T+c)\uppi \langle{\tt Q}_{N} \rangle $,
and (\ref{udNform}),
 we have
\begin{align*}
u_{\upvarepsilon}(dN+1) 
 & =  \left( \prod_{c\in\F_{q}} \uppi (T+c)\right)^{q-1} \left( \frac{f}{\sqrt{{\tt D}}}\right)^{q^{2}-1} + O(q^{-d(2N+2)}) .
\end{align*}
On the other hand, $\langle f\rangle=1$ and 
$ \langle fT+cf\rangle=\uppi (T+c) $, so by induction
\begin{align*} u_{\upvarepsilon}(dN+M)  & = u_{d-1-l}(d+M) \left(  \frac{f}{\sqrt{{\tt D}}} \right)^{q^{M+1}-1} + O(q^{-d(2N+2)}).
\end{align*}
Since $|f/\sqrt{{\tt D}}| = 1$ and $q^{M+1}$ is a power of the characteristic of $\F_{q}$, it follows that \[ \lim_{M\rightarrow\infty}  \left(  \frac{f}{\sqrt{{\tt D}}} \right)^{q^{M+1}-1} =  \frac{\sqrt{{\tt D}}}{f}  \lim_{M\rightarrow\infty}  \left(  \frac{f}{\sqrt{{\tt D}}} \right)^{q^{M+1}} =  \frac{\sqrt{{\tt D}}}{f}.\]  This gives
the first formula of (\ref{uepsilonform}).
The case  $(d,p)\not=1$ is treated similarly.
\end{proof}

  
  

   \begin{proof}[Proof of Theorem \ref{qtexplimit}] 
 We set as before $\upvarepsilon =q^{-dN-l}$ for $l$ fixed.  If $(d,p)=1$,  then by taking a $(q-1)$th root of (\ref{formulaxi}), and applying Lemmas \ref{signlemma}, \ref{tlemma}, \ref{ulemma},
 we obtain
 \begin{align}\label{firstapproxxi} \upxi_{\upvarepsilon} & =
  \upxi_{d-1-l}\frac{\sqrt{{\tt D}}}{f^{N}} +O(q^{-(3N+1)d}).
 \end{align}
If $(d,p)\not=1$, then using $\uppi=(fT)^{-1/(d+1)}$ we also obtain (\ref{firstapproxxi}).
In either event, if we abbreviate 
$ C_{\upvarepsilon}:= \upxi_{\upvarepsilon}- \upxi_{d-1-l}\frac{\sqrt{{\tt D}}}{f^{N}}$, we have
\begin{align*}
\left| \upxi_{\upvarepsilon}^{-1}-\upxi_{d-1-l}^{-1}\frac{f^{N}}{\sqrt{{\tt D}}} \right| 
=\left| \frac{-f^{N} C_{\upvarepsilon}  }{\sqrt{{\tt D}} \upxi_{\upvarepsilon}  \upxi_{d-1-l}}   \right| 
\leq {\rm const.}\times  \frac{q^{dN}\cdot q^{-d(3N+1)} }{|\upxi_{\upvarepsilon}|}.
\end{align*}
By Lemma \ref{tlemma}, $|\upxi_{\upvarepsilon}|={\rm const.}\times |\uppi^{d(N-1)}| = {\rm const.} \times q^{-d(N-1)}$,
hence
\[  \upxi^{-1}_{\upvarepsilon}  =  \upxi^{-1}_{d-1-l}\frac{f^{N}}{\sqrt{{\tt D}}} + O(q^{-d(N+2) }) .\]
Then
$ \lim_{N\rightarrow\infty} e_{\upvarepsilon}(z)   =  \lim_{N\rightarrow\infty} \upxi_{\upvarepsilon}\exp_{\Uplambda_{\upvarepsilon}(f)}
 \left(  \upxi_{\upvarepsilon}^{-1} z \right) 
 = e_{d-1-l}(z).
$
  \end{proof}

   As a consequence of (\ref{firstapproxxi}) of the proof of Theorem \ref{qtexplimit} we have
   
   \begin{coro}\label{xiconvergence}  $\lim_{N\rightarrow\infty} (f^{N}/\sqrt{{\tt D}})\upxi_{\upvarepsilon_{N,l}} =\upxi_{d-1-l}$
   \end{coro}

   In the sequel, we will need to know the absolute value of $\upxi_{\upvarepsilon}$ for $\upvarepsilon =\upvarepsilon_{N,l}$.
   
   \begin{prop}\label{transabsvalue}  For $l=0,\dots ,d-1$, $  |\upxi_{d-1-l}|=q^{-lq^{d-l}-(d-1)+\frac{1}{q-1}} $.   \end{prop}
   
   \begin{proof}  Using  (\ref{formforZd-1-l}), a calculation involving telescoping cancellations yields 
   \begin{align*} Z'_{d-1-l}(1) 
   & = -d -q-\cdots -q^{d-1-l} + (2d-1-l)q^{d-l} +\frac{2dq^{d-l}(1-q)+q^{d-l+1}}{q-1} \\
 &=  \frac{-(q-1) (d-1) +1  +lq^{d-l}(1-q)  }{q-1}  = -(d-1)-lq^{d-l} +\frac{1}{q-1}.
   \end{align*}
   \end{proof}
   
    
 In what follows  a Drinfeld module over an order \cite{Hayes0}, \cite{DGVI} is denoted variously
 \[  \D =\uprho= (\C_{\infty}, \uprho )   =\C_{\infty}/\Uplambda ,\]
  where $\Uplambda$ is the kernel of the associated exponential.
The quantum analog of Drinfeld module is defined informally in our setting  
as a limit
 \begin{align}\label{informalqtD} \breve{\D}^{\rm qt}(f):=``\lim_{\upvarepsilon\rightarrow 0}  \breve{\D}_{\upvarepsilon}(f)\text{''} =\{\breve{\D}_{0}:=\C_{\infty}/\upxi_{0}\mathfrak{a}_{0},\dots ,\breve{\D}_{d-1}:=
 \C_{\infty}/\upxi_{d-1}\mathfrak{a}_{d-1} \}, \end{align}
 where $\breve{\D}_{0},\dots , \breve{\D}_{d-1}$ are Hayes modules over the order $A_{f}$ and the $ \breve{\D}_{\upvarepsilon}(f)$ may be considered as examples of the notion of {\it approximate Drinfeld module}. 
 Leaving aside the technical details of its definition, $ \breve{\D}_{\upvarepsilon}(f)$ 
is analytically uniformized by the vector space $\breve{\Uplambda}_{\upvarepsilon}(f)$ -- 
 a multiple of $\Uplambda_{\upvarepsilon}(f)$.  It is natural then to define $j(\breve{\D}_{\upvarepsilon}(f)):=j_{\upvarepsilon}(f)$ (defined in \S 2 of \cite{DGIII}), and after taking limits, $j(\breve{\D}^{\rm qt}(f)) :=
j^{\rm qt}(f)$.
 
 It remains to make the limit in (\ref{informalqtD}) precise i.e.\ we must give a meaningful notion of
 limiting point, which we will refer to as a {\it quantum point}.  To do this, we will make use of connecting maps 
 $ \breve{\D}_{0}\longrightarrow \breve{\D}_{d-1-l}$, $ l=0,\dots ,d-1$, defined as follows.
Let  $\uptau$ denote the Frobenius map $x\mapsto x^{q}$,  $x\in \C_{\infty}$, so that
$\uprho_{i}: A_{f}\longrightarrow   \C_{\infty}\{ \uptau\} $,
where $\breve{\D}_{i} =(\C_{\infty}, \uprho_{i})$.
The $\ast$-action 
satisfies
$\mathfrak{a}_{i}\ast\uprho_{0}=\uprho_{d-i}$, 
and gives rise to a morphism
$ \Upphi_{i}:\uprho_{0}\longrightarrow \uprho_{d-i}$
where $\Upphi_{i}$ is 
the unique   
{\it monic} generator of the principal left ideal
in $H_{A_{f}}\{ \uptau\}$ generated by the polynomials $\uprho_{0,a}$, $a\in \mathfrak{a}_{i}$. See \cite{Hayes0}, \S\S 6, 8. 


Recall the notation $e_{0}(z)=\exp_{\Uplambda_{0}}(z)$.

\begin{lemm}\label{explicitrhocalc}  As a polynomial in the variable $x$, for $0<i\leq  d-1$, 
\[\Upphi_{i}(x)=x\prod_{0\not =\alpha\in \mathfrak{a}_{d-i}/\mathfrak{a}_{0}} \left(x- e_{0}(\upxi_{0}\alpha) \right).\]
\end{lemm}
\begin{proof}  In what follows, we identify $\mathfrak{a}_{d-i}/\mathfrak{a}_{0}$ 
with the vector space
$   \mathbb{F}_{q}\langle fT, ..., fT^{d-i}\rangle $.
To ease notation we will also write 
$ W_{a} := \{\upalpha\in A_{f}:\; \uprho_{0,a}(e_{0}(\upxi_{0}\upalpha ))=0\} $.  Then the g.c.d.\ of the set 
$ \{\uprho_{0,f}, \dots , \uprho_{0, fT^{i}}\}$ will be the separable 
$\F_{q}$-additive polynomial whose roots belong
to the $\mathbb{F}_{q}$-vector space 
$  \left\{ e_{0}(\upxi_{0}\upalpha ) : \; \upalpha \in W_{f}\cap W_{fT}\cap ... \cap W_{fT^{i}}\right\}$. 
In other words 
\[\Upphi_{i}(x)=x\prod_{0\not=\alpha\in W_{f}\cap W_{fT}\cap ... \cap W_{fT^{i}}} \left(x- e_{0}(\upxi_{0}\alpha) \right).\]
It thus remains for us to identify $ W_{f}\cap W_{fT}\cap ... \cap W_{fT^{i}}$.
By
Theorem 4.8 of \cite{Hayes0},
$ W_{fT^{i}} = [fT^{i}]^{-1} (f)  \mod (f) $.
When $i=0$, we have
$W_{f}=\mathbb{F}_{q}\langle 1, fT, ..., fT^{d-1}\rangle $. 
For $0<i<d$,  we claim that we may identify $W_{fT^{i}} = [fT^{i}]^{-1}(f)/(f)$ with
\begin{align*}    \mathbb{F}_{q}\langle T^{-i}, fT^{-i}, fT^{-i+1}, \dots , fT^{d-1-i}, f^{2}T^{-i}\dots \rangle/ (f) \\ 
 = \mathbb{F}_{q}\langle T^{-i}, fT^{-i},  \dots , fT^{-1}, fT,\dots , fT^{d-1-i}, f^{2}T^{-i}, \dots , f^{2}T^{-1}\rangle/ (f).
\end{align*}
Since $\deg (fT^{i})=d+i$, $W_{fT^{i}}$ is a $d+i$ dimensional $\F_{q}$-vector space, and since the set  $T^{-i}, fT^{-i},  \dots , fT^{-1}, fT,\dots , fT^{d-1-i}, f^{2}T^{-i}, \dots , f^{2}T^{-1}$ has $d+i$ elements, it suffices to show independence.
However this is clear, since the elements of this set all have distinct degrees, and $(f)$ contains no elements having these degrees.
Using the quadratic relation $f^{2}={\tt a}f+{\tt b}$, we may write $ f^{2}T^{-i}   = {\tt a}fT^{-i}+ {\tt b}T^{-i}$,
where ${\tt a}=T^{d}+c_{d-1}T^{d-1}+\cdots  + c_{1}T+c_{0}\in\F_{q}[T]$ and ${\tt b}\in\F_{q}^{\times}$.  Thus we may replace $f^{2}T^{-i}$ by $fT^{d-i}$, so that 
$W_{fT}  =\mathbb{F}_{q}\langle T^{-1}, fT^{-1},  fT,\dots, fT^{d-2}, fT^{d-1}\rangle $ and for $i>1$,
\begin{align*} 
W_{fT^{i}} & =\mathbb{F}_{q}\langle T^{-i}, fT^{-i} ,\dots, fT^{-1}, fT,\dots ,fT^{d-i}, f^{2}T^{-i+1}, \dots, f^{2}T^{-1}\rangle.
\end{align*}
We now show by induction that for $i>0$, \begin{align}\label{bigintersection} \bigcap_{j=0}^{i}W_{fT^{j}}=\mathbb{F}_{q}\langle fT, ..., fT^{d-i}\rangle.\end{align}
When $i=1$, $\F_{q}\langle fT,\dots ,fT^{d-1}\rangle \subset W_{f}\cap W_{fT}$.  Since $W_{fT}$ has no constant polynomials, $\dim (W_{f}\cap W_{fT})\leq d-1$,
so it follows that $\F_{q}\langle fT,\dots ,fT^{d-1}\rangle =W_{f}\cap W_{fT}$.  Now assume that (\ref{bigintersection}) is true up to $i$, and consider 
\[  \bigcap_{j=0}^{i+1}W_{fT^{j}}=\mathbb{F}_{q}\langle fT, \dots , fT^{d-i}\rangle \cap W_{fT^{i+1}} \supset \mathbb{F}_{q}\langle fT, \dots , fT^{d-i-1}\rangle .\]
It will suffice to show that $fT^{d-i}\not\in W_{fT^{i+1}}$.  Now using the quadratic relation $f^{2}={\tt a}f+{\tt b}$ we may write the element $ f^{2}T^{-(i+1)+1} = f^{2}T^{-i}\in W_{fT^{i+1}}$ as
\begin{align*} & fT^{d-i} + c_{d-1}fT^{d-i-1}+\dots + c_{1}fT+c_{0}f +{\tt b}T^{-i}  \\  \equiv \quad&  fT^{d-i} + c_{d-1}fT^{d-i-1}+\dots +c_{1}fT+{\tt b}T^{-i} \mod (f). \end{align*}
Then $fT^{d-i}\in W_{fT^{i+1}}$ would imply that $T^{-i}\in W_{fT^{i+1}}$ which is false.
\end{proof}



Given $w\in \C_{\infty}$ we define the associated {\bf {\em quantum point}} by \begin{align}\label{DefQPoint}    w^{\rm qt}  := \{w,\Upphi_{d-1}(w),\dots ,\Upphi_{1}(w)\} \end{align}
and denote by $ \C_{\infty}^{\rm qt} := \{ w^{\rm qt}|\; w\in \C_{\infty}\}$
the set of quantum points.  Since the $\Upphi_{i}$ are additive, $ \C_{\infty}^{\rm qt} $ is an $\F_{q}$ vector space (the sum of multi-points is carried out vectorially, according to the labelling in
(\ref{DefQPoint})).
For any $\upalpha\in A_{f}$, the formula
\[ \uprho^{\rm qt}_{\upalpha} (w^{\rm qt}) := \{ \uprho_{0,\upalpha}(w_{0}),\dots ,\uprho_{d-1,\upalpha}(w_{d-1})\}\]
defines an action of $A_{f}$ on $\C_{\infty}^{\rm qt}$ making the latter
an $A_{f}$ module.
Indeed, for any $\upalpha\in A_{f}$ and $ w^{\rm qt}\in \C_{\infty}^{\rm qt}$,  
$\uprho^{\rm qt}_{\upalpha} (w^{\rm qt})\in \C_{\infty}^{\rm qt}$: this follows immediately from the definition (\ref{DefQPoint}) of a quantum point and the fact
that the $\Upphi_{i}$ are morphisms of Drinfeld modules.
We refer to the $A_{f}$ module $\C_{\infty}^{\rm qt}$ as the {\bf {\em sign normalized quantum Drinfeld module}}
or {\bf {\em quantum Hayes module}} associated to $f$, denoting
it 
\[  \uprho^{\rm qt}= \uprho^{\rm qt}_{f} :=\{\uprho_{0},\dots ,\uprho_{d-1}\}  .\]

The points of $\uprho^{\rm qt}$ may be obtained as the multivalues of the quantum exponential function with a slightly different normalization.
 The need for a new normalization arises since $\Upphi_{l+1}\circ e_{0} \not=e_{d-1-l}$, contrary to what one might expect: see for example Proposition 4.9.4 and Corollary 4.9.5 in \S 4.9 of \cite{Goss}.
 To see what this normalization should be, we note that if we write 
 \begin{align}\label{derivativeofPhi} D_{l+1} & = \prod_{0\not =\alpha\in \mathbb{F}_{q}\langle fT, ..., fT^{d-1-l}\rangle} e_{0}(\upxi_{0}\upalpha )\nonumber   = \prod_{0\not =\alpha\in \mathfrak{a}_{d-1-l}/\mathfrak{a}_{0}} e_{0}(\upxi_{0}\upalpha )   = -
 \Upphi_{l+1}'(x).\nonumber \end{align}
 then
 \[  \upxi_{0}^{-1}\upxi_{d-1-l} (-D_{l+1})^{-1} \circ \Upphi_{l+1}\circ e_{0}\circ \upxi_{0}\upxi_{d-1-l}^{-1} = e_{d-1-l}  \]
 since both sides have divisor $\Uplambda_{d-1-l}=\upxi_{d-1-l}\mathfrak{a}_{d-1-l}$ with derivatives $\equiv 1$. Write
 \begin{align}\label{expqtnormalized} \widetilde{\exp}^{\rm qt}(z) := \{ e_{0}(z), \widetilde{e}_{1}(z) ,\dots ,  \widetilde{e}_{d-1}(z)\}, \end{align}
 where 
 \begin{align*} \widetilde{e}_{d-1-l} = \upxi_{0}\upxi_{d-1-l}^{-1} D_{l+1}\circ e_{d-1-l}\circ\upxi_{0}^{-1}\upxi_{d-1-l}.\end{align*}
Then
 for any $z\in \C_{\infty}$, $ \widetilde{\exp}^{\rm qt}(z)$ is a quantum point in the sense defined above, since for all $l$, $\Upphi_{l+1}\circ e_{0}(z)=\widetilde{e}_{d-1-l}(z)$.

   \section{Relative and Absolute Ray Class Fields}\label{RayClassFieldSection}
   


   
In this section we give the definitions of the
 ray class fields which will be the focus of this paper.
These are essentially $S$-definitions in the style of Rosen \cite{Ros}, \cite{Ros2},
   first appearing in \cite{Per}, \cite{Auer} (in their wide versions), but extended to non-Dedekind orders here. We recall that $S$-definitions
   are ring-specific and hence differ from the more general field-specific definitions that one encounters in texts such as \cite{Neu}, \cite{RamVal} (which in the function field setting
  usually give infinite ray class groups).  
  
    The necessity of defining ray class fields in the general setting of orders arises from the occurrence in the ray class group, associated to a modulus $\mathfrak{M}\subset\mathcal{O}_{K}$,  of nontrivial elements coming from the unit group $\mathcal{O}_{K}^{\times}$.
The structure of our class field generation results requires that there be no such elements, and in order to kill them, we must pass to orders of $\mathcal{O}_{K}$ and $A_{\infty_{1}}$ in which the problematic units no longer appear.   The ray class field thus
generated will be a finite extension of the usual ray class field associated to $\mathfrak{M}$.  This so called {\it unit ray class field} may be defined, {\it mutatis mutandis}, in the setting
of number fields.  See  \S \ref{CharZeroConj}.
    
 
Fix a nonconstant unit $f\in \mathcal{O}_{K}^{\times}$ with associated order $A_{f}$, defined in (\ref{forder}).  Define
 \[ \mathcal{O}_{f} = A[f,f^{-1}]= \F_{q} [T, f, f^{-1}] \subset \mathcal{O}_{K}.\]

 Denote by $\mathfrak{c}=\mathfrak{c}_{f}, \mathfrak{C}=\mathfrak{C}_{f}$ the conductors of $A_{f}\subset A_{\infty_{1}}$, $\mathcal{O}_{f}\subset \mathcal{O}_{K}$: thus, $\mathfrak{c}$
 is the largest $A_{\infty_{1}}$ ideal contained in $A_{f}$ and $\mathfrak{C}$
 is the largest $\mathcal{O}_{K}$ ideal contained in $\mathcal{O}_{f}$.

\begin{prop} $  \mathcal{O}_{f} \cap A_{\infty_{1}} = A_{f}$ and $ \mathfrak{C}\cap  A_{\infty_{1}} =\mathfrak{c} $.
\end{prop}

\begin{proof} Since $A_{f}\subset \mathcal{O}_{f}$, we have $ \mathcal{O}_{f} \cap A_{\infty_{1}} \supset A_{f}$.  Using the quadratic relations for $f$, $f^{-1}$, we may write
every element of $ \mathcal{O}_{f}$ in the form 
$ a_{-1}(T)f^{-1} +a_{0}(T) + a_{1}(T) f$, $a_{i}(T)\in A, \;\; i=-1,0,1$.
If such an element is also in $A_{\infty_{1}}$, it is regular at $\infty_{2}$, hence we must have $a_{-1}(T)=a_{0}(T)=0$, and then, using the quadratic relation
of $f$, $a_{1}(T)f\in A_{f}$.  Therefore, $ \mathcal{O}_{f} \cap A_{\infty_{1}} \subset A_{f}$ i.e.\  $ \mathcal{O}_{f} \cap A_{\infty_{1}} =A_{f}$. The $A_{\infty_{1}}$ ideal $ \mathfrak{C}\cap  A_{\infty_{1}}$ is contained
(by the equality we have just established) 
in $A_{f}$, hence $ \mathfrak{C}\cap  A_{\infty_{1}} \subset \mathfrak{c} $.  On the other hand, if $x\in \mathfrak{c}\subset  A_{f}\subset \mathcal{O}_{f}$,  and since $T\in \mathcal{O}_{f}$, we have $Tx\in \mathcal{O}_{f}$.   Trivially, $fx\in \mathcal{O}_{f}$. Therefore, $\mathcal{O}_{K}x\subset \mathcal{O}_{f}$, $x\in\mathfrak{C}$ and $\mathfrak{C}\cap  A_{\infty_{1}} \supset \mathfrak{c} $.
\end{proof}

The Hilbert class field
$ H_{A_{f}}$ of the order $A_{f}$ is defined as the field of invariants of the Drinfeld module $\uprho^{\mathfrak{a}}$,  where $\mathfrak{a}\subset A_{f}$ 
is any invertible ideal.
Alternatively, it is the smallest field of definition of $\uprho^{\mathfrak{a}}$.  
  See Theorem 6.6 and \S 8 of \cite{Hayes0}.  
Let $ {\sf I}^{\ast}_{A_{f}}$ be the group of invertible $A_{f}$ ideals, ${\sf P}_{A_{f}} $ the subgroup of principal ideals.  The ideal class group of $A_{f}$ is then
$ {\sf Cl}_{A_{f}} = {\sf I}^{\ast}_{A_{f}} / {\sf P}_{A_{f}} $;
it may be identified  with the Galois group of the extension
$ H_{A_{f}}/K  $.
See Theorem 8.10 of \cite{Hayes0}.

We may represent the elements of ${\sf Cl}_{A_{f}}$ using ideals in the $A_{\infty_{1}}$, as follows. Let
\[  {\sf I}^{\mathfrak{c}}_{A_{\infty_{1}}} = \big\{ \mathfrak{a} \text{ an } A_{\infty_{1}}\text{-ideal}  \; :\;\; (\mathfrak{a},\mathfrak{c})=1\big\}\supset {\sf P}^{\mathfrak{c}}_{A_{\infty_{1}}, f} = \big\langle a A_{\infty_{1}}\; : \;\; a\in A_{f},\; (a,\mathfrak{c})=1\big\rangle ,  \]
where $\langle X\rangle=$ the group generated by the set $X$.

\begin{prop}\label{AltOrderClassGpForm}
$ {\sf Cl}_{A_{f}} \cong {\sf I}^{\mathfrak{c}}_{A_{\infty_{1}}}  / {\sf P}^{\mathfrak{c}}_{A_{\infty_{1}}, f}. $
\end{prop}

\begin{proof} See Theorem 8.10  of \cite{Hayes0}.\end{proof}

\begin{note} For {\it any} non-zero ideal $\mathfrak{n}\subset A_{\infty_{1}}$, every class in ${\sf Cl}_{A_{\infty_{1}}}$ contains a representative relatively prime to $\mathfrak{n}$: see Lemma 1.3 of \cite{DGVI}. Thus
\[ {\sf Cl}_{A_{\infty_{1}}} \cong  {\sf I}^{\mathfrak{c}}_{A_{\infty_{1}}} / {\sf P}^{\mathfrak{c}}_{A_{\infty_{1}}},\quad   
{\sf P}^{\mathfrak{c}}_{A_{\infty_{1}}} =  \big\langle a A_{\infty_{1}}\; : \;\;  a\in A_{\infty_{1}},\; (a,\mathfrak{c})=1\big\rangle . \]
In particular,  $ {\sf P}^{\mathfrak{c}}_{A_{\infty_{1}}, f} $ is, in general, a proper subgroup of ${\sf P}^{\mathfrak{c}}_{A_{\infty_{1}}}$ and so, as expected 
(c.f.\ Corollary 1.4 of \cite{DGVI}), ${\sf Cl}_{A_{\infty_{1}}}$ is a quotient of ${\sf Cl}_{A_{f}} $.
\end{note}

We now give a description of $H_{A_{f}}$ using class field theory.   Let 
$C_{K}= \I_{K}/K^{\times}$ be the id\`{e}le class group.  Define
\[  \I_{A_{f}} := K^{\times}_{\infty_{1}} \cdot  \prod_{\mathfrak{p}\subset A_{\infty_{1}}} (A_{f})_{\mathfrak{p}}^{\times},  \]
where $(A_{f})_{\mathfrak{p}} $ is the completion of $A_{f}$ in $K_{\mathfrak{p}}$.  Write
$  C_{A_{f}}$ $ :=$ $  \left( \I_{A_{f}}\cdot K^{\times}\right) /K^{\times}   <C_{K}$.

\begin{theo}\label{CFTHCF} The subgroup $C_{A_{f}}< C_{K}$ is of finite index, and the map of id\`{e}les to ideal classes induces a canonical isomorphism
\[ C_{K}/ C_{A_{f}} \cong {\sf Cl}_{A_{f}}. \]
\end{theo}

\begin{proof} We will show that $C_{K}/ C_{A_{f}} \cong {\sf I}^{\mathfrak{c}} _{A_{\infty_{1}}}/{\sf P}^{\mathfrak{c}} _{A_{\infty_{1},f}}$.  Define first 
\[\I^{A_{f}} = K^{\times}_{\infty_{1}} \cdot \prod_{\mathfrak{p}| \mathfrak{c}} (A_{f})^{\times}_{\mathfrak{p}} \cdot \prod'_{\infty_{1}\not=\mathfrak{p}\;\nmid \;\mathfrak{c}}  K^{\times}_{\mathfrak{p}} \]
where the final product is restricted over the $(A_{f})^{\times}_{\mathfrak{p}} $ for $\mathfrak{p}\nmid \mathfrak{c}$ (note, in any event, that $(A_{f})^{\times}_{\mathfrak{p}}=(A_{\infty_{1}})^{\times}_{\mathfrak{p}}$ under the hypothesis
$\mathfrak{p}\nmid \mathfrak{c}$).  We first show that 
$\I_{K} =\I^{A_{f}}   \cdot K^{\times}$.
Indeed, for $\upalpha = (\upalpha_{\mathfrak{p}}) \in \I_{K}$, by the Approximation Theorem, we may write $\upalpha=a \upalpha'$ where $a\in K^{\times}$ 
and for all $\mathfrak{p}|\mathfrak{c}$, $\upalpha_{\mathfrak{p}}'\in (A_{\infty_{1}})^{\times}_{\mathfrak{p}}$.  
Now if $c\in \mathfrak{c}$, $\upalpha '' =c\upalpha'\in\I^{A_{f}}$, hence $\upalpha = \upalpha''ac^{-1}\in  \I^{A_{f}}   \cdot K^{\times}$.
With this identification, we may write
\begin{align}\label{altCK1} C_{K} = \frac{\I^{A_{f}}   \cdot K^{\times}}{ K^{\times}} \cong \frac{\I^{A_{f}} }{ \I^{A_{f}}   \cap K^{\times}} .\end{align}
Now the map
$ \I^{A_{f}}  \longrightarrow {\sf I}^{\mathfrak{c}}_{A_{\infty_{1}}}$,  $\upalpha\longmapsto (\upalpha ):= \prod_{\mathfrak{p}\not=\infty_{1}} \mathfrak{p}^{v_{\mathfrak{p}} (\upalpha_{\mathfrak{p}})} $, 
precomposed with the isomorphism (\ref{altCK1}),
induces an epimorphism
$  C_{K}\longrightarrow {\sf I}^{\mathfrak{c}} _{A_{\infty_{1}}}/{\sf P}^{\mathfrak{c}} _{A_{\infty_{1},f}} $, 
since $\I^{A_{f}}  \cap K^{\times}$ is sent to ${\sf P}^{\mathfrak{c}}_{A_{\infty_{1},f}} $. Note
$C_{A_{f}}$ is contained in the kernel: we show it is the kernel.  If $ \upalpha \in\I^{A_{f}} \mod\I^{A_{f}}   \cap K^{\times} $ is in the kernel, then 
 $(\upalpha ) = (a)\in {\sf P}^{\mathfrak{c}}_{A_{\infty_{1},f}} $.  Note then that $a^{-1}\upalpha\in \I_{A_{f}}$: indeed, since $(a,\mathfrak{c})=1$ and either $a$ or $a^{-1}\in A_{f}$, 
 $a\in (A_{f})_{\mathfrak{p}}^{\times}$ for all $\mathfrak{p}|\mathfrak{c}$ -- so multiplying $\upalpha$ by $a^{-1}$ conserves that property -- and the remaining coordinates of $a^{-1}\upalpha$
 must be units.  However, $ \I_{A_{f}}$
 is contained in the image of $C_{A_{f}}$ via the isomorphism  (\ref{altCK1}).  Therefore,
 \[ \upalpha \equiv a^{-1}\upalpha  \mod\I^{A_{f}}   \cap K^{\times} \;\;  \in\;\; \I _{A_{f}}  \mod\I^{A_{f}} \cap K^{\times}\subset \text{image of } C_{A_{f}} \text{ via  (\ref{altCK1})}. \]
 By Proposition \ref{AltOrderClassGpForm} we are done.
 \end{proof}

\begin{coro} $H_{A_{f}}$ corresponds to $C_{A_{f}}$ via class field theory.
\end{coro}

\begin{proof} By Theorem 8.10 of \cite{Hayes0}, $H_{A_{f}}$ corresponds to the $\mathfrak{c}$-ideal class group  $  {\sf I}^{\mathfrak{c}}_{A_{\infty_{1}}}/
{\sf P}^{\mathfrak{c}}_{A_{\infty_{1}}}$ $\cong{\sf Cl}_{A_{f}}$ via the ideal-theoretic version of Class Field Theory.  By Theorem \ref{CFTHCF}, $H_{A_{f}}$ 
corresponds to $C_{A_{f}}$ via the id\`{e}le-theoretic formulation of Class Field Theory. 
\end{proof}

As a consequence we have the tower of Hilbert class fields
$H_{A_{f}}
 \supset
H_{A_{\infty_{1}}} \supset H_{\mathcal{O}_{K}} $.

Now we turn to $\mathcal{O}_{f}$: here, we define the Hilbert class field $H_{\mathcal{O}_{f}}$ via class field theory as the abelian extension associated to
$C_{\mathcal{O}_{f}} = ( \I_{\mathcal{O}_{f}} \cdot K^{\times})/K^{\times} $,
where
\[   \I_{\mathcal{O}_{f}} = K^{\times}_{\infty_{1}}\cdot K^{\times}_{\infty_{2}}\cdot \prod_{\mathfrak{p}\not=\infty_{1},\infty_{2}}  (\mathcal{O}_{f})^{\times}_{\mathfrak{p}} .\]  The field $H_{\mathcal{O}_{f}}$
fits into the following diamond of extensions
\[ 
\begin{diagram}
H_{A_{f}} &  & \\
\vLine  & \rdLine & \\
H_{A_{\infty_{1}}}  && H_{\mathcal{O}_{f}} \\
  & \rdLine &\vLine \\
& &H_{\mathcal{O}_{K}} . 
\end{diagram}
\]
In analogy with the previous development, 
$  {\rm Gal}(H_{\mathcal{O}_{f}}/K) \cong {\sf Cl}(\mathcal{O}_{f} ) \cong {\sf I}^{\mathfrak{C}}_{\mathcal{O}_{K}}/ {\sf P}_{\mathcal{O}_{K},f} ^{\mathfrak{C}} $,
with
${\sf I}^{\mathfrak{C}}_{\mathcal{O}_{K}} = \{ \mathfrak{A}\text{ an } \mathcal{O}_{K}\text{-ideal}   :\; (\mathfrak{A},\mathfrak{C})=1\}\supset
 {\sf P}^{\mathfrak{C}} _{\mathcal{O}_{K},f}= \langle a \mathcal{O}_{K} : \; a\in\mathcal{O}_{f},\; (a,\mathfrak{C})=1\rangle  $. 

In \cite{DGIII}, \cite{DGV}, it is shown that for any $\mathfrak{a}\in {\sf Cl}_{A_{\infty_{1}}}$, $ H_{A_{\infty_{1}}} = K(j(\mathfrak{a}))$,where $j(\mathfrak{a})$ is the modular invariant of the ideal class $\mathfrak{a}$
and moreover,
\[  H_{\mathcal{O}_{K}} = K\left( \prod_{\upalpha\in j^{\rm qt}(f_{0})} \upalpha \right),\]
where $f_{0}$ is a fundamental unit.  We have the identification \cite{DGIII}:
\begin{align}\label{IDjqtWithJideals} j^{\rm qt}(f_{0}) = \{ j(\mathfrak{a}_{i}  )\}_{i=0,\dots ,d_{0}-1} ,\quad \mathfrak{a}_{i} = (f_{0}, f_{0}T,\dots , f_{0}T^{i}).\end{align}
The analogs of these statements for $H_{A_{f}}$ respectively $H_{\mathcal{O}_{f}}$ are proved in \S\S 5, 6  of \cite{DGVI}.
In particular, the $A_{f}$ analogs of the ideals in (\ref{IDjqtWithJideals}), which were defined in (\ref{aidef}) and are also denoted  $\mathfrak{a}_{i}  \subset A_{f}$, play a fundamental role in what follows.

Fix once and for all an ideal $\mathfrak{M}\subset \mathcal{O}_{K}$ and denote as follows the contractions
$ \mathfrak{M}_{f}:=\mathfrak{M}\cap \mathcal{O}_{f}$, $\mathfrak{m} := A_{\infty_{1}}\cap \mathfrak{M}$,   $\mathfrak{m}_{f} := A_{f}\cap \mathfrak{M}$,
all referred to, as usual, as moduli.   Note then when $f=f_{0}$ is a fundamental unit, then $\mathfrak{M}_{f_{0}}=\mathfrak{M} $ and $\mathfrak{m}_{f_{0}}=\mathfrak{m} $.

         \begin{lemm}\label{coprimalityofideals}   Let $\mathfrak{a}_{i}\subset A_{f}$,  $i=0,\dots, d-1$, be as in $(\ref{aidef}) $.  Then  $\mathfrak{a}_{i}$ is relatively prime to $\mathfrak{m}_{f}$ and $\mathfrak{c}$.
   \end{lemm}
   
   \begin{proof} We begin with the coprimality with respect to $\mathfrak{m}_{f}$.  It will be enough to show that $\mathfrak{a}_{0}=fA_{f}$ is relatively prime to $\mathfrak{m}_{f}$, since $\mathfrak{a}_{i}$ is the $(d-i)$th power of the prime ideal $\mathfrak{a}_{d-1}$.
  Suppose on the contrary
   that $(\mathfrak{a}_{0},\mathfrak{m}_{f})\not=1$.  Since $A_{f}/\mathfrak{m}_{f}$ is finite, its elements are either zero divisors or units, and if $f$ were to define a unit it
   would be relatively prime to $\mathfrak{m}_{f}$.  So $f$ must define a nontrivial zero divisor ($f\not\in\mathfrak{m}_{f}$ since $f\not\in\mathfrak{M}_{f}$ as the latter is a nontrivial ideal).  As such, there exists $a\in A_{f}$ so that $af\in \mathfrak{m}_{f}\subset\mathfrak{M}_{f}$. Since $f$ is an $\mathcal{O}_{f}$ unit, $a \in \mathfrak{M}_{f}$, and so  $a\in\mathfrak{m}_{f}$.  This contradicts $f$ being a nontrivial zero divisor,
   so this takes care of coprimality with respect to $\mathfrak{m}_{f}$.
   Now the ideal $\mathfrak{a}_{d-1}$ is maximal, hence prime, in $A_{f}$, and by Proposition \ref{PowerProp}, $ \mathfrak{a}_{d-1}^{d} =\mathfrak{a}_{0}= fA_{f}$,   so $\mathfrak{a}_{0}$ has a prime decomposition consisting of a power of a single prime  $\mathfrak{a}_{d-1}$.  In particular, the latter shows that $\mathfrak{a}_{d-1}$ is invertible, with inverse $\mathfrak{a}_{d-1}^{-1} = f^{-1} \mathfrak{a}_{d-1}^{d-1}$.  By Theorem 6.1 of \cite{Conrad1}, this implies that $\mathfrak{a}_{d-1}$ is relatively prime to $\mathfrak{c}$, and thus, so are all of its powers.
       \end{proof}

       As remarked in the proof of Lemma \ref{coprimalityofideals}, the ideal $\mathfrak{a}_{d-1}$ is maximal hence prime, and all elements of the group $\langle \mathfrak{a}_{d-1}\rangle$ are thus  powers of a single prime.  By the bijective correspondence between ideals relatively prime to $\mathfrak{c}$ in $A_{f}$ and $A_{\infty_{1}}$ (c.f.\ Theorem 1.1 of \cite{DGVI}), it follows
       that
      $  \mathfrak{m}_{\infty_{2}} =    \mathfrak{a}_{d-1} A_{\infty_{1}}    $
       where $\mathfrak{m}_{\infty_{2}}$ is the $A_{\infty_{1}}$ ideal of elements vanishing at $\infty_{2}$, which is explicitly generated by $f_{0}, f_{0}T,\dots , f_{0} T^{d_{0}-1} $ for $f_{0}$ 
       a fundamental unit of $\mathcal{O}_{K}$.

       \begin{prop}\label{ZProp}  Let $  \mathfrak{a}_{d-1}\in  {\sf Cl}_{A_{f}} $ be as above.  Then 
       $ \langle  \mathfrak{a}_{d-1}\rangle  = {\rm Ker}( {\sf Cl}_{A_{f}} \rightarrow {\sf Cl}_{\mathcal{O}_{f}}  )$. In particular, the Galois group
      $  Z={\rm Gal}(H_{A_{f}}/H_{\mathcal{O}_{f}})$ is canonically isomorphic to $ \langle \mathfrak{a}_{d-1}\rangle $ via Class Field Theory.
       \end{prop}
       
       \begin{proof} For $f=f_{0}$ a fundamental unit, this result was proved in Proposition 2 of  \cite{DGIII}; we will use the latter to prove the Proposition. Denote
       $ Z_{0} :=$ ${\rm Gal}(H_{A_{\infty_{1}}}/H_{\mathcal{O}_{K}})$  $\cong$ ${\rm Ker}( {\sf Cl}_{A_{\infty_{1}}} \rightarrow {\sf Cl}_{\mathcal{O}_{K}}  )$ $= \langle \mathfrak{a}_{d_{0}-1}\rangle$, where
         $  \mathfrak{a}_{d_{0}-1} = (f_{0}, f_{0}T,\dots , f_{0}T^{d_{0}-1} ) $.
         In what follows, we simply write $Z={\rm Ker}( {\sf Cl}_{A_{f}} \rightarrow {\sf Cl}_{\mathcal{O}_{f}}  )$ and $Z_{0}={\rm Ker}( {\sf Cl}_{A_{\infty_{1}}} \rightarrow {\sf Cl}_{\mathcal{O}_{K}}  ) =
         \langle
         \mathfrak{a}_{d_{0}-1}\rangle$.
       We have the following commutative diagram of short exact sequences:
       \[
       \begin{diagram}
   1 &\rTo &  Z & \rInto &  {\sf Cl}_{A_{f}} \cong {\sf I}_{A_{f}}^{\mathfrak{c}} / {\sf P}_{A_{f}}^{\mathfrak{c}} & \rOnto^{\Upphi} &    
       {\sf Cl}_{\mathcal{O}_{f}} \cong {\sf I}_{\mathcal{O}_{f}}^{\mathfrak{C}} / {\sf P}_{\mathcal{O}_{f}}^{\mathfrak{C}} &\rTo &1 \\ 
 & &   \dTo & &   \dTo_{\Uppsi_{A}} & & \dTo_{\Uppsi_{\mathcal{O}}} & & \\
  1 & \rTo&  Z_{0}=  \langle
         \mathfrak{a}_{d_{0}-1}\rangle &\rInto &  {\sf Cl}_{A_{\infty_{1}}} \cong  {\sf I}_{A_{\infty_{1}}} / {\sf P}_{A_{\infty_{1}}}   & \rOnto_{\Upphi_{0}} & {\sf Cl}_{\mathcal{O}_{K}} \cong  {\sf I}_{\mathcal{O}_{K}}/ {\sf P}_{\mathcal{O}_{K}}  & \rTo& 1,
        \end{diagram}\]
        in which the indicated identifications of class groups are those that were established earlier in this section.  The vertical maps $\Uppsi_{A}, \Uppsi_{\mathcal{O}}$ are induced by ideal expansion, e.g., $ \mathfrak{b}_{f} \longmapsto \mathfrak{b}:=\mathfrak{b}_{f}A_{\infty_{1}}$, 
        which
        are bijections on ideals prime to $\mathfrak{c}$ (see Theorem 1.1 of \cite{DGVI}).  
        We have the inclusion 
       $ \langle \mathfrak{a}_{d-1}\rangle\subset Z$ since $ \mathfrak{a}_{d-1}$ contains the $\mathcal{O}_{f}$-unit $f$.  Suppose there is an ideal class $\mathfrak{b}_{f}\in Z\setminus \langle \mathfrak{a}_{d-1}\rangle$;
       without loss of generality, we may assume that $\mathfrak{b}_{f}$ is prime to $\mathfrak{a}_{d-1}$.    Then by Theorems  1.3 and  1.4 of \cite{DGVI},
        $\mathfrak{b}_{f}$ may be written as a product of primes  that are prime to the conductor.  We claim that 
        $\mathfrak{a}_{d-1}A_{\infty_{1}} \subset \mathfrak{a}_{d_{0}-1}$.   Indeed, the generators of $\mathfrak{a}_{d-1}$ are of the form $fT^{j} $, $j=d_{0}m+i\in \{ 0,\dots ,d-1= nd_{0}-1\}$; writing the exponent of $T$ as 
        $j = md_{0} + i$ for $m<n$ and $i<d_{0}$, we
    have
        $fT^{j} = f_{0}^{n}T^{j}  = ( f_{0} T^{d_{0}}) ^{m}\cdot f_{0}^{n-m} T^{i} $.
        Using the quadratic relation for $f_{0}$, $f_{0}T^{d_{0}} \in \mathfrak{a}_{d_{0}-1}$, and this proves the claim.  But by
         Theorem 1.1 of \cite{DGVI}, the expansion $\mathfrak{a}_{d-1}A_{\infty_{1}}$ is prime, hence  $\mathfrak{a}_{d-1}A_{\infty_{1}}=\mathfrak{a}_{d_{0}-1}$.
        In particular,  the expansion  $\mathfrak{b}=\mathfrak{b}_{f}A_{\infty_{1}}$ is therefore a product of primes that do not involve $\mathfrak{a}_{d_{0}-1}$.  By the commutativity
        of the diagram, $\mathfrak{b}\in Z_{0} = \langle \mathfrak{a}_{d_{0}-1}\rangle$, contradiction.
        \end{proof}

 Every ideal $\mathfrak{a}\subset A_{f}$ relatively prime to $\mathfrak{c}$ is invertible (see Theorem 1.3 of \cite{DGVI}), therefore 
${\sf I}^{\mathfrak{c}}_{A_{f}} : =\big\{ \mathfrak{a} \text{ an }A_{f}\text{-fractional ideal}  : \; (\mathfrak{a},\mathfrak{c})=1\big\}$
forms a subgroup of $ {\sf I}^{\ast}_{A_{f}} $.  Every class in ${\sf Cl}_{A_{f}} $ contains a representative in $ {\sf I}^{\mathfrak{c}}_{A_{f}} $ (Lemma
1.3 of \cite{DGVI}), so
\begin{align}\label{condclassgroup} {\sf Cl}_{A_{f}}  
= \bigg( {\sf I}^{\mathfrak{c}}_{A_{f}}\cdot  {\sf P}_{A_{f}} \bigg)\bigg/{\sf P}_{A_{f}}  \cong  {\sf I}^{\mathfrak{c}}_{A_{f}}/   {\sf P}_{A_{f}}^{\mathfrak{c}}, \end{align}
where
$ {\sf P}_{A_{f}}^{\mathfrak{c}} :=  {\sf P}_{A_{f}}\cap {\sf I}^{\mathfrak{c}}_{A_{f}} = \big\langle aA_{f}\; : \;\; a\in A_{f},\; (a,\mathfrak{c}) =1\big\rangle $.
We may now define the {\bf {\em wide ray class group}} of $\mathfrak{m}_{f}$ as $ {\sf Cl}_{\mathfrak{m}_{f}}^{0} := {\sf I}^{0}_{\mathfrak{m}_{f}}/{\sf P}^{0}_{\mathfrak{m}_{f}}$,
   where $   {\sf I}^{0}_{\mathfrak{m}_{f}} = \{ \mathfrak{a} \in {\sf I}_{A_{f}}^{\mathfrak{c}}   :\; (\mathfrak{a},\mathfrak{m}_{f})=1\}$ and
${\sf P}^{0}_{\mathfrak{m}_{f}}=\langle bA_{f} \in {\sf I}^{0}_{\mathfrak{m}_{f}}    :\;\ b\in A_{f}, \;\; (b,\mathfrak{c})=1,\;\;  b\equiv 1 \mod\mathfrak{m}_{f} \rangle$.
 Thus, in this definition, we are imposing from the outset primeness to $\mathfrak{c}$, in the spirit of the identification (\ref{condclassgroup}).

  The {\bf {\em wide ray class field}} is defined as
   the unique abelian extension $K^{0}_{\mathfrak{m}_{f}}/K$ whose norm group corresponds to $ {\sf Cl}_{\mathfrak{m}_{f}}^{0}$ via reciprocity.
   More exactly, suppose that in $A_{\infty_{1}}$,
$  \mathfrak{m} = \prod \mathfrak{p}^{n_{p}} $.
   Then define
   \[  \I^{0}_{\mathfrak{m}_{f}} = K_{\infty_{1}}^{\times} \cdot \prod U_{\mathfrak{p},f}^{(n_{\mathfrak{p}})},\quad U_{\mathfrak{p},f}^{(n_{\mathfrak{p}})} := 
   U_{\mathfrak{p}}^{(n_{\mathfrak{p}})} \cap (A_{f})_{\mathfrak{p}} , \]
   where $U_{\mathfrak{p}}^{(n)}$ is the $n$th higher unit group in the local field $K_{\mathfrak{p}}$. 
If we define
$  C^{0}_{\mathfrak{m}_{f}} : = (\I^{0}_{\mathfrak{m}_{f}}\cdot K^{\times})/K^{\times}$, 
 we have the following analog of Theorem \ref{CFTHCF}:
   
   \begin{theo} $C^{0}_{\mathfrak{m}_{f}} <C_{K}$ is of finite index and 
   \[   C_{K}/ C^{0}_{\mathfrak{m}_{f}}  \cong   {\sf Cl}_{\mathfrak{m}_{f}}^{0}.\]
      \end{theo}
   
 \begin{proof} The proof is a straightforward modification of that of Theorem \ref{CFTHCF}.  Define
\[ \I^{\mathfrak{c}}_{\mathfrak{m}_{f}} = K^{\times}_{\infty_{1}} \cdot \prod_{\mathfrak{p}| \mathfrak{c}, \; \mathfrak{p}\nmid \mathfrak{m}} (A_{f})^{\times}_{\mathfrak{p}} \cdot  
\prod_{\mathfrak{p}|\mathfrak{m}}  U_{\mathfrak{p},f}^{(n_{\mathfrak{p}})}\cdot
\prod'_{\infty_{1}\not=\mathfrak{p}\;\nmid \;\mathfrak{c},\mathfrak{m}}  K^{\times}_{\mathfrak{p}} \]
where, as before, the last product is restricted over $(A_{f})^{\times}_{\mathfrak{p}} = (A_{\infty_{1}})_{\mathfrak{p}}^{\times}$.
Then 
$ \I_{K} =  \I^{\mathfrak{c}}_{\mathfrak{m}_{f}}  \cdot K^{\times} $ 
and we may define an epimorphism
$ C_{K}  \cong  \I^{\mathfrak{c}}_{\mathfrak{m}_{f}} /(\I^{\mathfrak{c}}_{\mathfrak{m}_{f}} \cap K^{\times} )\longrightarrow {\sf Cl}_{\mathfrak{m}_{f}}^{0}$, 
induced by 
\[   \I^{\mathfrak{c}}_{\mathfrak{m}_{f}} \ni \upalpha =( \upalpha_{\mathfrak{p}}) \longmapsto \prod \mathfrak{p}_{f}^{v_{\mathfrak{p}} (\upalpha_{\mathfrak{p}})}   \in  {\sf I}^{0}_{\mathfrak{m}_{f}}   \]
since the latter map takes $a \in \I^{\mathfrak{c}}_{\mathfrak{m}_{f}} \cap K^{\times}$ to $(a)$ which is relatively prime to $\mathfrak{c}$ and for which $a\equiv 1\mod \mathfrak{m}_{f}$.
The kernel of this map contains $C^{0}_{\mathfrak{m}_{f}}$ and if $\upalpha\in \I^{\mathfrak{c}}_{\mathfrak{m}_{f}} $ is contained in the kernel, then there is $a\in {\sf P}^{0}_{\mathfrak{m}_{f}}$
with $(\upalpha ) = (a)$, whereby 
\[  \upalpha \equiv \upalpha a^{-1} \mod \I^{\mathfrak{c}}_{\mathfrak{m}_{f}} \cap K^{\times}\;\; \in\;\; \I^{0}_{\mathfrak{m}_{f}} \mod \I^{\mathfrak{c}}_{\mathfrak{m}_{f}} \cap K^{\times}  ,  \]
which is contained in the image of $C^{0}_{\mathfrak{m}_{f}}$ in $ \I^{\mathfrak{c}}_{\mathfrak{m}_{f}}\mod \I^{\mathfrak{c}}_{\mathfrak{m}_{f}} \cap K^{\times} $.
   \end{proof}
  If we denote 
$  (A_{f}/\mathfrak{m}_{f})_{\mathfrak{c}}^{\times} := \{ a + \mathfrak{m}_{f}\in (A_{f}/\mathfrak{m})^{\times} : \; (a,\mathfrak{c})=1\}  $,
   note that we have isomorphisms
   ${\rm Gal}(K^{0}_{\mathfrak{m}_{f}}/H_{A_{f}})\cong {\rm Ker}\big(  {\sf Cl}_{\mathfrak{m}_{f}}^{0}\rightarrow {\sf Cl}_{A_{f}} \big) \cong (A_{f}/\mathfrak{m}_{f})_{\mathfrak{c}}^{\times}/\F_{q}^{\times}$,
    the last of which follows directly from the definitions of the class groups $ {\sf Cl}_{\mathfrak{m}_{f}}^{0}$ and ${\sf Cl}_{A_{f}} $.


The narrow versions of the above notions are defined in order to allow the sign group $\F_{q}^{\times}$ to have a Galois interpretation as well.  
Thus, if we denote by 
${\sf P}_{\mathfrak{m}_{f}}{}  =\{   bA_{f}\in {\sf P}_{\mathfrak{m}_{f}}^{0} |\;  b\text{ is monic}
 \} $,
then the {\bf {\em narrow ray class group}} is defined
$   {\sf Cl}_{\mathfrak{m}_{f}}{}  := {\sf I}_{\mathfrak{m}_{f}}/{\sf P}_{\mathfrak{m}_{f}}{} \twoheadrightarrow  {\sf Cl}_{\mathfrak{m}_{f}}^{0}$.
We will also call this simply the {\it ray class group}, since it is the one of principal interest (following the convention of \cite{Neu}).
Note that by our choice of sign function in \S  \ref{QDMSection},
$b$ is monic if and only if ${\rm sgn}(b)=1$.

\begin{note}  Every class in  ${\sf Cl}_{\mathfrak{m}_{f}}{} $ contains an integral element.  Indeed, take $\mathfrak{a}\in {\sf I}_{\mathfrak{m}_{f}}$:  $\mathfrak{a}$ is prime to $\mathfrak{c}$, so
the inverse $\mathfrak{a}^{-1}$ exists. Moreover, since $\mathfrak{a}$ is prime to $\mathfrak{m}_{f}$, so is $\mathfrak{a}^{-1}$ 
\end{note}

 We define again by reciprocity 
   the {\bf {\em narrow ray class field}} or simply {\bf {\em ray class field}} as
 $ K_{\mathfrak{m}_{f}}{}/K$, whose norm group $C{}_{\mathfrak{m}_{f}}$ is obtained from that of $K_{\mathfrak{m}_{f}}^{0}$ by
   replacing the factor of $K_{\infty_{1}}^{\times}$ in $\I^{0}_{\mathfrak{m}_{f}}$ by $ (K_{\infty_{1}}^{\times})_{+} = \text{the subgroup of sign 1 elements in } K_{\infty_{1}}^{\times}$:
   \begin{align}\label{normgroupnarrowrcf}
   \I_{\mathfrak{m}_{f}} =(K_{\infty_{1}}^{\times})_{+} \cdot \prod U_{\mathfrak{p},f}^{(n_{\mathfrak{p}})},\quad U_{\mathfrak{p},f}^{(n_{\mathfrak{p}})} := U_{\mathfrak{p}}^{(n_{\mathfrak{p}})} \cap (A_{f})_{\mathfrak{p}} . 
   \end{align}
   By the order reversing property of reciprocity, we have 
$ K_{\mathfrak{m}_{f}}{}\supset K_{\mathfrak{m}_{f}}^{0}$, 
   and 
   \begin{align}\label{WID} {\rm Gal}(K_{\mathfrak{m}_{f}}{}/H_{A_{f}})\cong (A_{f}/\mathfrak{m}_{f})_{\mathfrak{c}}^{\times},\quad {\rm Gal}(K_{\mathfrak{m}_{f}}{}/K_{\mathfrak{m}_{f}}^{0})\cong \F_{q}^{\times}.\end{align}

   
   We define the counterparts of the above constructions for $\mathfrak{M}_{f}\subset \mathcal{O}_{f}$.  The wide ray class group is defined exactly as above:
$   {\sf Cl}_{\mathfrak{M}_{f}}^{0} := {\sf I}_{\mathfrak{M}_{f}}/{\sf P}^{0}_{\mathfrak{M}_{f}}$,
   where 
  $   {\sf I}_{\mathfrak{M}_{f}} = \{ \mathfrak{A}\in {\sf I}_{\mathcal{O}_{f}}^{\mathfrak{C}}|\; (\mathfrak{A},\mathfrak{M}_{f})=1\}$ and
  ${\sf P}_{\mathfrak{M}_{f}}^{0}=\langle b\mathcal{O}_{f} \in  {\sf I}_{\mathfrak{M}_{f}} |\; b\in \mathcal{O}_{f},\; (b,\mathfrak{C})=1,\;b\equiv 1 \mod\mathfrak{M}_{f} \rangle $.

The narrow version takes into account that $\mathcal{O}_{K}\supset \mathcal{O}_{f}$ is defined using the two places $\infty_{1},\infty_{2}$, and requires ``positivity'' at both.  That is, if $\upsigma\in {\rm Gal}(K/k)$
is the non-trivial element, we define
${\sf P}_{\mathfrak{M}_{f}}{} =\{   b\mathcal{O}_{f} \in {\sf P}_{\mathfrak{M}_{f}}^{0}|\;  
b,b^{\upsigma}\text{ are monic}  \}$
and then 
the narrow ray class group is defined
$   {\sf Cl}_{\mathfrak{M}_{f}}{}  := {\sf I}_{\mathfrak{M}_{f}}/{\sf P}_{\mathfrak{M}_{f}}{} \twoheadrightarrow  {\sf Cl}_{\mathfrak{M}_{f}}^{0}$.
These groups correspond to subgroups of the id\`{e}le class group 
$    C_{\mathfrak{M}_{f}}<C^{0}_{\mathfrak{M}_{f}} < C_{K} $
and  by reciprocity, we define the class fields 
$K_{\mathfrak{M}_{f}}^{0}\subset K_{\mathfrak{M}_{f}}{}$.  For $f=f_{0}$ fundamental, $\mathfrak{M}=\mathfrak{M}_{f_{0}}$ and the wide ray class field
$K_{\mathfrak{M}}^{0}$ is characterized as the maximal extension amongst
  abelian extensions $K_{\mathfrak{f}}/K$ completely split at $\infty_{1},\infty_{2}$ and with conductor $\mathfrak{f}\leq \mathfrak{M}$.

   Since $C^{0}_{\mathfrak{m}_{f}} < C^{0}_{\mathfrak{M}_{f}} $, we have
    $K_{\mathfrak{m}_{f}}^{0}\supset K_{\mathfrak{M}_{f}}^{0}$.
       Unfortunately, it is not the case that $K_{\mathfrak{m}_{f}}{}\supset K_{\mathfrak{M}_{f}}{}$; in fact there is not even a natural map relating ray class groups
      ${\sf Cl}_{\mathfrak{m}_{f}}{}$ and  ${\sf Cl}_{\mathfrak{M}_{f}}{}$.  This situation leads us to define intermediate ray class fields 
      \begin{align}\label{intermediateRCFs}  K_{\mathfrak{M}_{f}, 1}, \quad \text{resp.}\quad K_{\mathfrak{M}_{f} ,2}\end{align} which are ``narrow'' only along $\infty_{1}$, resp.\ $\infty_{2}$:
      defined by the corresponding groups ${\sf Cl}_{\mathfrak{M}_{f}, 1}:=  {\sf I}_{\mathfrak{M}_{f}}/{\sf P}_{\mathfrak{M}_{f}, 1} $, ${\sf Cl}_{\mathfrak{M}_{f},2}:=  {\sf I}_{\mathfrak{M}_{f}}/{\sf P}_{\mathfrak{M}_{f}, 2} $, where 
     $ {\sf P}_{\mathfrak{M}_{f}, 1}  =\langle   b\mathcal{O}_{f} \in {\sf P}^{0}_{\mathfrak{M}_{f}}| \; b\equiv 1 \mod\mathfrak{M}_{f} \text{ and } b\text{ is monic}  \rangle$, 
       ${\sf P}_{\mathfrak{M}_{f}, 2}  =\langle   b\mathcal{O}_{f} \in {\sf P}^{0}_{\mathfrak{M}_{f}}|\;\;  b\equiv 1 \mod\mathfrak{M}_{f}\text{ and } b^{\upsigma}\text{ is monic}  \rangle $.
     \begin{prop}\label{intermediateRCFsProp} We have the following diagram of ray class fields for $\mathfrak{M}_{f}$, with extensions
      labeled by Galois groups:
           \begin{align*}
   \begin{diagram}
   & & K_{\mathfrak{M}_{f}}{} & &  \\
   & \ldLine^{\F_{q}^{\times}\circlearrowright} & &   \rdLine^{\F_{q}^{\times}\circlearrowright}  &  \\ 
   K_{\mathfrak{M}_{f}, 1} & & & &   K_{\mathfrak{M}_{f},2} \\
     & \luLine_{\F_{q}^{\times}\circlearrowright} & &   \ruLine_{\F_{q}^{\times}\circlearrowright} &\\ 
     & & K_{\mathfrak{M}_{f}}^{0} & &  \\
      & & \dLine & & \\
     & & H_{\mathcal{O}_{f}} & & 
   \end{diagram}
   \end{align*}
   \end{prop}
   
  \begin{proof}  To see that the Galois groups appearing are as claimed, we note first that the sign function $g\mapsto {\rm sgn}(g)$ (= coefficient of first non-zero term of the Laurent expansion of $g$ in $T^{-1}$) descends to a well defined epimorphism
   $ {\rm sgn}:{\sf P}_{\mathfrak{M}_{f}}^{0}\longrightarrow \F_{q}^{\times}$, in view of the defining condition $b\equiv 1\mod \mathfrak{M}_{f}$ for $b\mathcal{O}_{f}\in {\sf P}_{\mathfrak{M}_{f}}^{0}$.  (To see that the map is onto, one simply takes an element of $\mathfrak{M}_{f}$ of absolute value $>1$ and of sign $c\in\F^{\times}_{q}$ and adds $1$ to it to obtain $b\equiv 1\mod\mathfrak{M}_{f}$
   with sign $c$.)
   Then one sees that for $i=1,2$, 
$ {\rm Gal}( K_{\mathfrak{M}_{f}, i} /K_{\mathfrak{M}_{f}}^{0}) \cong {\sf P}_{\mathfrak{M}_{f}}^{0}/{\sf P}_{\mathfrak{M}_{f}, i} \cong \F_{q}^{\times}$,  where the last isomorphism is given by $b\mathcal{O}_{f}\mapsto {\rm sgn}(b)$, $b\mathcal{O}_{f}\mapsto {\rm sgn}(b^{\upsigma})$ for $i=1,2$. Similarly
${\rm Gal}( K_{\mathfrak{M}_{f}}{}/K_{\mathfrak{M}_{f}, i}) \cong {\sf P}_{\mathfrak{M}_{f}, i}/{\sf P}_{\mathfrak{M}_{f}}{} \cong \F_{q}^{\times}$. Finally, the map
   \[  {\rm Gal}(K_{\mathfrak{M}_{f}}{}/K_{\mathfrak{M}_{f}}^{0})\cong {\sf P}_{\mathfrak{M}_{f}}^{0}/{\sf P}_{\mathfrak{M}_{f}}{}\longrightarrow \F^{\times}_{q}\times \F^{\times}_{q},\quad b\mathcal{O}_{f}\mapsto ({\rm sgn}(b),{\rm sgn}(b^{\upsigma}) ) \]
   is an isomorphism.  Thus $K_{\mathfrak{M}_{f}}{}$ is the compositum of $K_{\mathfrak{M}_{f}, 1}$ and $K_{\mathfrak{M}_{f}, 2}$. 
   \end{proof}

By class field theory, we have $K_{\mathfrak{m}_{f}}{} \supset K_{\mathfrak{M}_{f}, 1}$ and hence
       the following diagram of field extensions labeled by Galois groups:
   \begin{align*}
   \begin{diagram}
   & & K_{\mathfrak{m}_{f}}{} & & \\
   & \ldLine^{W\circlearrowright} & &   \rdLine^{Z'\circlearrowright}  & \\ 
    H_{A_{f}} & & & & K_{\mathfrak{M}_{f}, 1}\\
      & \luLine_{Z\circlearrowright} & &   \ruLine_{W'\circlearrowright} & \\ 
       & & H_{\mathcal{O}_{f}} & & \\
       & & \dLine & & \\
       & & K & &
   \end{diagram}
   \end{align*}
    In the diagram $Z:={\rm Gal}(H_{A_{f}}/H_{\mathcal{O}_{f}})  \cong {\rm Ker}( {\sf Cl}_{A_{\infty_{1}}} \twoheadrightarrow {\sf Cl}_{\mathcal{O}_{K}}  ) =  \langle \mathfrak{a}_{d_{0}-1}  \rangle \subset {\sf Cl}_{A_{\infty_{1}}} $, by  Proposition \ref{ZProp}.
   The group $Z':={\rm Gal}(K_{\mathfrak{m}_{f}}{}/K_{\mathfrak{M}_{f}, 1})$ is isomorphic, again by Class Field Theory, to 
${\rm Ker}( {\sf Cl}_{\mathfrak{m}_{f}}{} \twoheadrightarrow{\sf Cl}_{\mathfrak{M}_{f}, 1}   )$. 
As noted in (\ref{WID}), the group $W := {\rm Gal}(K_{\mathfrak{m}_{f}}{}/H_{A_{f}})$  may be identified with 
$  {\rm Ker}( {\sf Cl}_{\mathfrak{m}_{f}}{} \twoheadrightarrow {\sf Cl}_{A_{f}} ) 
   \cong (A_{f}/\mathfrak{m}_{f})_{\mathfrak{c}}^{\times}$.   We may similarly identify $W':={\rm Gal}(K_{\mathfrak{M}_{f}, 1}/H_{\mathcal{O}_{f}})$
 with 
   \[   {\rm Ker}( {\sf Cl}_{\mathfrak{M}_{f}, 1} \twoheadrightarrow {\sf Cl}_{\mathcal{O}_{f} })=
   \left( \{ b\mathcal{O}_{f}\in {\sf I}_{\mathfrak{M}_{f}}  |\; b\not\equiv 1\mod \mathfrak{M}_{f}\text{ for $b\in \mathcal{O}_{f}$ monic}\}\cup \{ 1\}\right) /\text{\sf P}_{\mathfrak{M}_{f}, 1}  .\]
   
   The restriction maps of Galois groups yield  homomorphisms
$  \uppsi: W\longrightarrow W'$, $\upphi:Z'\longrightarrow Z$.
 Let
$B:= 
   \langle fA_{f}\mod \text{\sf P}_{\mathfrak{m}_{f}}{}\rangle $.
We claim
$ B \subset {\rm Ker}(\uppsi)\cap  {\rm Ker}(\upphi)    \subset  W\cap Z'$.
   Indeed, $B\subset {\rm Ker}(\uppsi)$ since $fA_{f}$ maps to $f\mathcal{O}_{f}= (1)$ in ${\sf Cl}_{\mathfrak{M}_{f}, 1}$
   (as $f$ is a unit in $\mathcal{O}_{f}$) and where $B\subset {\rm Ker}(\upphi)$ since $fA_{f}$ is a principal ideal and
   so is trivial in ${\sf Cl}_{A_{f}}$.

   
   \begin{lemm}\label{GaloisGroupLemma}  Let $B$ be as above.  The restriction maps  $\uppsi: W\longrightarrow W', \upphi:Z'\longrightarrow Z$. are epimorphisms and
$ B={\rm Ker}(\uppsi) = {\rm Ker}(\upphi)=W\cap Z'$.  
\end{lemm}
\begin{proof} We begin by showing that $\uppsi$ is an epimorphism.  Let $\upalpha \mathcal{O}_{f}\mod \text{\sf P}_{\mathfrak{M}_{f}, 1} $ correspond by reciprocity to $\upsigma'\in W'$.  Then since $\mathcal{O}_{f}=\bigcup_{n\in\N} f^{-n}A_{f}$ we may
write $\upalpha =af^{-n}$, for $a\in A_{f}$ and $n\in\N$, and $\upalpha  \mathcal{O}_{f}=a \mathcal{O}_{f}$. By reciprocity, $\uppsi$ is induced by the expansion
$bA_{f}\mod\text{\sf P}_{\mathfrak{m}_{f}}{} \longmapsto b\mathcal{O}_{f}\mod\text{\sf P}_{\mathfrak{M}_{f}, 1}$.   Therefore if $\upsigma$ corresponds to $aA_{f} \mod \text{\sf P}_{\mathfrak{M}_{f}, 1}$ via reciprocity,
$ \uppsi \left(\upsigma \right) =\upsigma'$, hence 
$\uppsi$ is onto.  We also have surjectivity of $\upphi$: by Lemma \ref{coprimalityofideals}, $\mathfrak{a}_{i}\in {\sf I}_{\mathfrak{m}_{f}}$ (note that being prime to the conductor
implies invertibility for any power of a prime ideal, by Theorem 1.4 of \cite{DGVI}), thus
we have (again passing via reciprocity
to maps on ray classes)
$  \upphi \left( \mathfrak{a}_{i}\mod { {\sf P}}_{\mathfrak{m}_{f}}{} \right)=\mathfrak{a}_{i}\mod{ {\sf P}}_{A_{f}} $. Thus $\upphi$ is onto.
Now let $aA_{f}\mod {{\sf P}}_{\mathfrak{m}_{f}}{}\in {\rm Ker}(\uppsi )$.  Since (narrow) ray classes always have integral representatives (e.g.\ see \cite{Neu} page 368), 
without loss of generality we may assume that $a\in A_{f}$ (and is monic).    As $aA_{f} \mod {{\sf P}}_{\mathfrak{m}_{f}}{}\in {\rm Ker}(\uppsi )$, 
$a\mathcal{O}_{f}$ has a monic generator $\equiv 1\mod \mathfrak{M}_{f}$, and such a generator must be a multiple of $a$ by a power of $f$. That is, there exists $m\in\Z$ so that $a\equiv f^{m}\mod\mathfrak{M}_{f}$.  If $m\geq 0$ this implies $a\equiv f^{m}\mod\mathfrak{m}_{f}$ i.e. 
$aA_{f}\equiv f^{m}A_{f} \mod { {\sf P}}_{\mathfrak{m}_{f}}{}$ so $aA_{f}\mod \text{\sf P}_{\mathfrak{m}_{f}}{}\in B$. (Here we use the fact that $f^{m}$ is monic.)  If $m<0$ then we
have $af^{-m}\equiv 1\mod \mathfrak{M}_{f}$ which implies $af^{-m}\equiv 1\mod \mathfrak{m}_{f}$ and $aA_{f}\mod \text{\sf P}_{\mathfrak{m}_{f}}{}\in B$ again.
Thus ${\rm Ker}(\uppsi )=B$.  By the diagram above, this completes the proof, since ${\rm Ker}(\upphi )\supset B$.  Indeed, let $d=\#Z$, $e=\#W$, $d'=\#Z'$, $e'=\#W'$,
$b=\#B$ and $b'=\#{\rm Ker}(\upphi )\geq b$.  Then $de=d'e'$,  $e=be'$ and $d' =db'$,  so
$  dbe' =de= d'e'=db'e' $ implies $b'=b$, hence ${\rm Ker}(\upphi )=B$.  
\end{proof}



It follows that if $f\equiv 1\mod\mathfrak{m}_{f}$,  $B=1$ and
  $W'\cong W$,
    $Z'\cong Z$.  In particular,
\begin{align}\label{GalProd} {\rm Gal}(K_{\mathfrak{m}_{f}}{}/H_{\mathcal{O}_{f}})\cong W\times Z.\end{align}
In this case, $K_{\mathfrak{m}_{f}}$ is the compositum of $H_{A_{f}}$ and $K_{\mathfrak{M}_{f}, 1}$.  This is the situation which is of interest to us. 
Thus, if $f=f_{0}$ is a fundamental unit and  $\not\equiv 1\mod \mathfrak{m}= \mathfrak{m}_{f_{0}}$, we replace $f_{0}$ by a suitable power $f=f^{n}_{0}$ so that $f\equiv 1\mod\mathfrak{m}$ (implying
$f\equiv 1\mod\mathfrak{m}_{f}$ as well).  





It will be imperative, if we wish to prove a Theorem which has a plausible analog in the number field case, to work with moduli $\mathfrak{M}_{f}\subset \mathcal{O}_{f}$ with 
$f\equiv 1\mod \mathfrak{M}_{f}$, and {\it not} use moduli $\mathfrak{m}_{f}\subset A_{f}$: this is because the class fields $K_{\mathfrak{m}_{f}}$ for $\mathfrak{m}\subset A_{f}$ do not (yet) have a counterpart
in characteristic $0$.  In this connection, we will need to make an additional requirement: that the conjugate modulus 
\[ \mathfrak{M}_{f}^{\varphi}, \quad \varphi = \text{ the nontrivial element of } {\rm Gal}(K/ k) \]
satisfies as well
$ f\equiv 1\mod \mathfrak{M}_{f}^{\varphi}$ , or equivalently,  $f-1\in \mathfrak{M}_{f}^{\varphi} $.
Since 
$(f-1)^{\varphi} = f^{\varphi}-1 $ $= N(f)f^{-1} -1 \in \mathfrak{M}^{\varphi}$ $ \Leftrightarrow $ $f-N(f) \in  \mathfrak{M}_{f}^{\varphi} $,
we see that $ f-1\in \mathfrak{M}_{f}\cap \mathfrak{M}^{\varphi}_{f} $ $ \Leftrightarrow $
$ N(f)= \text{the norm of $f$} = (-{\tt b}_{0})^{n} =1$.

\begin{rema}\label{N=1list}
We record here the possible values of $n$ which produce $N(f)=1$, according to the value of ${\tt b}_{0}$:
\begin{enumerate}
     \item[1.] If ${\tt b}_{0}=1$, $N(f)=1$ for all $n$ even.
     \item[2.] If ${\tt b}_{0}=-1$, $N(f)=1$ is always satisfied.
     \item[3.] If ${\tt b}_{0}\not= \pm 1$, then $N(f)=1$ implies
     \begin{enumerate}
     \item[i.] ${\tt b}^{n}_{0}=1$ for $n$ even i.e.\ the order of ${\tt b}_{0}$ in $\F_{q}^{\times}$ divides $n$.
     \item[ii.] ${\tt b}^{n}_{0}=-1$ for $n$ odd i.e.\ the order of ${\tt b}_{0}$ in $\F_{q}^{\times}$ divides $2n$.
     \end{enumerate}
     \end{enumerate}
     In particular, for any ${\tt b}_{0}\in \F_{q}^{\times}$, there will be infinitely many $n$ for which $N(f)=1$.
     \end{rema}


In view of the above discussion, we make the following definitions.  Consider for each $f$ with $N(f)=1$ the family of ideals 
\begin{align} \label{funitmoduli} \mathcal{M}_{f}=\bigg\{  \mathfrak{M}_{f}\subset \mathcal{O}_{f}\; :\;\;  f\equiv 1\mod \mathfrak{M}\subset \mathcal{O}_{K}
\bigg\} 
.\end{align}
Since $(f-1)\mathcal{O}_{f}\subset \mathfrak{M}_{f}$ for all $ \mathfrak{M}_{f}\in \mathcal{M}_{f}$, $\mathcal{M}_{f}$ is finite.
\begin{prop}  $ \mathcal{M}_{f}$ is conjugate-invariant:
$\mathfrak{M}_{f}\in \mathcal{M}_{f}$ $\Leftrightarrow $ $
\mathfrak{M}^{\varphi}_{f}\in \mathcal{M}_{f}$.
\end{prop}

\begin{proof}  Since $N(f)=1$, $f^{\varphi}=f^{-1}$, 
 hence $f-1\in \mathfrak{M}$ implies
 $f^{-1}-1\in \mathfrak{M}^{\varphi}$, and since $f$ is an $\mathcal{O}_{K}$-unit, the latter implies $f-1\in \mathfrak{M}^{\varphi}$.
\end{proof}
Every ideal  $\mathfrak{M}_{f}\in \mathcal{M}_{f}$ contains the special modulus
\[  \mathfrak{M}^{\rm unit}_{f} :=  (f-1)\mathcal{O}_{f}\]
which we call the {\bf {\em unit modulus}}.   Note that the contraction of $(f-1)\mathcal{O}_{f}$ to $A_{f}$ is just $(f-1)A_{f}$.  Indeed, suppose that  $g\in \mathcal{O}_{f}$ is
such that $g(f-1)\in A_{f}$.  As $f$ has a zero at $\infty_{2}$, $f-1$ has degree $0$ there, hence $g$ cannot have a pole at $\infty_{2}$, since this would imply $g(f-1)$ also has a pole at $\infty_{2}$.  Thus $g\in A_{f} = \mathcal{O}_{f}\cap A_{\infty_{1}}$.  

Write 
\begin{align}\label{admissiblemoduli} \mathcal{M} =\bigcup_{N(f)=1} \mathcal{M}_{f}.\end{align} 

\begin{theo}\label{cofinalitytheo}The family of {\bf {\em unit ray class fields}}
\[    \mathcal{K}_{\rm unit} := \left\{ K_{\mathfrak{M}_{f}} \; : \;\;  \mathfrak{M}_{f}\in \mathcal{M} \right\}\]
is cofinal in the set of finite abelian
extensions of $K$ split completely at $\infty_{1}, \infty_{2}$.
\end{theo}

\begin{proof} It will be enough to show that the class fields
$ K_{\mathfrak{M}^{\rm unit}_{f}} $
are cofinal, since they are clearly cofinal in  $ \mathcal{K}_{\rm unit} $.  Let $\mathfrak{M}\subset \mathcal{O}_{K}$ be a modulus: then there exists $f=f_{0}^{n}$ with $f-1\in \mathfrak{M}$ and $N(f)=1$.
By the id\`{e}lic description of ray class fields, the
contraction $\mathfrak{M}_{f}$ of $\mathfrak{M}$ to $\mathcal{O}_{f}$ satisfies
$  K_{\mathfrak{M}} \subset K_{\mathfrak{M}_{f}} \subset   K_{\mathfrak{M}^{\rm unit}_{f}} \in\mathcal{K}_{\rm unit}$.
\end{proof}

The class field $K_{\mathfrak{M}^{\rm unit}_{f}}$ will be called the {\bf {\em principal unit ray class field}} associated to $f=f_{0}^{n}$.   The unit ray class fields may be defined in the number field setting, where they may be viewed as real quadratic analogs of the
cyclotomic family.

   Consider the generalized cyclotomic polynomial 
      \[  P_{n} (X) = X^{n-1} - {\tt b}_{0}X^{n-2} + {\tt b}_{0}^{2}X^{n-3}  -\cdots + (-1)^{n-1} {\tt b}_{0}^{n-1}  \]
      where $-{\tt b}_{0}$ is the constant term in the minimal polynomial  (\ref{minpolyoff}) for $f_{0}$.  In Theorem \ref{PhiInConductor} of the Appendix, we prove (using a slightly different notation) that
      \begin{align}\label{PhiContainedInConductor}  \Uptheta_{f}:=  P_{n}(f_{0}^{2} )\in\mathfrak{c}_{f}\subset  \mathfrak{C}_{f}.\end{align}

      \begin{theo}\label{PrimenessToC} Suppose that $N(f_{0})=-{\tt b}_{0}=1$ and let $\mathfrak{M}_{f} \in \mathcal{M}_{f}$. 
      Then $\mathfrak{M}_{f}$ is prime to $\mathfrak{C}_{f}$, and as a consequence, $\mathfrak{M}_{f}$ is invertible in $\mathcal{O}_{f}$.
      \end{theo} 

      \begin{proof} 
      We first show that for any $f=f_{0}^{n}$  (not necessarily satisfying $N(f)=1$), that if
  \begin{align}\label{bnnotequal-1n-1} {\tt b}_{0}^{n}\not= (-1)^{n-1}
   \end{align}
      then $\mathfrak{M}_{f}$ is prime to $\mathfrak{C}_{f}$.  Indeed,
      by (\ref{PhiContainedInConductor}), $(f_{0}^{2}-{\tt b}_{0}) \Uptheta_{f}= f_{0}^{2n}+(-1)^{n}{\tt b}_{0}^{n}\in \mathfrak{C}_{f}$,
      and since $f-1=f_{0}^{n}-1\in\mathfrak{M}_{f}$, $f_{0}^{2n}-1\in \mathfrak{M}_{f}$.  Thus 
   $ (-1)^{n}{\tt b}^{n}_{0} +1=$ $ (f_{0}^{2n} + (-1)^{n}{\tt b}^{n}_{0}) -( f_{0}^{2n}-1) $ $ \in  \mathfrak{C}_{f} +\mathfrak{M}_{f} $.
      So, if ${\tt b}_{0}^{n}\not= (-1)^{n-1}$, this implies $\mathfrak{M}_{f} + \mathfrak{C}_{f}=\mathcal{O}_{f}$.  
     If $N(f_{0})=-{\tt b}_{0}=1$, then ${\tt b}_{0}^{n}=(-1)^{n}$, so since ${\rm char}( \F_{q})\not =2$, ${\tt b}^{n}_{0}\not= (-1)^{n-1}$
      and the result follows. 
      \end{proof}

     
      \begin{exam} 
      Consider the quadratic unit $g$ satisfying
      \begin{align}\label{formofg} g^{2} =T^{2}g -1.\end{align}  Then $N(g)=1$, and one may check that the square root of the discriminant $\sqrt{T^{4}-4}$ defines
      an element of $k_{\infty}$, hence $g\in k_{\infty}\setminus k$ is real.  Let $K=k(g)$ be the associated quadratic extension. Since the degree in $T$ of the linear coefficient in (\ref{formofg}) is 2, the order of $g$ at $\infty_{1}$ or $\infty_{2}$
      is $\pm 2$.  It follows that if $g$ is a power of a unit $f$, the power must be 2, i.e., $g=f^{2}$, and furthermore, the unit $f$ must satisfy
      $f^{2}= {\tt a} f + {\tt b} $ where $ {\tt a} = T + d$.  But then 
  $ g^{2} = T^{2}g -1 =  (  {\tt a} f + {\tt b} )^{2} = ( {\tt a} ^{2}+2 {\tt b} )g- {\tt b} ^{2}  $ 
      hence $d=0= {\tt b} $ i.e.\ $f$ cannot in fact be a unit.  This shows that $g$ is a fundamental unit. Since the constant term in (\ref{formofg}) is $-1$, we conclude, using item 2.\ in the proof of Theorem \ref{PrimenessToC} above, that for all $\mathfrak{M}\subset\mathcal{O}_{K}$, if $g^{n}-1\in\mathfrak{M}$, the contraction
      $\mathfrak{M}_{g^{n}}$ to the order $\mathcal{O}_{g^{n}}$ is prime to the conductor and thus invertible.
      \end{exam}

Consider again $\uprho_{i}$ the Hayes module associated to the ideal $\mathfrak{a}_{i}\subset A_{f}$.  We recall that when $A_{f}=A_{\infty_{1}}$, there exists a transcendental element $\upxi_{i}\in \C_{\infty}$ such that
$\uprho_{i}$ is isomorphic to $\C_{\infty}/\Uplambda_{i}$ where $\Uplambda_{i}= \upxi_{i}\mathfrak{a}_{i}$.  The isomorphism is given by 
\[ e_{i}:\C_{\infty}\longrightarrow \C_{\infty},\quad e_{i}(z)= z\prod_{0\not=\upalpha\in\mathfrak{a}_{i}}\left(1-\frac{z}{\upxi_{i}\upalpha}\right). \]
The module of $\mathfrak{m}$ torsion points,
$ \uprho_{i}[\mathfrak{m}] $, consists of $z\in \overline{K}\subset \C_{\infty}$ for 
which $\uprho_{i,\upbeta}(z)=0$ for all $\upbeta\in \mathfrak{m}$. $ \uprho_{i}[\mathfrak{m}] $ is a cyclic module, c.f.\   \S 16 of \cite{Hayes} or Proposition 2.3 of  \cite{DGVI}. By Hayes theory \cite{Hayes}, \cite{Goss}, \cite{Thak}, \cite{VS}, for any $i$ and generator $t\in \uprho_{i}[\mathfrak{m}]$
\begin{align}\label{primitivitylemma}  K_{\mathfrak{m}} = H_{A_{\infty_{1}}}(\uprho_{i}[\mathfrak{m}] )=H_{A_{\infty_{1}}}(t ). \end{align}
The second equality evinces the ``cyclotomic quality'' of the extension $K_{\mathfrak{m}}/H_{A_{\infty_{1}}}$: the generator $t$ is the analog of a primitive root of unity and
${\rm Gal}(K_{\mathfrak{m}}/H_{A_{\infty_{1}}})$ $\cong (A_{\infty_{1}}/\mathfrak{m})^{\times}$, exactly as is the case for cyclotomic extensions of $\Q$.  Thus Hayes' explicit 
Class Field Theory may be regarded as an analog of the Kronecker-Weber Theorem, available for any global field in positive characteristic.

The remainder of this paper will be devoted to proving the analog of the first equality appearing in (\ref{primitivitylemma}) for $K_{\mathfrak{M}_{f},i}/H_{\mathcal{O}_{f}}$, $i=1,2$.
The second equality in (\ref{primitivitylemma}) will not appear -- not surprising, as $\mathcal{O}_{f}$
is a ``rank 2'' object.  Our theorem will then be seen to be a real quadratic analog of the Theorem of Weber-Fueter.  

We now introduce the group which will play the role of $\uprho_{i}[\mathfrak{m}]$ for the extension $K_{\mathfrak{M}_{f},1}/H_{\mathcal{O}_{f}}$. In what follows,
we restrict to $\mathfrak{M}_{f}\in \mathcal{M}$, the latter family defined above in (\ref{admissiblemoduli}).  In particular, $\mathfrak{M}_{f}$ is either prime to $\mathfrak{C}$ -- implying the contraction $\mathfrak{m}_{f}$ is 
prime to $\mathfrak{c}$ --  or is the principal unit modulus $\mathfrak{M}_{f}^{\rm unit}$.  In either case, by Proposition 2.3 and Note 2.2 of \cite{DGVI},
 $\uprho_{0}[\mathfrak{m}_{f}]$ is of rank 1, and so is generated by some $t_{0}\in \uprho_{0}[\mathfrak{m}_{f}]$.
 
 Now, note
that if $t_{0}\in\uprho_{0}[\mathfrak{m}_{f}]$, then the action by the additive polynomial  $\Upphi_{i}$ associated to $\mathfrak{a}_{i}$ (defined in \S \ref{QDMSection}) satisfies
$ t_{i}=\Upphi_{i}(t_{0})\in \uprho_{d-i}[\mathfrak{m}_{f}] $,
and if $t_{0}$ is a generator of $\uprho_{0}[\mathfrak{m}_{f}]$, then $t_{d-i}$ is a generator of $ \uprho_{d-i}[\mathfrak{m}_{f}]$.
Moreover, if $\upsigma\leftrightarrow \mathfrak{a}_{i}$ via reciprocity, then $t_{0}^{\upsigma}=t_{i}$.  Thus the
orbit set $t^{\rm qt}=\{ t_{0},\dots ,t_{d-1}\}$ defines a quantum point in the sense of \S \ref{QDMSection}.
 We may thus define the {\bf {\em quantum $\mathfrak{M}_{f}$-torsion}}
 of $ \uprho^{\rm qt}$ by 
 \begin{align}\label{QTDef}   \uprho^{\rm qt}[\mathfrak{M}_{f}]  := \bigg\{ \{ t^{\upsigma}|\; \upsigma\in Z \}  \in {\C}_{\infty}^{\rm qt}\; :\;\; t\in  \uprho_{i}[\mathfrak{m}_{f}] \text{ for some } i \bigg\} ,\end{align}
 which is clearly an $A_{f}$ submodule of the quantum Drinfeld module $ \uprho^{\rm qt}$.
 
 \begin{prop}\label{IsoByQuantExp} The normalized quantum exponential $\widetilde{\exp}^{\rm qt}$ induces an isomorphism of $A_{f}$-modules
 \[ \uprho_{0}[\mathfrak{m}_{f}] \longrightarrow  \uprho^{\rm qt}[\mathfrak{M}_{f}] ,\quad t\longmapsto  \widetilde{\exp}^{\rm qt} (t) .\]
 \end{prop}
 
 \begin{proof} Follows  from 
 the definition (\ref{expqtnormalized})  of the normalized quantum exponential.
 \end{proof}
 \begin{note} The reader should not be surprised by the assertion of Proposition \ref{IsoByQuantExp}: even though the Drinfeld modules $\uprho_{i}$ are mutually non isomorphic,
 their $\mathfrak{m}$ torsion points are all isomorphic to $A_{f}/\mathfrak{m}_{f}$, c.f.\ Proposition 2.3 of \cite{DGVI}.
 \end{note}
 The trace 
 \[  {\sf Tr}( \uprho^{\rm qt}[\mathfrak{M}_{f}]) :=  \left\{\left. \sum_{\upsigma\in Z } t^{\upsigma} \right|\;\; t\in  \uprho_{i}[\mathfrak{m}_{f}] \text{ for some } i \right\}  \]
 is an abelian subgroup (but {\it not} an $\mathcal{O}_{f}$ submodule) of $K_{\mathfrak{M}_{f},1}$ since $ \sum_{\upsigma\in Z } t^{\upsigma}={\sf Tr}_{Z}(t)$
 where ${\sf Tr}_{Z}:K_{\mathfrak{m}_{f}}\longrightarrow K_{\mathfrak{M}_{f},1}$ is the trace map. 
In the sections which follow, we will show that
\[ K_{\mathfrak{M}_{f},1}= H_{\mathcal{O}_{f}}({\sf Tr}( \uprho^{\rm qt}[\mathfrak{M}_{f}])).\]

On the other hand, it is not possible to define {\it analytically} the quantum Drinfeld module associated to $A_{f^{-1}}$ (which, when $f=f_{0}$, is just $A_{\infty_{2}}$) {\it within} $k_{\infty}$, i.e.\ via its corresponding quantum exponential, since the latter
diverges, as we are working with the embedding in $k_{\infty}$ associated to $\infty_{1}$.  Nonetheless, there is no problem in defining
the quantum Drinfeld module of $A_{f^{-1}}$ {\it algebraically} within $k_{\infty}$.  In fact, 
if $  \varphi \in {\rm Gal}( K_{\mathfrak{M}_{f}} / \F_{q}(T))$ is a lift of the non trivial element of ${\rm Gal}(K/\F_{q}(T))$ to $K_{\mathfrak{M}_{f}}$, then $\varphi$ maps $A_{f}$ onto  $A_{f^{-1}}$ and $H_{A_{f}}$ onto 
$H_{A_{f^{-1}}}$.  

As for the ray class fields, denote as before the Galois conjugate of $\mathfrak{M}_{f}$ by $\mathfrak{M}_{f}^{\varphi}$.   Then $\varphi$ sends
$ K_{\mathfrak{M}_{f} , 1}$ onto $ K_{\mathfrak{M}_{f}^{\varphi} , 2}$ -- this follows since the action of $\varphi$ on $C_{K}$ takes $C_{\mathfrak{M}_{f},1}$ = the subgroup defining $ K_{\mathfrak{M}_{f} , 1}$ to
$C_{\mathfrak{M}^{\varphi} _{f},2}$ = the subgroup defining  $ K_{\mathfrak{M}_{f}^{\varphi} , 2}$.
 In particular, the corresponding quantum Drinfeld module is 
$ ( \uprho^{{\rm qt }})^{\varphi} $,
and the image by $\varphi$ of the $\mathfrak{M}_{f}$ torsion satisfies 
$ \left( \uprho^{\rm qt}[\mathfrak{M}_{f}]\right)^{  \varphi}= ( \uprho^{{\rm qt }})^{\varphi}  [\mathfrak{M} ^{\varphi}]   $. 
Thus, we have
\[  K_{\mathfrak{M}_{f} , 2}= \varphi \left( K_{\mathfrak{M}^{\varphi}_{f} , 1}\right) =H_{\mathcal{O}_{f}}({\sf Tr}( \uprho^{\rm qt}[\mathfrak{M}^{\varphi}_{f}])^{  \varphi} )
 = H_{\mathcal{O}_{f}}({\sf Tr}( (\uprho^{\rm qt} )^{  \varphi} [\mathfrak{M}_{f}]))
.\]
 Since $K_{\mathfrak{M}_{f}}$
  is the compositum of $K_{\mathfrak{M}_{f} , 1}$ and $K_{\mathfrak{M}_{f} , 2}$, this will lead immediately to the explicit description of $K_{\mathfrak{M}_{f}}$ announced in the Introduction.  

\section{Generation of Ray Class Fields I: The Case $d=2$}\label{ellipticsection}

In what follows, $f$ is a unit of norm 1, $N(f)=1$.  Let $K_{\mathfrak{M}_{f},1}$ be the component corresponding to $\infty_{1}$ of the unit ray class field $K_{\mathfrak{M}_{f}}$, which was defined in \S \ref{RayClassFieldSection}.
In this section we prove the theorem on the generation of $K_{\mathfrak{M}_{f},1}$
 in the case $d=2$.
 We remark that when $d=1$ there is nothing
to prove since everything reduces to Hayes theory, as $H_{\mathcal{O}_{K}}=H_{A_{\infty_{1}}}$ and $K_{\mathfrak{M}_{f},1}=K_{\mathfrak{m}_{f}}$.
If $d_{0}$ is the corresponding
order at $\infty_{1}$ of the fundamental unit $f_{0}$, so that $d=2=nd_{0}$, then $n=1$, i.e.,
$f=f_{0}$.  Thus $\mathcal{O}_{f}=\mathcal{O}_{K}$ and the unit ray class field is just the usual ray class field:
$K_{\mathfrak{M}_{f}} =K_{\mathfrak{M}} $.  
In particular, we are working with the family $\mathcal{M}_{f_{0}}$, which reduces to the (finite set of) moduli $\mathfrak{M}\subset \mathcal{O}_{K}$ for which $f_{0}\equiv 1\mod \mathfrak{M}$.
In view of this, the arguments in this section apply to the case $A_{f}= A_{\infty_{1}}$, and our notation will reflect this.

In what follows $\mathfrak{M}\subset\mathcal{O}_{K}$ is an integral ideal, $\mathfrak{m}=A_{\infty_{1}}\cap\mathfrak{M}$ the contraction to $A_{\infty_{1}}$.
Since $d=2$, the Galois group $Z$ is in bijection with the group $\{ 1=[\mathfrak{a}_{0}],[\mathfrak{a}_{1}]\}\subset {\sf Cl}_{A_{\infty_{1}}}$.
Let $\uprho_{i}:A_{\infty_{1}}\rightarrow H_{A_{\infty_{1}}}\{ \uptau\}$ be the Hayes module  associated to $\mathfrak{a}_{i}$, $i=0,1$.  Analytically,  it corresponds to the normalized lattice $\Uplambda_{i}=\upxi_{i}\mathfrak{a}_{i}$ where $\upxi_{i}\in \C_{\infty}$ is defined in \S \ref{QDMSection}. The exponential, denoted $e_{i}:\C_{\infty}\longrightarrow \C_{\infty}$, is defined
\[ e_{i}(z) = z\prod_{0\not=\uplambda\in\Uplambda_{i}}  \left(1- \frac{z}{\uplambda}\right) .\]

To the ideal $\mathfrak{m}$ we may associate the torsion modules $\uprho_{i}[\mathfrak{m}]$, each of which is equal to 
$e_{i}(\mathfrak{m}^{-1}\Uplambda_{i} \mod \Uplambda_{i} )$.    Thus we may write $t\in \uprho_{i}[\mathfrak{m}]$ as
$t=e_{i}(\upmu\upxi_{i})$, $\upmu\in \mathfrak{m}^{-1}\mathfrak{a}_{i}\mod \mathfrak{a}_{i}$.  By Hayes theory \cite{Hayes}, \cite{Thak}, 
$ K_{\mathfrak{m}}= H_{A_{\infty_{1}}}\left(\uprho_{i}[\mathfrak{m}] \right)$
for any $i$.  

The extension $K_{\mathfrak{m}}/K_{\mathfrak{M},1}$ is Galois, with Galois group $\cong Z$; the action of the non-trivial element $\mathfrak{a}_{1}\in Z$ induces one on torsion,
sending $\uprho_{0}[\mathfrak{m}]$ to $\uprho_{1}[\mathfrak{m}]$ via
$t\mapsto  \Upphi_{1}(t)$,
where $ \Upphi_{1}\in H_{A_{\infty_{1}}}\{ \uptau\}$ is the unique monic additive polynomial generating 
the {\it principal} left ideal in $H_{A_{\infty_{1}}}\{ \uptau \}$ generated by the polynomials $\uprho_{0,a}$, $a\in \mathfrak{a}_{1}$.
See \cite{Goss}.  In Lemma \ref{explicitrhocalc}  an explicit description of these polynomials was given.
Recall, \S \ref{RayClassFieldSection}, that 
$  {\sf Tr}( \uprho^{\rm qt}[\mathfrak{M}])\subset K_{\mathfrak{M},1}$ denotes the group of traces of $\mathfrak{M}$ quantum torsion points.

\begin{theo}\label{ellipticrayclasstheorem} Let $K$ be a real elliptic function field quadratic over $\F_{q}(T)$, let $\mathfrak{M}\subset\mathcal{O}_{K}$ be
an ideal.    Let $t\in K_{\mathfrak{m}}$ be a $W$ primitive, as guaranteed by {\rm (\ref{primitivitylemma})}.  Then 
\[ K_{\mathfrak{M},1} = H_{\mathcal{O}_{K}}({\sf Tr}_{Z}(t)) =H_{\mathcal{O}_{K}}(    {\sf Tr}( \uprho^{\rm qt}[\mathfrak{M}])  ).\]
\end{theo}

\begin{proof} The Galois group $W:= {\rm Gal}(K_{\mathfrak{m}}/H_{A_{\infty_{1}}})\cong
 {\rm Ker}({\sf Cl}^{\mathfrak{m}}_{A_{\infty_{1}}}\longrightarrow {\sf Cl}_{A_{\infty_{1}}} )$ by reciprocity,
which is in turn isomorphic to
$(A_{\infty_{1}}/\mathfrak{m})^{\times}$.  By definition of $K_{\mathfrak{M},1}$ (as the fixed field of $Z$ acting on $K_{\mathfrak{m}}$)  we have 
$W\cong {\rm Gal}(K_{\mathfrak{M},1}/H_{\mathcal{O}_{K}})$ as well.
 Let $t\in \uprho_{0}[\mathfrak{m}]$ be primitive for $K_{\mathfrak{m}}/H_{A_{\infty_{1}}}$, which exists 
 by (\ref{primitivitylemma}).  
Therefore, 
in order to prove the Theorem, we must
show that for all $1\not=\upsigma\in W$, $ {\sf Tr}_{Z}(t)^{\upsigma}\not= {\sf Tr}_{Z}(t)$.
Suppose that $\upsigma$ corresponds to 
$bA_{\infty_{1}} \in {\rm Ker} ( {\sf Cl}^{\mathfrak{m}}_{A_{\infty_{1}}}  \rightarrow {\sf Cl}_{A_{\infty_{1}}})\subset  {\sf Cl}^{\mathfrak{m}}_{A_{\infty_{1}}}$,
where $b$ is monic.
By the description of the ray class group action on torsion, we know that $\upsigma$ acts via the additive polynomial associated to the ideal $bA_{\infty_{1}}$
and since the latter is principal, and $b$ is monic, this will just be the action by $\uprho_{0,b}$.
Then since $Z$ and $W$ commute with each other,
\begin{align}\label{traceeqnd=2}   {\sf Tr}_{Z}(t)^{\upsigma}  & = \uprho_{0,b}\bigg(   e_{0}(\upmu \upxi_{0})   +   \Upphi_{1}(e_{0}(\upmu \upxi_{0})) \bigg) 
= e_{0}(b\upmu \upxi_{0}) +\Upphi_{1}(e_{0}(b\upmu \upxi_{0})).\end{align}
Note that in deriving (\ref{traceeqnd=2}), we used the fact that, since $b$ corresponds to a principal ideal in ${\sf Cl}^{\mathfrak{m}}_{A_{\infty_{1}}}$, its $\ast$ action 
on Hayes modules is trivial, and so it preserves each torsion group
$\uprho_{i}[\mathfrak{m}]$, $i=0,1$.  
Thus we must show that 
\begin{align}\label{TraceEquation} {\sf Tr}_{Z}(t)^{\upsigma} -{\sf Tr}_{Z}(t)= e_{0}((b-1)\upmu \upxi_{0} )  +\Upphi_{1}(e_{0}((b-1)\upmu\upxi_{0})) \not=0\end{align}
where the first equality is obtained using the additivity of $e_{0}$ and $ \Upphi_{1}$. 
By Lemma \ref{explicitrhocalc}, 
$  \Upphi_{1}(x) = x\prod_{c\in \F_{q}^{\times}}  \left(x-e_{0}(\upxi_{0}cfT )  \right) $. 
Since $t$ is primitive, $b\leftrightarrow\upsigma$ acts non-trivially on $t$ i.e. $e_{0}((b-1)\upmu \upxi_{0} )\not=0$.  Therefore we may divide
(\ref{TraceEquation})
out by this constant, 
so we are reduced to showing that 
\begin{align}\label{firstreduction} \prod_{c\in \F_{q}^{\times}}  e_{0}\big( \upxi_{0}( (b-1)\upmu -cfT)\big)  \not=-1.  \end{align}
Notice that if any of $(b-1)\upmu -cfT=0\mod (f)$, the LHS of (\ref{firstreduction}) is $0$ and we are done.  So we may assume
this is not the case.   In what follows we write $\exp_{\mathfrak{a}_{0}}(z)= e_{(f)}$.  Then for all $z$, 
\begin{align}\label{exp0vsexpf} e_{0}(\upxi_{0}z)=\upxi_{0}e_{(f)}(z)\end{align}
so that showing (\ref{firstreduction}) is equivalent to showing 
\begin{align}\label{unnormprod}  \prod_{c\in \F_{q}^{\times}}  e_{(f)}\big( (b-1)\upmu -cfT\big)  \not=-\upxi_{0}^{1-q}.  \end{align}
Now by Proposition \ref{transabsvalue}, we know that 
$ |\upxi_{0}|^{1-q} = q^{q^{2}-2} $.
We will then show that
\begin{align}\label{funduneq} \left|  \prod_{c\in \F_{q}^{\times}}  e_{(f)}\big( (b-1)\upmu -cfT\big)\right| \not=q^{q^{2}-2}.  \end{align}
We may choose $(b-1)\upmu \in \mathfrak{m}^{-1}(f)\mod (f)$ in its $(f)$-class so that $|(b-1)\upmu|<|f^{2}|=q^{4}$, since in $(f)$ there are elements of absolute value $q^{4},q^{5},\dots $.
Thus we may reduce our considerations to two cases:

\begin{case} $|(b-1)\upmu|<|fT|=q^{d+1}=q^{3}$.
\end{case}

For each $c\not=0$ we have 
\begin{align*}  \left|   e_{(f)}\big( (b-1)\upmu -cfT\big)\big)\right| & = |(b-1)\upmu-fT|\prod_{0\not=\upalpha\in (f)}   \left| 1-\frac{(b-1)\upmu-fT}{\upalpha}    \right|\\
& = |fT|\prod_{c'\in \F_{q}^{\times}} \left| 1-\frac{fT}{c'f} \right|  =q^{3}\cdot q^{q-1} =q^{q+2}.
\end{align*}
Therefore, the LHS of (\ref{funduneq}) is $q^{(q-1)(q+2)}\not=q^{q^{2}-2}$.

\begin{case} $|(b-1)\upmu|=|fT|$.
\end{case}

We know that for some $c$,  $(b-1)\upmu-cfT$ exhibits cancellation i.e.\ $|(b-1)\upmu-cfT|<|fT|$.
Without loss of generality we may assume that this happens for $c=1$,  i.e.\ 
$|(b-1)\upmu-fT|<|fT|=$ $|(b-1)\upmu-cfT|$, for all $c\not=0,1$ in $\F_{q}$.   Moreover,
we may choose $(b-1)\upmu$ in its $(f)$ class so that $|(b-1)\upmu-fT|<|f|=q^{2}$: this modification of $(b-1)\upmu$ mod $(f)$ does not effect the cancellation with $fT$.   
Then  
\begin{align}\label{expabs} |e_{(f)}((b-1)\upmu-fT)| & =  |(b-1)\upmu-fT|\prod_{0\not=\upalpha\in (f)}   \left| 1-\frac{(b-1)\upmu-fT}{\upalpha}    \right| \\
& = |(b-1)\upmu-fT|<q^{2}.   \nonumber 
\end{align}
On the other hand, for $c\not=0,1$ we have 
\begin{align*} |e_{(f)}((b-1)\upmu-cfT)| 
& = |(b-1)\upmu-cfT| \prod_{c'\in\F_{q}^{\times}}\left| 1-\frac{(b-1)\upmu-cfT}{c'f}   \right| 
 = q^{q+2}.
\end{align*}
There are $q-2$ such factors on the left-hand side of (\ref{funduneq}), so if (\ref{funduneq}) were an equality, we could solve:
$ | e_{(f)}\big( (b-1)\upmu -fT\big)|  =q^{-(q-2)(q+2)}\cdot q^{q^{2}-2} =q^{2} $,
which contradicts the result found in (\ref{expabs}).
\end{proof}


\section{Generation of Ray Class Fields II: The Case $d=3$}\label{genus2section}

We treat in this section the case  $d=3$ to introduce  arguments which will appear in a much more technical presentation in the next section for $d>3$.  
As in the previous section, the hypothesis $d=3$ implies $d=d_{0}$, i.e., $f=f_{0}$, and so once again, we have $A_{f}=A_{\infty_{1}}$ and $K_{\mathfrak{M}_{f},1} =K_{\mathfrak{M},1}$.   Thus
we are in the situation of $K$ a real, genus {\rm 2} function field quadratic over $\F_{q}(T)$, in which $f=f_{0}$ is a fundamental unit, for which $N(f)=1$ and $f\equiv 1  \mod \mathfrak{M}$.

\begin{theo}\label{genus2rayclasstheorem} Let $K$ be a real, genus {\rm 2} function field quadratic over $\F_{q}(T)$, let $\mathfrak{M}\subset\mathcal{O}_{K}$ be
an ideal with $f\equiv 1 \mod\mathfrak{M}$ and let $K_{\mathfrak{M},1}$ be the component of the ray class field associated to $\infty_{1}$. Then 
\[ K_{\mathfrak{M},1} = H_{\mathcal{O}_{K}}(    {\sf Tr}( \uprho^{\rm qt}[\mathfrak{M}])  ).\]
\end{theo}


\begin{proof}  
 It will be enough to show that for any  $ bA_{\infty_{1}} \in {\rm Ker} ( {\sf Cl}^{\mathfrak{m}}_{A_{\infty_{1}}}  \rightarrow {\sf Cl}_{A_{\infty_{1}}})$, with $b$ monic,  there exists $t\in \uprho_{0}[\mathfrak{m}]$ for which
\begin{align}\label{inequation1}  {\rm Tr}_{Z}(t)^{\upsigma}\not={\rm Tr}_{Z}(t)  ,\end{align}
where $\upsigma\in W\cong{\rm Gal}(K_{\mathfrak{M},1}/H_{\mathcal{O}_{K}})$ is the element corresponding to $ bA_{\infty_{1}} $.
If this is the case, then $H_{\mathcal{O}_{K}}( {\sf Tr}( \uprho^{\rm qt}[\mathfrak{M}]))\subset K_{\mathfrak{M},1}$ is not the fixed field of any nontrivial subgroup of 
$W$, which means it coincides with
$K_{\mathfrak{M},1}$.
In the case under consideration,
${\sf Tr}_{Z}(t) = t+  \Upphi_{2}(t) + \Upphi_{1}(t)$,
so writing $t=e_{0}(\upmu\upxi_{0})$, (\ref{inequation1}) may be written  as
\begin{align}\label{inequation2}
0 & \not= & 
e_{0}((b-1)\upmu \upxi_{0} )  + \Upphi_{2}(e_{0}((b-1)\upmu\upxi_{0})) + \Upphi_{1}(e_{0}((b-1)\upmu\upxi_{0})) \\
 & = & e_{0}((b-1)\upmu \upxi_{0} )\left\{ 1 + \prod_{c\in\F_{q}^{\times}}e_{0}\left(((b-1)\upmu -cfT ) \upxi_{0}\right) \right.  \nonumber \\
   &&  \left. + \prod_{\substack{ c',c''\in\F_{q} \\ (c',c'')\not=(0,0)}}e_{0}\left(((b-1)\upmu -c'fT -c''fT^{2}) \upxi_{0}\right) \right\} . \nonumber  
\end{align}
We begin first by taking $t\in K_{\mathfrak{m}}$ a $W$ primitive element, as guaranteed by (\ref{primitivitylemma}).
 By primitivity, $0\not = e_{0}((b-1)\upmu \upxi_{0} )$, so this factor may be canceled from either side of (\ref{inequation2}).  Using 
 $e_{0}(\upxi_{0}x) = \upxi_{0}e_{(f)}(x)$ (see (\ref{exp0vsexpf}) in Theorem \ref{ellipticrayclasstheorem}), it 
remains to show that \begin{align}\label{inequation3}
-1 & \not=  
\upxi_{0}^{q-1}\prod_{0\not=   c\in \F_{q}}  e_{(f)}((b-1)\upmu -cfT)\; \times \\ 
&  \left\{ 1+  \upxi_{0}^{q(q-1)}\prod_{\substack{ c',c''\in\F_{q} \\ c''\not=0}}  e_{(f)}((b-1)\upmu -c'fT-c''fT^{2})\right\} .  \nonumber
\end{align}
This will be true in all but one case, and in this exceptional event, we will modify $\upmu$ to $\tilde{\upmu}$ so that we still have $0\not = e_{0}((b-1)\tilde{\upmu} \upxi_{0} )$
and such that (\ref{inequation3}) holds.  The analysis breaks down according to the size of $|(b-1)\upmu|$.  Recall that by Proposition \ref{transabsvalue}, 
$ |\upxi_{0}|^{q-1}=q^{3-2q^{2}}$ and
$|\upxi_{0}|^{q(q-1)}=q^{3q-2q^{3}}$.

\begin{case2}\label{d=3firstcase}  $|(b-1)\upmu|<|fT|=q^{4}$.
\end{case2}

 For any $c\in\F_{q}^{\times}$
\begin{align}\label{fTexp} | e_{(f)}((b-1)\upmu -cfT) | = |fT|\prod_{0\not= c'\in\F_{q}} \left|1-\frac{fT}{c'f}\right| = q^{4}\cdot q^{q-1}=q^{q+3}\end{align}
and for any $c', c''\in\F_{q}$, $c''\not=0$, we have
\begin{align}\label{fT2exp} | e_{(f)}((b-1)\upmu -c'fT-c''fT^{2}) | & = |fT^{2}|\prod_{0\not= c'''\in\F_{q}} \left|1-\frac{fT^{2}}{c'''f}\right| 
=q^{2q+3}. \end{align}
The factor contained in braces in (\ref{inequation3}) has absolute value equal to 
\[ \left| \upxi_{0}^{q(q-1)}\prod_{\substack{ c',c''\in\F_{q} \\ c''\not=0}}  e_{(f)}((b-1)\upmu -c'fT-c''fT^{2})\right| = q^{3q-2q^{3}}\cdot q^{(2q+3)q(q-1)}  = q^{q^{2}},\]
where we neglect the summand $1$ since $q^{q^{2}}>1$.
Taking absolute values of either side of (\ref{inequation3}) then gives
$1\not= q^{3-2q^{2}}\cdot q^{(q+3)(q-1)}\cdot q^{q^{2}} = q^{2q}$,
which confirms (\ref{inequation3}).

\begin{case2}\label{d=3secondcase}  $|(b-1)\upmu|=|fT^{2}|=q^{5}$.
\end{case2}

We may write for $c_{0}'\not=0$,
$(b-1)\upmu = c_{0}'fT^{2} + c_{0}'' fT + z +$ lower,
where $|z|<|fT|$, however,  since we are working with torsion elements mod $(f)$, we may further assume that $|z|<|f|=q^{3}$.
Since we will argue using only the absolute values of either side of (\ref{inequation3}), 
the particular values of $c_{0}',c_{0}''$ are irrelevant, so we set $c_{0}'=c_{0}''=1$: that this presents no loss of generality will be clear to the reader after having seen the argument that follows.
Here we have 
\[ | e_{(f)}((b-1)\upmu -cfT) |= |fT^{2}|\prod_{0\not= c'\in\F_{q}} \left|1-\frac{fT^{2}}{c'f}\right| = q^{5}\cdot q^{2(q-1)}=q^{2q+3} .\]
We split the product occurring within the braces in (\ref{inequation3}) as 
\begin{align*}  
\upxi_{0}^{q(q-1)}\prod_{\substack{c',c''\in \F_{q} \\ c''\not=0,1}}e_{(f)}((b-1)\upmu -c'fT-c''fT^{2})  \times \\ 
\left( \prod_{1\not= c'\in\F_{q}} e_{(f)}\left( (1-c')fT + z +\text{lower}\right) \right)\times e_{(f)}(z+\text{lower}) .
\end{align*}
The exponentials in the first product have absolute value $q^{2q+3}$, as in (\ref{fT2exp}), those in the second have absolute value $q^{q+3}$, as
in (\ref{fTexp}). Thus the absolute value
of the above product is
\begin{align}\label{braceproduct} q^{3q-2q^{3}}\cdot q^{(2q+3)q(q-2)   }\cdot q^{(q+3)(q-1)}\cdot | e_{(f)}(z+\text{lower})| = q^{-(q+3)} \cdot | e_{(f)}(z+\text{lower}) |.\end{align}
On the other hand, on the right hand side of (\ref{inequation3}), the factor outside the braces has absolute value
$q^{3-2q^{2}}\cdot q^{(2q+3)(q-1)}=q^{q}>1$.  If the inequation affirmed in (\ref{inequation3}) were false,  the absolute value of the 
term in braces would have to be $q^{-q}$.  In order for this to happen, the product in the braces must have leading term $-1$, so that its absolute value is $1$.
But in view of (\ref{braceproduct}) this forces $ | e_{(f)}(z+\text{lower}) |=q^{q+3}$.  But the latter is impossible,  since $|z|<q^{3}=|f|$ implies
$ | e_{(f)}(z+\text{lower}) | = |z|<q^{3} < q^{q+3} $.

\begin{case2}  $|(b-1)\upmu|=|fT|=q^{4}$.
\end{case2}

We may write
$(b-1)\upmu = c_{0}'fT + y +$  lower,
where $|y|<|f|=q^{3}$, and without loss of generality, we take $c_{0}'=1$.
Taking absolute values of the RHS of (\ref{inequation3}) gives
\[| e_{(f)}(y+\text{lower}) |\cdot q^{3-2q^{2}}\cdot q^{(q+3)(q-2)}\cdot \max\{ 1, q^{3q-2q^{3}}\cdot q^{(2q+3)q(q-1) }\}=q^{q-3} | e_{(f)}(y+\text{lower}) |,\]
since $q^{3q-2q^{3}}\cdot q^{(2q+3)q(q-1)}=q^{q^{2}}>1$.
If $| e_{(f)}(y+\text{lower}) |=|y|\not=q^{3-q}$ we are done.  
So suppose otherwise and let us additionally assume that $q\not=2,3$.  Replace 
$\upmu$ by $\tilde{\upmu}=f\upmu$.  We have
$ (b-1)\tilde{\upmu} = f^{2}T+ fy + \text{lower}  = fy + \text{lower} \mod (f)$,
so $ (b-1)\tilde{\upmu}$ defines a non-trivial  $\mathfrak{m}$ torsion element $\mod (f)$, since it is congruent $\mod (f)$ to $fy +\text{lower}$.  Indeed, the latter
has absolute value $|fy|\leq q^{2}<|f|$, and so could
not give an element of $(f)$ (we stress that $fy\not=0$ since $y\not=0$). As
$ (b-1)\tilde{\upmu}$ is a non-trivial $\mathfrak{m}$ torsion element, it follows that $\tilde{\upmu}$ is as well.  In particular, 
 $\tilde{\upmu}$ defines $\tilde{t}:=e_{0}(\tilde{\upmu}\upxi_{0})\in \uprho_{0}[\mathfrak{m}] -\{ 0\}$ for which 
 $\tilde{t}^{(b-1)}=e_{0}((b-1)\tilde{\upmu}\upxi_{0})\not=0$.
We may now proceed as in {\it Case} \ref{d=3firstcase} to show (\ref{inequation3}) holds.
When $q=2$, we define instead $\tilde{\upmu}=fT\upmu$.  Then modulo $(f)$, 
$(b-1)\tilde{\upmu}\equiv fTy+\text{lower}$ and the latter has absolute value $q^{5}$ i.e.\
$  (b-1)\tilde{\upmu} = c'fT^{2} + \text{lower} \mod (f)$.
The elements $\tilde{\upmu}$ 
and $(b-1)\tilde{\upmu}$ are again nontrivial, and since $|fTy+ \text{lower} |=|fT^{2}|$, we may replace $(b-1)\tilde{\upmu}$ by
$fTy+ \text{lower}$ and  argue as in {\it Case} \ref{d=3secondcase} to conclude
 (\ref{inequation3}).
 When $q=3$, define instead $\tilde{\upmu}=fT^{2}\upmu$. Then modulo $(f)$, $| (b-1)\tilde{\upmu} |=q^{5}$, and again we  refer to the argument of 
  {\it Case} \ref{d=3secondcase}.
\end{proof}

\section{Generation of Ray Class Fields III: General Case}\label{generalcase}

In this section we prove the main theorem for $d>3$; the cases of $d\leq 3$ were treated in
\S\S \ref{ellipticsection}, \ref{genus2section}.  Unlike those sections, here we must take $f\in \mathcal{O}_{K}^{\times}$ a general and not necessarily  fundamental unit, and work
within the full generality of orders. 

The class fields
which will receive an explicit treatment will be those that occur in the family of unit ray class fields $\mathcal{K}_{\rm unit}$, which form a cofinal family in view of Theorem \ref{cofinalitytheo} of
\S  \ref{RayClassFieldSection}.   In particular,
in what follows, $\mathfrak{M}_{f}\subset \mathcal{O}_{f}$ is a modulus belonging to the family $\mathcal{M}_{f}$ described in (\ref{funitmoduli}).  Thus, $N(f)=1$ and $f\equiv 1\mod  \mathfrak{M}_{f}$.
By Theorem \ref{PrimenessToC} of \S  \ref{RayClassFieldSection}, $\mathfrak{M}_{f}$, as well as the contraction $\mathfrak{m}_{f} =\mathfrak{M}_{f} \cap A_{f}$, are invertible order ideals.
The associated unit ray class field $K_{\mathfrak{M}_{f}}$ may be written as a compositum  
$  K_{\mathfrak{M}_{f}} = K_{\mathfrak{M}_{f},1} \cdot K_{\mathfrak{M}_{f},2}$
where the component $ K_{\mathfrak{M}_{f},i}$ is the class field defined to be narrow only along $\infty_{i}$, $i=1,2$.  See equation (\ref{intermediateRCFs}) and Proposition \ref{intermediateRCFsProp} of 
 \S \ref{RayClassFieldSection}.



\begin{theo}\label{generalrayclasstheorem} Let $K$ be a function field quadratic and real over $\F_{q}(T)$, $f$ a unit with $N(f)=1$ and $\mathfrak{M}_{f}\subset \mathcal{O}_{f}$  
an ideal belonging to the family $ \mathcal{M}_{f}$.  Let $K_{\mathfrak{M}_{f},1}$ be the component of the unit ray class field associated to $\infty_{1}$.  Then 
\[ K_{\mathfrak{M}_{f},1} = H_{\mathcal{O}_{f}}(    {\sf Tr}( \uprho^{\rm qt}[\mathfrak{M}_{f}])  ).\]
\end{theo}

\begin{proof} As before, $\mathfrak{m}_{f} =\mathfrak{M}_{f} \cap A_{f}$.   Let $1\not=\upsigma\in W\cong {\rm Gal}(K_{\mathfrak{M}_{f},1}/H_{\mathcal{O}_{f}})$
correspond to $ bA_{f} \in {\rm Ker} ( {\sf Cl}_{\mathfrak{m}_{f}} \rightarrow {\sf Cl}_{A_{f}})$, with $b$ monic.
As in the proof of Theorem \ref{genus2rayclasstheorem}, we must find 
$t=e_{0}(\upxi_{0} \upmu)\in \uprho_{0}[\mathfrak{m}_{f}]$, $\upmu\in (\mathfrak{m}_{f}(f)^{-1})^{\ast}\mod (f) =\mathfrak{m}_{f}^{-1} (f)\mod (f)$, such that 
${\rm Tr}_{Z}(t)^{\upsigma}\not={\rm Tr}_{Z}(t) $,
or equivalently, writing $t^{\upsigma-{\rm id}}:=t^{\upsigma}-t$,
\begin{align}\label{EquivalentFormulation} 0\not= {\rm Tr}_{Z}(t^{\upsigma-{\rm id}})=t^{\upsigma-{\rm id}}+\Upphi_{d-1}(t^{\upsigma-{\rm id}})+\cdots
+\Upphi_{1}(t^{\upsigma-{\rm id}}) .\end{align}

Assume first that $t=e_{0}(\upxi_{0}\upmu )$, where $\upmu$ is a generator of the analytic torsion module $\mathfrak{m}_{f}^{-1}(f) \mod (f)$: that
this module is rank 1 follows from Proposition 2.3 and  Note 2.2 of \cite{DGVI}.  Note that, mod $(f)$,  $\upmu\not\in (b-1)^{-1}=$ the reciprocal ideal of $(b-1)$.  Otherwise, we
would have the inclusion of $A_{f}$ modules
$ \mathfrak{m}_{f}^{-1} (f) \mod (f)$ $\subset$ $ (b-1)^{-1}  \mod (f)$,
or in other words, $\mathfrak{m}_{f}^{-1}(f)\subset  (b-1)^{-1} $.  But this implies $\mathfrak{m}_{f}\supset (b-1)(f)$, and since $\mathfrak{m}_{f}$ and $(f)$ are relatively prime by Lemma \ref{coprimalityofideals}, we must have
$b-1\in\mathfrak{m}_{f}$, contradicting our choice of $b$.
Since $t^{\upsigma-{\rm id}}\not=0$, we may divide out (\ref{EquivalentFormulation}) by it.  In view of the explicit description of the additive polynomials $ \Upphi_{i}$ given in Lemma \ref{explicitrhocalc},
(\ref{EquivalentFormulation}) is equivalent to
\begin{align*}
-1\not= & \prod_{c_{1}\in\F_{q}^{\times}}e_{0}\left(\upxi_{0} ((b-1)\upmu -c_{1}fT )\right)\left\{1+ \prod_{\substack{ c_{1},c_{2}\in \F_{q} \\ 0\not= c_{2}}}e_{0}\left(\upxi_{0} ((b-1)\upmu -c_{1}fT-c_{2}fT^{2} )\right)\right.\times \\ 
& \left.\left\{1+\cdots 
     \left\{ 1+\prod_{\substack{ c_{1},\dots ,c_{d-1}\in\F_{q} \\ 0\not= c_{d-1}}}e_{0}\left(\upxi_{0} ((b-1)\upmu -c_{1}fT-\dots -c_{d-1}fT^{d-1} )\right)  \right\}\cdots\right\}\right\}.
\end{align*}
Rewriting using the identity $e_{0}(\upxi_{0}x) = \upxi_{0}e_{(f)}(x)$, this in turn amounts to showing
\begin{align*}
-1\not= &\upxi_{0}^{q-1} \prod_{c_{1}\in\F_{q}^{\times}}e_{(f)}\left((b-1)\upmu -c_{1}fT \right)\left\{1+ \upxi_{0}^{(q-1)q}\prod_{\substack{ c_{1},c_{2}\in \F_{q} \\ 0\not= c_{2}}}e_{(f)}\left( (b-1)\upmu -c_{1}fT-c_{2}fT^{2} \right)\right. \times\\ 
& \left\{1+\cdots 
 \left.    \left\{ 1+\upxi_{0}^{(q-1)q^{d-2}}\prod_{\substack{ c_{1},\dots ,c_{d-1}\in\F_{q} \\ 0\not= c_{d-1}}}e_{(f)}\left((b-1)\upmu -c_{1}fT-\dots -c_{d-1}fT^{d-1} \right) \right\}\cdots\right\}\right\}.
\end{align*}
If we denote 
\[  {\tt P}_{j}=\upxi_{0}^{(q-1)q^{j-1}}\prod_{\substack{ c_{1},\dots ,c_{j}\in\F_{q} \\ 0\not= c_{j}}}e_{(f)}\left((b-1)\upmu -c_{1}fT-\dots -c_{j}fT^{j} \right) 
\]
then we may write the above inequation more concisely as
\begin{equation}\tag{$\maltese$}  -1\not= {\tt P}_{1}\left\{ 1+{\tt P}_{2}\left\{ 1+ \cdots \left\{1+ {\tt P}_{d-1}\right\}\cdots\right\}\right\} .
\end{equation}
Recall, by Proposition \ref{transabsvalue}, that $|\upxi_{0}|^{q-1}=q^{d-(d-1)q^{2}}$.  As in the proof of Theorem \ref{genus2rayclasstheorem} we proceed according to the size of $|(b-1)\upmu |$.
Note that we assume  
$ |(b-1)\upmu|\not=|f|=q^{d} $ and $ |(b-1)\upmu|<|f|^{2}=q^{2d}$ since $(b-1)\upmu$ is taken modulo $(f)={\rm Ker}(e_{(f)})$.  

\begin{case3}\label{Case31} $|(b-1)\upmu |<q^{d+1}=|fT|$.
\end{case3}

Here there is no cancellation in the arguments of the exponentials appearing in ($\maltese$).  Using the product formula for the exponential, we obtain
\begin{align*}
\begin{array}{lclcl}
 |e_{(f)}\left((b-1)\upmu -c_{1}fT\right)| & = &  q^{d+1}\cdot q^{q-1} & = &q^{d+q}, \\
  |e_{(f)}\left( (b-1)\upmu -c_{1}fT-c_{2}fT^{2} \right)| & = & q^{d+2}\cdot q^{2(q-1)} & = & q^{d+2q},\\
 & \vdots  &  &   \vdots\\
  |e_{(f)}\left((b-1)\upmu -c_{1}fT-\dots -c_{d-1}fT^{d-1} \right)| & =&  q^{d+d-1}\cdot q^{(d-1)(q-1)}& = & q^{d+(d-1)q}.
  \end{array}
  \end{align*}
  We calculate absolute values starting from the innermost term $1+{\tt P}_{d-1}$ of ($\maltese$) and then radiate outwards.  There are $(q-1)q^{d-2}$ exponential factors in ${\tt P}_{d-1}$, so 
   \begin{align}\label{onion1}  |{\tt P}_{d-1}|=  q^{(d-(d-1)q^{2})q^{d-2}} \cdot q^{(d+(d-1)q)(q-1)q^{d-2}} =q^{q^{d-1}} = |1+{\tt P}_{d-1}|>1.\end{align}
   Continuing to the next set of braces, $\{ 1+ {\tt P}_{d-2}\{ 1+{\tt P}_{d-1}\}\}$, we see that 
   \begin{align}\label{onion2} |{\tt P}_{d-2}\{ 1+{\tt P}_{d-1}\}| & =q^{(d-(d-1)q^{2})q^{d-3}} \cdot q^{(d+(d-2)q)(q-1)q^{d-3}}   \cdot q^{q^{d-1}}  \\ \nonumber
   & = q^{2q^{d-2}} =|\{ 1+ {\tt P}_{d-2}\{ 1+{\tt P}_{d-1}\}\}| >1 .\end{align}
   Inductively, the RHS of ($\maltese$) has absolute value $q^{(d-1)q}\not=1$, which verifies ($\maltese$).
   
   \begin{case3}\label{Case32} $|(b-1)\upmu |=|fT^{d-1}| =q^{2d-1}$
   \end{case3}

We may assume  
$  (b-1)\upmu = c_{d-1}'fT^{d-1} +c_{d-2}'fT^{d-2}+\cdots +c_{1}'fT + z +$ lower, 
  where $|z|\leq |f|$, however since we are considering torsion elements $\mod (f)$, we may in fact assume $|z|<|f|=q^{d}$.
  We first assume $z+\text{lower}\not\equiv 0\mod (f)$.
  Without loss of generality, we may assume that $c_{d-1}'=\cdots =c_{1}'=1$ (as we did in {\it Case} \ref{d=3secondcase} of the proof of Theorem \ref{genus2rayclasstheorem}).
 Here, the arguments of the exponentials appearing in ${\tt P}_{d-1}$ may present cancellation.  Therefore, to calculate its absolute value, we factorize
 \begin{align*} {\tt P}_{d-1} & = &\upxi_{0}^{(q-1)q^{d-2}}\prod_{\substack{c_{1},\dots ,c_{d-1}\in \F_{q} \\   0,1\not= c_{d-1}}}e_{(f)}\left(  (1-c_{d-1})fT^{d-1} +\dots + (1-c_{1})fT+z+\text{lower}\right) \times \\
& & \prod_{\substack{c_{1},\dots ,c_{d-2}\in\F_{q}\\  1\not= c_{d-2}}}e_{(f)}\left(  (1-c_{d-2})fT^{d-2} +\dots + (1-c_{1})fT+z+\text{lower}\right) \times \\
& & \vdots\; \\
& & \prod_{ 1\not= c_{1}\in\F_{q}}e_{(f)}\left((1-c_{1})fT +z+\text{lower} \right) \times \\
& & e_{f}\left(z+\text{lower} \right).
 \end{align*}
Calculating absolute values of exponentials exactly as we did in {\it Case} \ref{Case31}, and using the fact that
$ |e_{(f)}\left(z+\text{lower} \right)|=|z|$ since $|z|<|f|$,  we then have 
 \begin{align*}  |{\tt P}_{d-1}| & = &  q^{(d-(d-1)q^{2})q^{d-2}}\cdot q^{(d+(d-1)q)(q-2)q^{d-2}}\cdot q^{(d+(d-2)q)(q-1)q^{d-3}}\cdots \cdot q^{(d+q)(q-1)} |z|.
 \end{align*}
 The final exponent of $q$ in the line above is calculated as follows: adding and subtracting $(d+(d-1)q)q^{d-2}$ to the sum of the exponents gives
 \begin{align*}
-(d-1)(q^{d}+q^{d-1}) +
(q-1)\bigg\{  (d+(d-1)q)q^{d-2} +(d+(d-2)q)q^{d-3}+\cdots + (d+2q)q+(d+q) \bigg\} & = \\
-(d-1)(q^{d}+q^{d-1})+(q-1)\bigg\{    (d-1)q^{d-1} +dq^{d-2}+(d-2)q^{d-2}+dq^{d-3}+\cdots + 2q^{2}+ dq+q+d \bigg\} &= \\
-(d-1)(q^{d}+q^{d-1})+(q-1) \bigg\{ (d-1)q^{d-1}   +(2d-2)q^{d-2}+(2d-3)q^{d-3} + \cdots  +(d+1)q +d\bigg\}  &=\\
-(2d-2)q^{d-2}+(q-1)\bigg\{ (2d-3)q^{d-3} + \cdots +(d+1)q +d  \bigg\}  &=\\ 
-q^{d-2}-q^{d-3}-\cdots -q-d .\;\; &
 \end{align*}
 We conclude that
$ |{\tt P}_{d-1}| = q^{-q^{d-2}-q^{d-3}-\cdots -q-d}|z|$.
  On the other hand,  
 $ |{\tt P}_{1}| =q^{d-(d-1)q^{2}}\cdot q^{(d+(d-1)q)(q-1)}=q^{q}>1$.
 So for ($\maltese$) to be false,  
 $ |1+{\tt P_{2}}\{ 1+ \cdots \}\cdots \}|=$ $q^{-q}<1$.
 For this, we must have $|{\tt P_{2}}\{ 1+ \cdots \}\cdots \}|=1$.
 But 
 $|{\tt P_{2}}| $ $= q^{(d-(d-1)q^{2})q}\cdot q^{(d+(d-1)q)q(q-1)} $ $ =q^{q^{2}}>1$.
 Inductively, we are led to $|{\tt P_{d-2}}|=q^{q^{d-2}}$ and hence we must have $|{\tt P_{d-1}}|=1$, which in light of the above, implies that (since $d\geq 3$)
$ |z|   =$ $q^{q^{d-2}+q^{d-3}+\cdots +q+d}$ $\geq q^{d}>|z|$, 
 contradiction.
  If $z+\text{lower}\equiv 0$, then the above argument shows ${\tt P}_{d-1}=0$ hence $|1+{\tt P}_{d-1}|=1$, which does not
 contradict ($\maltese$).
 
 \begin{case3}\label{Case33} $|(b-1)\upmu |=|fT| =q^{d+1}$.
 \end{case3}
 
 Here  $(b-1)\upmu = cfT + y + \text{lower}$; without loss of generality, take $c=1$.   Assume first that $y+\text{lower}\not\equiv 0\mod (f)$, $|y|<|f|$.
Exactly as in {\it Case} \ref{Case31},  $|1+{\tt P}_{d-1}|=q^{q^{d-1}}$ and inductively $|1+{\tt P_{2}}\{ 1+ \cdots \}\cdots \}|=q^{(d-2)q^{2}}$.
 Thus if ($\maltese$) were false, 
 \begin{align*}
 \begin{array}{lllll}
 1 & = & |{\tt P}_{1}|q^{(d-2)q^{2}} & = & q^{d-(d-1)q^{2}} \cdot q^{(d+q)(q-2)}\cdot q^{(d-2)q^{2}} \cdot |e_{(f)}(y+\text{lower})|\\
 &&  & =  &q^{(d-2)q-d}\cdot |y|.
  \end{array}
 \end{align*}
 Note here that we can now immediately dispense with the case $y+\text{lower}= 0$, which is in conflict with the above.
 So we get equality precisely when $ |y|=q^{d-(d-2)q}$.  If we do not, we are done.  Otherwise, replace $\upmu$ by $\tilde{\upmu}=f\upmu$.  Then
 $\mod (f)$ we have
 \[ (b-1)\tilde{\upmu} \equiv fy +\text{lower} \mod (f) , \quad |fy|=q^{2d-(d-2)q}.\]  
 Replace $(b-1)\tilde{\upmu}$ by the $(f)$ equivalent $ fy +\text{lower}$.
 Since we are assuming $d>3$ in this section,  $ |fy|<|f|=q^{d}$ and non zero except when
 $d=4$ and $q=2$. Indeed, the inequality $d/(d-2)<q$ for all $d\geq 5$ implies $2d-(d-2)q<d$.
  Leaving aside the case $d=4$ and $q=2$ for the moment,  $\tilde{\upmu}$ is non-trivial, since it is moved to a nontrivial element by
  $b-1$.  In particular, we may divide out the analogue of the expression (\ref{EquivalentFormulation})
  obtained by replacing $\upmu$ by $\tilde{\upmu}$ by $e_{0}(\upxi_{0}(b-1)\tilde{\upmu})$ (as we did in the beginning of the proof of this Theorem). Therefore we may repeat all of the above arguments replacing $\upmu$ by $\tilde{\upmu}$, where the estimate $|(b-1)\tilde{\upmu}|<q^{d}$ puts us in the setting of {\it Case} \ref{Case31}.
   If $d=4$, $q=2$,  then since $|y|=1$,  
   \[ (b-1)\tilde{\upmu} =fy +\text{lower} \equiv z +\text{lower} \mod (f), |z|<|f|.\]
   If $z+\text{lower}\not\equiv 0\mod (f)$, we may proceed as in {\it Case} \ref{Case31}.  Otherwise, since $|z|<|f|$, we must have $z+\text{lower}=0$ in which case
 $ (b-1)\tilde{\upmu} = f(b-1)\upmu \equiv 0\mod (f)$.
   This in turn implies $(b-1)\upmu$ is an $(f)$ torsion point i.e.\ $(b-1)\upmu\in A_{\infty_{1}}$ or $\upmu\in (b-1)^{-1}$.  But we have already shown at the beginning
   of the proof of this Theorem that this cannot occur for $\upmu$ a generator.

 
  \begin{case3} $|(b-1)\upmu |=|fT^{j}| =q^{d+j}$, $1<j<d-1$.
 \end{case3}
 
 Assume again $(b-1)\upmu= fT^{j}+\cdots +fT+z+\text{lower}$.
 First note that 
 \[ |{\tt P}_{1}| = q^{d-(d-1)q^{2}} \cdot q^{(d+jq)(q-1)}=q^{-q((d-1-j)q-(d-j) )}\]
 which is $<1$ provided $(d-1-j)q-(d-j)>0$.  The latter is always true except when $j=d-2$ and $q=2$: indeed for $j<d-2$,
 \[   \frac{d-j}{d-1-j} < 
 2 \leq q, \]
 and when $j=d-2$ one checks by hand that the inequality $(d-1-j)q-(d-j)>0$ holds provided $q>2$.
   The case $j=d-2$ and $q=2$ will be dealt with later
 so we assume first that when $j=d-2$, $q>2$.  Note that more generally, for $k<j$
 \begin{align}\label{Pkcalc}   |{\tt P}_{k}| = q^{(d-(d-1)q^{2})q^{k-1}} \cdot q^{(d+jq)(q-1)q^{k-1}} =q^{-q^{k}((d-1-j)q-(d-j) )} ,\end{align}
 also $<1$ by our present hypothesis. 
  Then in order to contradict ($\maltese$), we require that 
$ | 1+{\tt P}_{2}\{ 1+ \cdots \{ 1+{\tt P}_{d-1}\} \cdots\}  |= q^{q((d-1-j)q-(d-j) )} >1$,
 which implies that
$ | {\tt P}_{2}\{ 1+ \cdots \{ 1+{\tt P}_{d-1}\} \cdots\}  |= q^{q((d-1-j)q-(d-j) )}$.
  Inductively (and using the calculation (\ref{Pkcalc})), we see that in order to contradict ($\maltese$) we must have
 \begin{align*}  | 1+{\tt P}_{j}\{ 1+ {\tt P}_{j+1}\{ \cdots \{ 1+{\tt P}_{d-1}\} \cdots\}\}  | & = | {\tt P}_{j}\{ 1+ {\tt P}_{j+1}\{ \cdots \{ 1+{\tt P}_{d-1}\} \cdots\}\}   | \\ 
 &  = q^{(q+\cdots +q^{j-1})((d-1-j)q-(d-j) )} >1.\end{align*}
  Starting with $1+{\tt P}_{d-1}$ and moving outwards, we calculate (recall the analogous calculations made in (\ref{onion1}), (\ref{onion2}) in {\it Case} \ref{Case31})
$   | 1+{\tt P}_{j+1} \{ 1+ \cdots \{ 1+{\tt P}_{d-1}\} \cdots\} | =q^{(d-(j+1))q^{j+1}}$.
To calculate $|{\tt P}_{j}|$, we factor ${\tt P}_{j}$ exactly as we did ${\tt P}_{d-1}$  in {\it Case} \ref{Case32}:
  \begin{align*} {\tt P}_{j} & = &\upxi_{0}^{(q-1)q^{j-1}}\prod_{\substack{c_{1},\dots ,c_{j} \in \F_{q}\\   0,1\not= c_{j}}}e_{(f)}\left(  (1-c_{j})fT^{j} +\dots + (1-c_{1})fT+z+\text{lower}\right) \times \\
& & \prod_{\substack{c_{1},\dots ,c_{j-1} \in \F_{q}\\   1\not= c_{j-1}}}e_{(f)}\left(  (1-c_{j-1})fT^{j-1} +\dots + (1-c_{1})fT+z+\text{lower}\right) \times \\
& & \vdots\; \\
& & \prod_{ 1\not= c_{1}\in\F_{q}}e_{(f)}\left((1-c_{1})fT +z+\text{lower} \right) \times \\
& & e_{(f)}\left(z+\text{lower} \right)
 \end{align*}
 So in analogy with that previous computation we obtain
 \begin{align}\label{exceptionalP} |{\tt P}_{j}|=q^{-q^{j-1}-\cdots -q-d} |e_{(f)}\left(z+\text{lower} \right)| =q^{-q^{j-1}-\cdots -q-d}|z|.\end{align}
 Putting everything together, we conclude that if the inequation in ($\maltese$) is false, 
 \begin{align*}
 |z| & = & q^{-(d-1-j)q^{j+1} + ((d-1-j)q-(d-j))(q+\cdots + q^{j-1}) + q^{j-1}+\cdots + q +d} \\
 & = & q^{(d-1-j)q(-q^{j}+q^{j-1}-1)+d} .
 \end{align*}
 This implies that $z\not=0$.  Now replace $\upmu$ by $\tilde{\upmu}=f\upmu$ which gives 
 $(b-1)\tilde{\upmu}\equiv fz + \text{lower}\mod (f)$ and 
$  |fz+\text{lower}| = q^{2d-(d-1-j)q(q^{j}-q^{j-1}+1)}$.
 By Lemma \ref{exponentlemma} below, we are done since $|fz+\text{lower}|<1<|f|$, and this can be handled using the argument of {\it Case} \ref{Case31}.
What remains is the case $j=d-2$ and $q=2$.  Here $|{\tt P}_{1}|=1$ which implies we would need 
$ 1+{\tt P}_{2}\{ 1+ \cdots \{ 1+{\tt P}_{d-1}\} \cdots\}  | =1$ 
in order to contradict ($\maltese$).  But the latter 
implies either $|{\tt P}_{2}\{ 1+ \cdots \{ 1+{\tt P}_{d-1}\} \cdots\}  | <1$
or 
${\tt P}_{2}\{ 1+ \cdots \{ 1+{\tt P}_{d-1}\} \cdots\}  =c+\text{lower}$, $-1\not= c\in\F_{q}$.
But as $q=2$, $-1=1$ and so the latter is not possible.  Since in addition, by (\ref{Pkcalc}), $|{\tt P}_{2}|=1$, this implies 
${\tt P}_{3}\{ 1+ \cdots \{ 1+{\tt P}_{d-1}\} \cdots\}  =1 +$ lower. 
Inductively, ${\tt P}_{d-2}\{ 1+ {\tt P}_{d-1}\}  =1 +$ lower.
Therefore, since $|{\tt P}_{d-1}+1|=2^{2^{d-1}}$ (see (\ref{onion1})),
 $ |{\tt P}_{d-2}| =2^{-2^{d-1}}$.
 By (\ref{exceptionalP}),  
 $ |{\tt P}_{d-2}| =2^{-2^{d-3}-\cdots -2-d}|z|  $
 or 
 $ |z| = 2^{-2^{d-1}    +2^{d-3}+\cdots +2+d}$,
 which can be seen to be $<1$ for $d\geq 4$.  Therefore if we replace $\upmu$ by $\tilde{\upmu}=f\upmu$, this case is taken care of as well
 since $(b-1)\tilde{\upmu}\equiv fz+\text{lower}\mod (f)$ and $|fz+\text{lower}|<|f|<|fT|$, and this can be handled via the argument in {\it Case} \ref{Case31}.
 \end{proof}
 
  \begin{lemm}\label{exponentlemma} Let $d\geq 4$, $2\leq j\leq d-2$ and if $j=d-2$, $q>2$.  Then
  \[ q^{2d-(d-1-j)q(q^{j}-q^{j-1}+1)}<1.\]
 \end{lemm}
 
 \begin{proof} First consider the inequality for $j=2$. 
 If $d=4$, $j=2=d-2$ so we may assume here $q>2$.  Then
 $2d=8< 1\cdot 3(3^{2}-3+1)\leq 1\cdot q (q^{2}-q +1)$.
 Notice that there are no more cases of $j$ to consider for $d=4$.
 If $d=5$, 
$2d= 10< 2\cdot 2(2^{2}-2+1)\leq 2\cdot q (q^{2}-q +1)$.
Otherwise, If $d>5$, 
$2d/(d-3)< 5<2(2^{2}-2+1)\leq  q (q^{2}-q +1) $.
 Now consider the inequality for $j=d-2$ and $q>2$.  We may assume $d\geq 5$ here.
When $d=5$, 
 $  2d= 10< 1\cdot 3(3^{3}-3^{2}+1) \leq 1\cdot q (q^{3}-q^{2} +1)$. 
 When $d>5$, 
 $2d < 3\cdot (3^{d-2}-3^{d-3}+1)=3\cdot (3^{d-3}\cdot 2 +1)\leq q(q^{d-2}-q^{d-3} +1) $.
 Notice by the above calculations, the cases $d=4,5$ are finished.
Now consider the general case.  
The statement of the Lemma is a clear consequence of $2d<(d-1-j)(2^{j+1}-2^{j}+2)$, which is in turn implied by $2d<(d-j-1)(2^{j+1}-2^{j})$, or equivalently,
$d<(d-j-1)(2^{j}-2^{j-1})$.
Let us call $f(j):=(d-1-j)(2^{j}-2^{j-1})-d$: thus we want to show $f(j)>0$ on $[3,d-3]$.  When $d=6$ this interval is just the point $3$, and $f(3)>0$.
Now $f'(j)=-(2^{j}-2^{j-1})+(d-j-1)(2^{j}-2^{j-1})\log{2}$.
Note that in the chosen interval $2^{j}-2^{j-1}>0$. 
Therefore for integral values of $j$,
\begin{align}\label{derivativeineq} f'(j)\geq 0 \iff d-j-1\geq 1/\log{2}\approx 3.3219 \iff j\leq d-5.\end{align}
If $d=7$, (\ref{derivativeineq}) is false in $[3,4]$ and hence $f(j)$ in decreasing there, so its enough
to check $f(4)>0$, which is true.  
Now 
for $d\geq 8$, $f(j)$ has a critical point in the interval $[3,d-3]$.  Since the sign of $f'(j)$ changes from positive to negative as we pass through the critical point from the left to the right, it follows that
the minimum occurs at one of the endpoints of the interval $[3,d-3]$.  For $j=3$,  $f(3)= (d-4)\cdot 4-d>0$, and
for $j=d-3$,  $f(d-3)= 2\cdot (2^{d-3}-2^{d-4} )=2^{d-3}>0$ as well.

 \end{proof}
 
 \section{Conjecture in Characteristic Zero}\label{CharZeroConj}
 
 In this section we present a conjectural adaptation of the ideas developed here to number fields, which may lead to an analog of the Main Theorem in this setting.

 Let $K/\Q$ be a real quadratic extension, $\uptheta_{0}\in \mathcal{O}_{K}^{\times}$ a fundamental unit.
For $\upvarepsilon >0$, the set of $\upvarepsilon$-diophantine approximations \cite{Ge-C} is
$\Uplambda_{\upvarepsilon}(\uptheta_{0} ) = \{ n\in\Z \;:\;\; \| n\uptheta_{0}\| <\upvarepsilon \} 
$. 
The $\upvarepsilon$ sine function of $\uptheta_{0}$ is
\[   \sin_{\Uplambda_{\upvarepsilon}(\uptheta_{0} ) }(z) = z\prod_{0<n\in \Uplambda_{\upvarepsilon}(\uptheta_{0} )} \left( 1 - \frac{z^{2}}{n^{2}} \right)
 , \]
an analytic function of $z\in\C$.
Writing $  \cos_{\Uplambda_{\upvarepsilon}(\uptheta_{0} ) }(z)  = \frac{d \sin_{\Uplambda_{\upvarepsilon}(\uptheta_{0} ) }(z)}{dz}$ ,
the $\upvarepsilon$-exponential is then
$  \exp_{\Uplambda_{\upvarepsilon}(\uptheta_{0} ) }(z)  =  \cos_{\Uplambda_{\upvarepsilon}(\uptheta_{0} ) }(z)  +i\sin_{\Uplambda_{\upvarepsilon}(\uptheta_{0}) }(z)  $.
As in \S \ref{QDMSection} of this paper, 
$ \lim_{\upvarepsilon\rightarrow 0 }  \exp_{\Uplambda_{\upvarepsilon}(\uptheta_{0} ) }(z) = z$,
so a renormalization scheme is called for.  

Using as a template Wallis' formula for $\uppi$,
we define 
\[   \uppi_{\uptheta_{0}, \upvarepsilon} = \frac{1}{\upbeta_{1}}   \prod_{i=1}^{\infty} \frac{n_{i}^{2}}{\upbeta_{i} \upbeta_{i+1}} ,\]
where
$  \Uplambda_{\upvarepsilon}^{+}(\uptheta_{0} ) = \{ 0<n_{1}<n_{2}<\cdots \} \subset \Uplambda_{\upvarepsilon}(\uptheta_{0} ) $
are the positive zeros of $ \sin_{\Uplambda_{\upvarepsilon}(\uptheta_{0})}(z) $ and
$    0< \upbeta_{1} <\upbeta_{2}<\cdots $
are the positive zeros of $ \cos_{\Uplambda_{\upvarepsilon}(\uptheta_{0})}(z) $.  
We now normalize $\Uplambda_{\upvarepsilon}(\uptheta_{0} )$, as was done in positive characteristic in \S  \ref{QDMSection}:  
$ \breve{\Uplambda}_{\upvarepsilon}(\uptheta_{0} ):=  \uppi_{\uptheta, \upvarepsilon}  \Uplambda_{\upvarepsilon}(\uptheta_{0} ) $.
We denote by 
\[ s_{\upvarepsilon}(z)  = \sin_{\breve{\Uplambda}_{\upvarepsilon}(\uptheta_{0} ) }(z) = \uppi_{\uptheta_{0}, \upvarepsilon} \sin_{\Uplambda_{\upvarepsilon}(\uptheta_{0} ) }( \uppi_{\uptheta_{0}, \upvarepsilon}^{-1}z),  \quad c_{\upvarepsilon}(z) =\frac{ds_{\upvarepsilon}(z)}{dz} =   \cos_{\Uplambda_{\upvarepsilon}(\uptheta_{0} ) }(\uppi_{\uptheta_{0}, \upvarepsilon}^{-1}z)   \]
the normalized sine and cosine functions.
The associated normalized exponential is then
$  e_{\upvarepsilon}(z) :=  c_{\upvarepsilon}(z)  + is_{\upvarepsilon}(z) 
$.  The quantum exponential is the multi-valued limit
\[  \exp^{\rm qt}= \exp_{\uptheta_{0}}^{\rm qt} :\C\multimap \C, \quad  \exp^{\rm qt}(z) = \lim_{\upvarepsilon\rightarrow 0} e_{\upvarepsilon}(z)  . \]
In Theorem \ref{valueseqtchar0} below, we will show that $ \exp^{\rm qt}$ is non-trivial.  

We now describe a characteristic zero analog of the notion of Drinfeld module.  The main reference for this material is \cite{GLL}.   It will be seen that the structure of ``Drinfeld module in characteristic zero'' lacks certain key features present
in the classical (positive characteristic) setting, so its theory is at the moment inchoate.

Let $K/\Q$ be a finite extension and fix an infinite place $\upsigma$.  Denote by $\K$ the completion of $K$ with respect to $\upsigma$ (i.e.\ $\K$ is either $\R$ or $\C$). Define
\begin{align*} A_{\upsigma}  & = \big\{ \upalpha\in \mathcal{O}_{K} \; : \;\; |\upsigma'(\upalpha )| \leq 1 \text{ for all infinite places $\upsigma'\not=\upsigma, \overline{\upsigma} $} \big\}   \\
&=\{ 0\} \cup \upmu_{K}\cup \pm \left\{ \text{Pisot-Vijayaraghavan and Salem numbers in }\mathcal{O}_{K} \right\} \subset \K,
\end{align*}
where $\upmu_{K}=$ $K$-roots of unity. Moreover
$A_{\upsigma}$ is a {\bf {\em quasicrystal ring}}: a quasicrystal which is a multiplicative monoid.   See \S 2 of \cite{GLL}.

We recall that a {\bf {\em quasicrystal}}  $ \mho\subset \K$ (see \cite{Meyer}, \cite{Moody}) is a Delaunay set -- a relatively dense, uniformly discrete subset of $\K$ --
which is almost closed with respect to the sum: there exists $F\subset\K$ a finite set such that
$\mho -\mho \subset \mho+F$. 
The family of quasicrystal rings $A_{\upsigma}$ is the exact analog of the family of ``small Dedekind domains'' used throughout function field arithmetic (such as the rings
$A_{\infty_{1}}$, $A_{\infty_{2}}$ of this paper) and which are defined as  rings of functions  regular outside
of a fixed point of a curve.

Let $K/\Q$ again be real quadratic and fix $\uptheta_{0}>0$ a fundamental unit.  Assume $K\subset \R$, denote the infinite places $\upsigma_{1}= {\rm id}$, 
$\upsigma_{2}$, and write $\upalpha'=\upsigma_{2}(\upalpha) $. We fix  $A_{\upsigma_{1}}$.   

In what follows we will need to replace $\R_{+}$ by a space which is locally a Cantor set.   Let
$\log_{\uptheta_{0}^{-1}} (\mathcal{O}_{K})= \{ \ell \in \R_{+}\; : \;\; \uptheta_{0}^{-\ell} \in \mathcal{O}_{K} \}  $,
a dense subset of $\R_{+}$.  
Remove each point $\ell\in \log_{\uptheta_{0}^{-1}} (\mathcal{O}_{K})$ from $\R_{+}$ and replace it by a pair of ordered points
$\ell_{-} < \ell_{+} $.
The resulting set 
\[ \hat{\R}_{+}= \bigg(   \R_{+}\setminus\log_{\uptheta_{0}^{-1}} (\mathcal{O}_{K}) \bigg) \bigcup_{\ell\in \log_{\uptheta_{0}^{-1}} (\mathcal{O}_{K})}\{  \ell_{-} < \ell_{+}\} \]
is topologized by the induced total order: it is locally compact, perfect and totally disconnected.  Denote 
\[ \widehat{[0,1)} :=\{ x\in \hat{\R}_{+}\; : \;\; 0\leq x <1\}. \]

By a {\bf {\em quasicrystal fractional ideal}} is meant a quasicrystal $\mathfrak{a}\subset\R$
 which is closed with respect to the multiplicative action of $A_{\upsigma_{1}}$; if $\mathfrak{a}\subset A_{\upsigma_{1}}$ it is called integral.
For each $x\in \hat{\R}_{+}$, we define 
an integral quasicrystal ideal
$\mathfrak{a}_{x}  \subset A_{\upsigma_{1}} $
as follows.   
\[
\mathfrak{a}_{x} =\left\{ 
\begin{array}{ll}
  \{ \upalpha \in A_{\upsigma_{1}} : \;\;   |\upalpha'|  \leq \uptheta^{-x} \} & \text{ if $x=\ell_{+}$ for some $\ell\in  \log_{\uptheta_{0}^{-1}} (\mathcal{O}_{K})$}\\
  \{ \upalpha \in A_{\upsigma_{1}} : \;\;   |\upalpha'|  < \uptheta^{-x} \} & \text{otherwise }\\
 \end{array}
\right.
\]
Observe that if $y<x$ then 
$ \mathfrak{a}_{x}\subsetneq \mathfrak{a}_{y} $.

Recall the cyclic group of ideals $Z = \{ \mathfrak{a}_{0}, \dots ,\mathfrak{a}_{d_{0}-1} \}$, $ \mathfrak{a}_{i}\subset A_{\infty_{1}} $,
defined in \S \ref{QDMSection},  (\ref{aidef}); it fits into the short exact sequence
$  1\rightarrow Z\hookrightarrow {\sf Cl}_{A_{\infty_{1}}} \twoheadrightarrow {\sf Cl}_{K}\rightarrow 1 $,
and thus 
$ Z\cong {\sf Gal}( H_{A_{\infty_{1}}}/ H_{\mathcal{O}_{K}} ) $. 
The quasicrystal analog of $Z$ is 
\[  \hat{Z} = \{ \mathfrak{a}_{x} \; : \;\; x\in \widehat{[0,1)}\}  .\]
It represents a submonoid of the {\it monoid} of quasicrystal ideal classes ${\sf Cl}_{A_{\upsigma_{1}}}$, see \S 3 of \cite{GLL}.  There is an exact sequence of monoids
$  1\rightarrow \hat{Z} \hookrightarrow {\sf Cl}_{A_{\upsigma_{1}}}\twoheadrightarrow {\sf Cl}_{K}\rightarrow 1$ ;
 the subset $ Z_{0} = \{ \mathfrak{a}_{\ell_{+}}  \; :\;\;  \ell\in  \log_{\uptheta_{0}^{-1}} (\mathcal{O}_{K}) \} \subset\hat{Z}$
forms a dense {\it subgroup} of $\hat{Z}$.

 In the function field case,  the quantum modular invariant $j^{\rm qt}(f_{0})$ satisfies
 \begin{align}\label{jffform}  j^{\rm qt}(f_{0})=\{ j(\mathfrak{a}_{0}),\dots ,j(\mathfrak{a}_{d_{0}-1})\}  \end{align}
 where the $j(\mathfrak{a}_{i})$ are the $j$-invariants of the $\mathfrak{a}_{i}$.   See \cite{DGIII}, \cite{DGV}.   We may similarly define $j$-invariants of the $\mathfrak{a}_{x}$ (see \S 5 of \cite{GLL}),
$ j(\mathfrak{a}_{x}) \in\R$.
 In a private communication, R. Pink proved the following generalization of (\ref{jffform}):
 \begin{theo}[\cite{GLL}, \S 5]\label{Pink} Let $j^{\rm qt}(\uptheta_{0} )$ be the quantum modular invariant of $\uptheta_{0}$ {\rm (}as defined in \cite{Ge-C}{\rm )}.  Then 
 \[ \overline{j^{\rm qt}(\uptheta_{0} )} = \{ j(\mathfrak{a}_{x}) \; : \;\; \mathfrak{a}_{x}\in \hat{Z}\} .\]
 \end{theo}
 
 The point of using the Cantor set  $\widehat{[0,1)}$ to parametrize ideals, is that the function 
$x\longmapsto j(\mathfrak{a}_{x})$
 is {\it continuous}.  See Th\'{e}or\`{e}me 6 of \cite{GLL}.  Then as a Corollary, we have that the closure in $\R$ of $j^{\rm qt}(\uptheta_{0} )$ is the continuous image of a Cantor set.
 Using Pink's result, we may now state a conjecture about Hilbert class fields:
 
 \begin{conj1} Let $K/\Q$ be real quadratic.  Then there exists a measure $dm$ on the Cantor set $\widehat{[0,1)}$ such that
 \[ H_{K} = K \bigg( \Prodi _{x\in \widehat{[0,1)} } j (\mathfrak{a}_{x}) ^{dm}\bigg) \]
 where the expression in the parentheses is some notion of product integral \cite{Dollard} defined on $\widehat{[0,1)}$.
 \end{conj1}
 
In the function field case, to every fractional ideal $\mathfrak{a}\subset \C_{\infty}$ over $A_{\infty_{1}}$, we may associate the quotient
$ \D_{\mathfrak{a}} := \C_{\infty}/\mathfrak{a} $
 as well as an isomorphism of abelian groups given by the exponential 
 $  \exp_{\mathfrak{a}} :  \D_{\mathfrak{a}} \longrightarrow (\C_{\infty}, +) $, which in turn
gives rise to a Drinfeld module 
$ \uprho: A_{\infty_{1}} \longrightarrow \C_{\infty} \{ \uptau\}$.  The normalized version, which involves scaling the lattice $\mathfrak{a}$ by a special transcendental element 
 $\upxi_{\mathfrak{a}}$, defines a Hayes module.  
 
 In characteristic zero,
 given a (fractional) quasicrystal ideal $\mathfrak{a}\subset \R$ over $A_{\upsigma_{1}}$, the analog of the ``analytic'' Drinfeld module $\D_{\mathfrak{a}}$ is the completion 
 \[ \hat{ \SI}_{\mathfrak{a}} =\overline{\{ r + \mathfrak{a} \; : \;\; r\in \R\}}  \]
 of the set of translates $r+\mathfrak{a}$ in the space of 1-dimensional quasicrystals in $\R$, see for example \cite{BBG}. This completion $\hat{\SI}_{\mathfrak{a}} $ has the structure of a 1-dimensional solenoid, equipped with 
a canonical Cantor transversal: the completion of $\mathfrak{a}$,
 \[  \hat{\mathfrak{a}}   =  \overline{\{ \upalpha + \mathfrak{a} \; : \;\; \upalpha \in \mathfrak{a}\}} .\]
 We view $\hat{\mathfrak{a}}$ as a kind of ``thickening'' of the additive identity $0$.  There is a natural multiplicative action of $A_{\upsigma_{1}}$  on $\hat{ \SI}_{\mathfrak{a}}$,
 and accordingly, we say that $\hat{ \SI}_{\mathfrak{a}}$ has {\bf {\em quasicrystal multiplication}}.  See \S 6 of \cite{GLL} for more details.
 
In \S 7 of \cite{GLL},  normalized exponentials of the quasicrystal ideals $\mathfrak{a}_{x}$ are defined
$ e_{\mathfrak{a}_{x}}: \C\longrightarrow \C$, $e_{\mathfrak{a}_{x}}(z) = c_{\mathfrak{a}_{x}}(z) + is_{\mathfrak{a}_{x}}(z)$,
where as before
\[   s_{\mathfrak{a}_{x}}(z)=  z\prod_{0<\upalpha\in \mathfrak{a}_{x}} \left (1-\frac{z^{2}}{ (\uppi_{x}\upalpha)^{2} }\right) ,  \quad c_{\mathfrak{a}_{x}}(z) = s_{\mathfrak{a}_{x}}'(z) \]
and 
\[ \uppi_{\mathfrak{a}_{x}} = \frac{1}{\upbeta_{0}}   \prod_{i=1}^{\infty} \frac{\upalpha_{i}^{2}}{\upbeta_{i} \upbeta_{i+1}} , \]
where now the $\upalpha_{i}$ are the positive elements of $\mathfrak{a}_{x}$ and $\upbeta_{i}$ are the positive roots of $c_{x}(z)$.  If we restrict to $\R\subset \C$
the exponential extends to a continuous map 
\[  e_{\mathfrak{a}_{x}}: \hat{\SI}_{\mathfrak{a}_{x}} \longrightarrow \C. \]
See Th\'{e}or\`{e}me 12 of \cite{GLL}.

We now state and prove the analog of Theorem \ref{qtexplimit} of this paper:

\begin{theo}\label{valueseqtchar0} For all $z\in\C$, the closure of $\exp^{\rm qt}(z)$ in $\R$ is
\[ \overline{\exp^{\rm qt}(z)} = \left\{ e_{\mathfrak{a}_{x}}(z)|\; x\in \widehat{[0,1)}\right\}.\]
\end{theo}

\begin{proof}  It is enough to prove that the multi-points of $\exp^{\rm qt}(z)$ coincide with the set $\{  e_{\mathfrak{a}_{x}}(z)\}$ for
$x\not=\ell_{+}$.  Consider for fixed $x\in [0,1)$ the values
$  \upvarepsilon_{x,m} = \uptheta_{0}^{-x-m}, \quad m = 0,1,2,\dots $.
Let us write correspondingly
$\uppi_{x,m}$,  $\Uplambda_{x,m}$, $ \breve{\Uplambda}_{x,m}$, $e_{x,m}(z)$
for the associated ``pi'', the corresponding $\Uplambda$, the corresponding normalized $\breve{\Uplambda}$ and the normalized exponential.  
Let $ \Updelta = \uptheta_{0} -\uptheta_{0}'$ where as usual $\uptheta_{0}'$ is the Galois conjugate of $\uptheta_{0}$. 
We first note that the following renormalization scheme applied to $\Uplambda_{x,m}$ gives $\mathfrak{a}_{x}$:
$\lim_{m\rightarrow\infty} ( \Updelta/\uptheta_{0}^{m} )\Uplambda_{x,m}  = \mathfrak{a}_{x}$
This is the exact analog of the renormalization scheme used in the function field setting \cite{DGIII}; it is shown in the course of the proof of Theorem \ref{Pink} above that appears in
 \cite{GLL}.  
This implies the convergence 
\[   \frac{\Updelta}{\uptheta_{0}^{m} } \sin_{x,m} \left(\frac{\uptheta_{0}^{m} }{\Updelta} z\right)   = z\prod_{ 0<\upalpha\in \Uplambda_{x,m} }  \left (1-\frac{z^{2}}{ ( \Updelta \uptheta_{0}^{-m}\upalpha)^{2} }\right) \longrightarrow
\sin_{\mathfrak{a}_{x}} (z),\]
where $\sin_{x,m} (z)$, $\sin_{\mathfrak{a}_{x}} (z)$ are the un-normalized sine functions of $\Uplambda_{x,m}$,  $\mathfrak{a}_{x}$.  Taking derivatives of both sides gives 
a similar result for the cosines:
\begin{align}\label{cosconv}  \cos_{x,m} \left(\frac{\uptheta_{0}^{m} }{\Updelta} z\right) \longrightarrow
\cos_{\mathfrak{a}_{x}} (z).   \end{align}
  In particular, we have that if $\Upomega_{x,m} $ is the zero set of $\cos_{x,m}(z)$, then (\ref{cosconv})  implies 
$   \lim_{m\rightarrow\infty} (\Updelta/\uptheta_{0}^{m} )\Upomega_{x,m} = \Upomega_{\mathfrak{a}_{x}} $
 where $\Upomega_{\mathfrak{a}_{x}}$ is the zero set of $\cos_{\mathfrak{a}_{x}}(z)$.
 Therefore,
 \[ \frac{\uppi_{x,m}}{\Updelta\uptheta_{0}^{-m}} = \frac{1}{\upbeta_{0} \Updelta\uptheta_{0}^{-m}}    \prod_{i=1}^{\infty} \frac{ (\Updelta\uptheta_{0}^{-m}\upalpha_{i})^{2}}{ (\Updelta\uptheta_{0}^{-m}\upbeta_{i} )(\Updelta\uptheta_{0}^{-m}\upbeta_{i+1})}\longrightarrow \uppi_{\mathfrak{a}_{x}}.  \]
 Thus, if we normalize by $\uppi_{x,m}$ we get 
 \[  s_{x,m}(z)   = z\prod_{ 0<\upalpha\in \Uplambda_{x,m} }  \left (1-\frac{z^{2}}{ ( \uppi_{x,m}\upalpha)^{2} }\right)  \sim z\prod_{ 0<\upalpha\in \Uplambda_{x,m} } 
  \left (1-\frac{z^{2}}{ ( \uppi_{\mathfrak{a}_{x}}\Updelta\uptheta_{0}^{-m}\upalpha)^{2} }\right) \longrightarrow s_{\mathfrak{a}_{x}} (z)\]
  with a similar convergence of cosines
$  c_{x,m}(z)\longrightarrow c_{\mathfrak{a}_{x}}(z) $.
  This proves the result.
\end{proof}

We now have all that we need to state a conjecture generalizing the Main Theorem of this paper.   Denote $\uptheta_{0}\in \mathcal{O}_{K} $ a fundamental unit and let  $\uptheta=\uptheta_{0}^{n}$ be a power with norm $1$.
We are interested in the order $\mathcal{O}_{\uptheta} = \Z [ \uptheta, \uptheta^{-1}]$.  Let 
$A_{\uptheta} = A_{\upsigma_{1}} \cap \mathcal{O}_{\uptheta} $
be the associated ``quasicrystal order''.
The material presented in the above paragraphs generalizes in a straightforward way to structures defined over the order $A_{\uptheta}$: in particular, we may speak of ``Drinfeld modules'' over $A_{\uptheta}$, 
exponentials and solenoids associated to quasicrystal ideals $\mathfrak{a}\subset A_{\uptheta}$, and so on.

We consider the family of ideals 
\[ \mathcal{M}_{\uptheta} = \bigg\{ \mathfrak{M}_{\uptheta} \subset \mathcal{O}_{\uptheta} \; :\;\;   \uptheta\equiv 1\mod \mathfrak{M}_{\uptheta} \bigg\} . \]  
As before, if ${\rm id}\not=\varphi\in {\rm Gal}( K/\Q )$, then $\mathfrak{M}_{\uptheta}\in  \mathcal{M}_{\uptheta}$ $\Leftrightarrow $ $\mathfrak{M}^{\varphi}_{\uptheta}\in  \mathcal{M}_{\uptheta}$.
Write $\mathcal{M}= \bigcup_{N(\uptheta )=1} \mathcal{M}_{\uptheta} $.
The existence of the unit ray class field $K_{\mathfrak{M}_{\uptheta}}$ follows as before by class field theory, and the family 
$ \mathcal{K}_{\rm unit} = \{ K_{\mathfrak{M}_{\uptheta}}  :\; \mathfrak{M}_{\uptheta}\in  \mathcal{M}\}$
is again cofinal.   We may
write
$ K_{\mathfrak{M}_{\uptheta}} =K_{\mathfrak{M}_{\uptheta},1}\cdot K_{\mathfrak{M}_{\uptheta},2} $,
where $K_{\mathfrak{M}_{\uptheta},i}$, $i=1,2$, means narrowness along the corresponding place $\upsigma_{i}$ only.

Consider the contraction
$ \mathfrak{m}_{\uptheta}= \mathfrak{M}_{\uptheta}\cap A_{\uptheta} \subset A_{\uptheta}$,
which is a quasicrystal ideal in the sense described above.
For each quasicrystal ideal $\mathfrak{a}_{x}$, the multiplicative action of $\mathfrak{m}_{\uptheta}\subset A_{\uptheta}$ on the associated solenoid $\hat{\SI}_{\mathfrak{a}_{x}}$ was described in the previous
subsection.   We say that $t\in \hat{\SI}_{\mathfrak{a}_{x}}$
is an $\mathfrak{m}_{\uptheta}$ torsion point if for all $\upalpha\in \mathfrak{m}_{\uptheta}$, 
$ \upalpha\cdot t \in \hat{ \mathfrak{a}}_{x}$. 
The set of $\mathfrak{m}_{\uptheta}$ torsion points defines an {\it almost} $A_{\uptheta}$ module: see \S 6 of \cite{GLL} for precise definitions.  We denote this almost module by
$  \hat{\SI}_{\mathfrak{a}_{x}} [\mathfrak{m}_{\uptheta}]$.  

In the positive characteristic case, it was necessary to normalize the quantum exponential in order that its multipoints form a Galois orbit, see \S \ref{QDMSection}.  The normalization consists essentially of replacing the transcendental factor $\upxi_{i}$ of the Hayes module $\uprho_{i}$ by $\upxi_{0}$
and normalizing by the derivative $D_{i}$ of $\Upphi_{i}$. 
In characteristic zero, the analog of $D_{i}$  would be some kind of product integral
\begin{align}\label{CantorProductDx}  D_{x} = \Prodi_{0\not=\upalpha\in \widehat{ \mathfrak{a}_{x}/\mathfrak{a}_{0}}} e_{0}(\uppi_{\mathfrak{a}_{0}} \upalpha )^{d\upnu} , \end{align}
where $d\upnu$ is a  measure on
the quasicrystal quotient
\[ \widehat{ \mathfrak{a}_{x}/\mathfrak{a}_{0}}:= \overline{ \{ r + \mathfrak{a}_{0} \; : \;\; r\in  \mathfrak{a}_{x} \} } : \]
which is, like  $\hat{\mathfrak{a}}_{x}$, a Cantor set.   In particular, the status of the formula (\ref{CantorProductDx}) is conjectural, awaiting an appropriate definition of $d\upnu$.
This is very much in the spirit of the conjectured multiplicative measure $dm$ on $\widehat{[0,1)}$ needed to formulate the ``Cantor product'' occurring in 
 Main Conjecture 1, stated above.
Assuming the existence of this measure, which would define $D_{x}$, we
may renormalize the exponentials as in the positive characteristic case,  to obtain $\widetilde{e}_{\mathfrak{a}_{x}}$.   The following is the ray class field counterpart of Main Conjecture 1:


\begin{conj2}    
There exists an {\rm additive} measure $d\upmu$ on $\widehat{[0,1)}$ such that $K_{\mathfrak{M}_{\uptheta},1}$ is generated by
the averages of quantum torsion points:
\[    {\rm Tr}( \exp^{\rm qt} (t) ) := \int_{\widehat{[0,1)}} \widetilde{e}_{\mathfrak{a}_{x}} (t)  \; d\upmu (x)  \]

\end{conj2}

If this conjecture is verified, we may, following the discussion at the end of \S \ref{RayClassFieldSection}, obtain an explicit description of $K_{\mathfrak{M}_{\uptheta}}$.

\section*{Appendix: The Cyclotomic Element of the Conductor of $A_{f}$}

In what follows will denote
the fundamental unit $f_{0}$ simply as $f$, and the general power unit as $f^{k}$.  As in (\ref{minpolyoff}), the fundamental unit $f$ will be
normalized to satisfy
$f^{2} =  {\tt a} f + {\tt b} $, with $  {\tt a} \in \F_{q}[T]$ monic and $  {\tt b} \in \F_{q}^{\times}$.
The degree of $ {\tt a} $ will be denoted $d$ rather than $d_{0}$.  With this notation, 
$A_{f^{k}} = \F_{q}[ f^{k}, f^{k}T,\dots , f^{k}T^{dk-1}]$.

Consider the polynomial 
      \[   P_{k} (X) = X^{k-1} -  {\tt b} X^{k-2} +  {\tt b} ^{2}X^{k-3}  -\cdots + (-1)^{k-1}  {\tt b} ^{k-1}.  \]
     and denote by $\mathfrak{c}_{f^{k}}\subset A_{f^{k}}$ the conductor of $A_{f^{k}}\subset A_{\infty_{1}}$. Denote
$ \Uptheta_{k}:= P_{k}(f^{2} ) $.
      
      \begin{theo}\label{PhiInConductor}For all $k$, $ \Uptheta_{k}\in \mathfrak{c}_{f^{k}}$.      \end{theo}
      
      \begin{proof}  For $k=2$, clearly
     $  \Uptheta_{2} = f^{2} - {\tt b} \in A_{f^{2}}$.  Moreover, \[ f \Uptheta_{2} = f (f^{2}- {\tt b} )  =   f^{2} {\tt a} ,\;  fT \Uptheta_{2}=f^{2}T {\tt a}  ,\dots ,  fT^{d-1} \Uptheta_{2}=f^{2}T^{d-1} {\tt a}  \in A_{f^{2}}\]
     so $ \Uptheta_{2}\in\mathfrak{c}_{f^{k}}$.
     Similarly, 
     $  \Uptheta_{3} =$  $f^{4}- {\tt b} f^{2} + {\tt b} ^{2} =$ $ f^{2} ( {\tt a} f + {\tt b} ) - {\tt b} f^{2} + {\tt b} ^{2} =$ $  {\tt a} f^{3}+ {\tt b} ^{2} \in A_{f^{3}}$
and for $i=0,\dots, d-1$,      
$ f T^{i} \Uptheta_{3}   =$   $ T^{i} \big(  {\tt a}  ( {\tt a} f + {\tt b} ) f^{2}+f {\tt b} ^{2} \big) =$ $ T^{i} \big(  {\tt a} ^{2}f^{3}+  {\tt b}  {\tt a} f^{2}+f {\tt b} ^{2} \big) 
  =$ $   {\tt a} ^{2} f^{3}T^{i} +  {\tt b} f^{3}T^{i}  \in A_{f^{3}}$,
since $\deg( {\tt a} ^{2}T^{i})\leq 3d-1$.
In this connection we claim that
\[  f \Uptheta_{k} = \left\{ \begin{array}{ll}  
f^{k} \sum_{i=0}^{k/2-1} \binom{k-1-i}{i}  {\tt b} ^{i}  {\tt a} ^{k-1-2i} & \text{if $k$ is even} \\
\\
f^{k} \sum_{i=0}^{(k-1)/2} \binom{k-1-i}{i}  {\tt b} ^{i}  {\tt a} ^{k-1-2i} & \text{if $k$ is odd}
\end{array}
\right.   .\]
From this formula it follows that $f \Uptheta_{k} , \dots , fT^{d-1} \Uptheta_{k} \in A_{f^{k}}$.  This almost shows  that $ \Uptheta_{k}\in \mathfrak{c}_{f^{k}}$: we need 
$ \Uptheta_{k} \in A_{f^{k}}$ as well, which will be shown using a similar formula below. The proof will be an induction, alternating along the even and odd cases.  
Suppose first that $k$ is odd and assume that the formula is true at $k$: then  
we want to show 
\begin{align*}   f \Uptheta_{k+1}  & =f^{k+1} \sum^{\frac{k-1}{2}}_{i=0} \binom{k-i}{i}  {\tt b} ^{i}  {\tt a} ^{k-2i} \\
& =f^{k+1} \sum^{\frac{k-1}{2}}_{i=0} \binom{k-1-i}{i}  {\tt b} ^{i}  {\tt a} ^{k-2i} +f^{k+1}\sum^{\frac{k-1}{2}}_{i=0} \binom{k-1-i}{i-1}  {\tt b} ^{i}  {\tt a} ^{k-2i} \\
& \\
& =: \text{\ding{202}} +  \text{\ding{203}} 
 \end{align*} 
For $k$ odd, from the definition of $ \Uptheta_{k}$, we have
$  \Uptheta_{k+1} =f^{2} \Uptheta_{k} - {\tt b} ^{k}$.  Thus
\begin{align*}
f \Uptheta_{k+1} & = f^{2}\bigg[ f \Uptheta_{k}\bigg] -f {\tt b} ^{k}  = f^{2}\cdot f^{k} \sum^{\frac{k-1}{2}}_{i=0} \binom{k-1-i}{i}  {\tt b} ^{i}  {\tt a} ^{k-1-2i}  - {\tt b} ^{k}f \\
&=f^{k+1} \sum^{\frac{k-1}{2}}_{i=0} \binom{k-1-i}{i} {\tt b} ^{i}  {\tt a} ^{k-2i} + {\tt b} f^{k}  \sum^{\frac{k-1}{2}}_{i=0} \binom{k-1-i}{i} {\tt b} ^{i}  {\tt a} ^{k-1-2i} - {\tt b} ^{k}f \\
& = \text{\ding{202}}  + {\tt b} f^{k}  \sum^{\frac{k-1}{2}}_{i=0} \binom{k-1-i}{i} {\tt b} ^{i} {\tt a} ^{k-1-2i} - {\tt b} ^{k}f .
\end{align*}
The remaining sum above may be written
\begin{align*} f^{k}  \sum^{\frac{k+1}{2}}_{i=1} \binom{k-i}{i-1} {\tt b} ^{i} {\tt a} ^{k+1-2i} - {\tt b} ^{k}f 
& = f^{k-1} (f^{2}- {\tt b} )   \sum^{\frac{k-1}{2}}_{i=1} \binom{k-i}{i-1} {\tt b} ^{i} {\tt a} ^{k-2i}  + f^{k}  {\tt b} ^{\frac{k+1}{2}} - {\tt b} ^{k}f \\
& = \text{\ding{203}} + f^{k+1}   \sum^{\frac{k-1}{2}}_{i=1} \binom{k-1-i}{i-2} {\tt b} ^{i} {\tt a} ^{k-2i}  \\
&- f^{k-1}  \sum^{\frac{k-1}{2}}_{i=1} \binom{k-i}{i-1} {\tt b} ^{i+1} {\tt a} ^{k-2i} + f^{k}  {\tt b} ^{\frac{k+1}{2}}- {\tt b} ^{k}f .
 \end{align*}
 Therefore, for $k$ odd, we are reduced to showing that 
 \begin{align}\label{odd}
 & f^{k+1}   \sum^{\frac{k-1}{2}}_{i=1} \binom{k-1-i}{i-2} {\tt b} ^{i} {\tt a} ^{k-2i}  - f^{k-1}  \sum^{\frac{k-1}{2}}_{i=1} \binom{k-i}{i-1} {\tt b} ^{i+1} {\tt a} ^{k-2i}   \\
 &+ f^{k}  {\tt b} ^{\frac{k+1}{2}}- {\tt b} ^{k}f =0.  \nonumber
 \end{align}
 
Let us leave the case $k$ odd for the moment, and apply a similar analysis to the case when $k$ is even, where, once again, we assume the formula true.   
In this case, we want to show 
\begin{align*}   f \Uptheta_{k+1}  & =f^{k+1} \sum^{\frac{k}{2}}_{i=0} \binom{k-i}{i}  {\tt b} ^{i}  {\tt a} ^{k-2i} \\
& =f^{k+1} \sum^{\frac{k}{2}}_{i=0} \binom{k-1-i}{i}  {\tt b} ^{i}  {\tt a} ^{k-2i} +f^{k+1}\sum^{\frac{k}{2}}_{i=0} \binom{k-1-i}{i-1}  {\tt b} ^{i}  {\tt a} ^{k-2i} \\
& =f^{k+1} \sum^{\frac{k}{2}-1}_{i=0} \binom{k-1-i}{i}  {\tt b} ^{i}  {\tt a} ^{k-2i} +f^{k+1}\sum^{\frac{k}{2}}_{i=1} \binom{k-1-i}{i-1}  {\tt b} ^{i}  {\tt a} ^{k-2i} \\
& \\
& =: \text{\ding{202}} +  \text{\ding{203}} 
\end{align*}
By the recursion (for $k$ even)
$ \Uptheta_{k+1} =f^{2} \Uptheta_{k} + {\tt b} ^{k}$,
 we have
\begin{align*}
f \Uptheta_{k+1} & = f^{2}\bigg[ f \Uptheta_{k}\bigg] +f {\tt b} ^{k} = f^{2}\cdot f^{k} \sum^{\frac{k}{2}-1}_{i=0} \binom{k-1-i}{i}  {\tt b} ^{i}  {\tt a} ^{k-1-2i}  + {\tt b} ^{k}f \\
&=f^{k+1} \sum^{\frac{k}{2}-1}_{i=0} \binom{k-1-i}{i} {\tt b} ^{i}  {\tt a} ^{k-2i} + {\tt b} f^{k}  \sum^{\frac{k}{2}-1}_{i=0} \binom{k-1-i}{i} {\tt b} ^{i}  {\tt a} ^{k-1-2i} + {\tt b} ^{k}f \\
& = \text{\ding{202}}  + {\tt b} f^{k}  \sum^{\frac{k}{2}-1}_{i=0} \binom{k-1-i}{i} {\tt b} ^{i} {\tt a} ^{k-1-2i} + {\tt b} ^{k}f .
\end{align*}
The remaining terms in the last line above can be written
\begin{align*}
f^{k}  \sum^{\frac{k}{2}}_{i=1} \binom{k-i}{i-1} {\tt b} ^{i} {\tt a} ^{k +1-2i} + {\tt b} ^{k}f & = f^{k-1} (f^{2}- {\tt b} )  \sum^{\frac{k}{2}}_{i=1} \binom{k-i}{i-1} {\tt b} ^{i} {\tt a} ^{k -2i} + {\tt b} ^{k}f \\
& = \text{\ding{203}} + f^{k+1}  \sum^{\frac{k}{2}}_{i=1} \binom{k-1-i}{i-2} {\tt b} ^{i} {\tt a} ^{k -2i} \\
&\quad -f^{k-1}\sum^{\frac{k}{2}}_{i=1} \binom{k-i}{i-1} {\tt b} ^{i+1} {\tt a} ^{k -2i} + {\tt b} ^{k}f.
\end{align*}
So for $k$ even, we are reduced to showing
\begin{align}\label{even} f^{k+1}  \sum^{\frac{k}{2}}_{i=1} \binom{k-1-i}{i-2} {\tt b} ^{i} {\tt a} ^{k -2i}  -f^{k-1} \sum^{\frac{k}{2}}_{i=1} \binom{k-i}{i-1} {\tt b} ^{i+1} {\tt a} ^{k -2i} + {\tt b} ^{k}f=0. \end{align}

We now proceed to prove both equations (\ref{odd}) and (\ref{even}) by induction, assuming they are true for all lower values. That is, if we assume $k$ is even, we seek to show (\ref{even}), assuming it true for $k-2$  and assuming as well that  (\ref{odd}) holds for $k-1$, where the latter takes the form
\begin{align}\label{oddk-1} &  f^{k} \sum_{i=2}^{\frac{k}{2}-1}  \binom{k-2-i}{i-2} {\tt b} ^{i}  {\tt a} ^{k-1-2i}-f^{k-2}  \sum_{i=1}^{\frac{k}{2}-1}   \binom{k-1-i}{i-1} {\tt b} ^{i+1}  {\tt a} ^{k-1-2i}  \\
& + f^{k-1}  {\tt b} ^{\frac{k}{2}} -  {\tt b} ^{k-1}f =0 \nonumber \end{align}
Multiplying (\ref{oddk-1}) by $-f {\tt a} $ and adding to (\ref{even}) gives 
\begin{align*}  & f^{k+1}\sum_{i=2}^{\frac{k}{2}-1}   \binom{k-2-i}{i-3} {\tt b} ^{i}  {\tt a} ^{k-2i} -f^{k-1} \sum_{i=2}^{\frac{k}{2}-1}   \binom{k-1-i}{i-2} {\tt b} ^{i+1}  {\tt a} ^{k-2i} \\
 &+f^{k+1}\ \binom{\frac{k}{2}-1}{\frac{k}{2}-2}   {\tt b} ^{\frac{k}{2} } - f^{k-1}\binom{\frac{k}{2}}{\frac{k}{2}-1}   {\tt b} ^{\frac{k}{2} +1 } - f^{k} {\tt a}   {\tt b} ^{\frac{k}{2}} +  {\tt b} ^{k-1}f^{2}  {\tt a} + {\tt b} ^{k}f.
  \end{align*}
  Resolving the term $- f^{k} {\tt a}   {\tt b} ^{\frac{k}{2}}$ gives
$   -f^{k-1} (f^{2}- {\tt b} )  {\tt b} ^{\frac{k}{2}} = - f^{k+1} {\tt b} ^{\frac{k}{2}}+  f^{k-1}   {\tt b} ^{\frac{k}{2}+1}$.
  These can be combined with the isolated binomial terms, and re-incorporated in the $\Upsigma$ sums:
 \begin{align}\label{reincorp}  & f^{k+1}\sum_{i=2}^{\frac{k}{2}}   \binom{k-2-i}{i-3} {\tt b} ^{i}  {\tt a} ^{k-2i} -f^{k-1} \sum_{i=2}^{\frac{k}{2}}   \binom{k-1-i}{i-2} {\tt b} ^{i+1}  {\tt a} ^{k-2i}  \\
&  + {\tt b} ^{k-1}f^{2}  {\tt a}  + {\tt b} ^{k}f. \nonumber
  \end{align} 
  But the last two terms give 
  \begin{align}\label{resolutionoflast} f  {\tt b} ^{k-1} ( {\tt a} f +  {\tt b} )   = f^{3}  {\tt b} ^{k-1} . \end{align}
  Thus we may factor out $f^{2} {\tt b} $, and after re-indexing the sums we obtain $f^{2} {\tt b} $ times (\ref{even}) for the values $k+1$, $k-1$ replaced by $k-1$, $k-3$ etc.  By induction, this is zero,
  and we are done.  
  Now reverse the roles of even and odd.  We assume $k$ is odd and we want to show 
  (\ref{odd}), assuming it true for $k-2$  and assuming as well that  (\ref{even}) holds for $k-1$, where the latter takes the form
  \begin{align}\label{even-1} & f^{k}  \sum^{\frac{k-1}{2}}_{i=1} \binom{k-2-i}{i-2} {\tt b} ^{i} {\tt a} ^{k-1 -2i}  -f^{k-2} \sum^{\frac{k-1}{2}}_{i=1} \binom{k-1-i}{i-1} {\tt b} ^{i+1} {\tt a} ^{k -1-2i} \\ &+ {\tt b} ^{k-1}f=0.\nonumber \end{align}
  Adding $-f {\tt a} $ times (\ref{even-1}) to  (\ref{odd}) now gives
  \begin{align*}  & f^{k+1}\sum_{i=1}^{\frac{k-1}{2}} \binom{k-2-i}{i-3}  {\tt b} ^{i}  {\tt a} ^{k-2i} - f^{k-1}\sum_{i=1}^{\frac{k-1}{2}}\binom{k-i-1}{i-2} {\tt b} ^{i+1}  {\tt a} ^{k-2i} \\
  & + f^{k} {\tt b} ^{\frac{k+1}{2}} -f^{3} {\tt b} ^{k-1} \end{align*}
  where we have made use of (\ref{resolutionoflast}).  Factoring out $f^{2} {\tt b} $ as before gives (\ref{odd}) at the previous odd value, where it is true by the induction hypothesis.
  What remains is to show that $ \Uptheta_{k}\in A_{f^{k}}$.  In this regard, we claim that
\begin{align}\label{FormulaForPhi}   \Uptheta_{k} = f^{k} \sum_{i=0} ^{ [\frac{k}{2}] -1} \binom{k-2-i}{i}  {\tt b} ^{i}  {\tt a} ^{k-2-2i} + (- {\tt b} )^{k-1} ,\end{align}
from which it immediately follows that $ \Uptheta_{k}\in A_{f^{k}}$ for all $k$.  We again prove this
by an even-odd induction. Suppose the result is true for $k$ an even integer.  Then we want to prove (\ref{FormulaForPhi}) for $k+1$ i.e.\
\begin{align}\label{FormulaForPhikOdd}    \Uptheta_{k+1} = f^{k+1} \sum_{i=0} ^{ \frac{k}{2}-1} \binom{k-1-i}{i}  {\tt b} ^{i}  {\tt a} ^{k-1-2i} +  {\tt b} ^{k}  .  \end{align}
Now, the recursion $ \Uptheta_{k+1} = f^{2} \Uptheta_{k} + (-1)^{k+2} {\tt b} ^{k}$ gives the true formula
\begin{align}\label{trueformula}   \Uptheta_{k+1} & =  f^{k+2} \sum_{i=0} ^{ \frac{k}{2} -1} \binom{k-2-i}{i}  {\tt b} ^{i}  {\tt a} ^{k-2-2i}   -f^{2}  {\tt b} ^{k-1} +  {\tt b} ^{k}   ;\end{align}
we show that this is equivalent to (\ref{FormulaForPhikOdd}).  First, we rewrite (\ref{FormulaForPhikOdd}) as 
\begin{align*}
  \Uptheta_{k+1} & = f^{k} \cdot  {\tt a} f \sum_{i=0}^{\frac{k}{2}-1} \binom{k-1-i}{i}  {\tt b} ^{i}  {\tt a} ^{k-2-2i} + {\tt b} ^{k} \\
 & =  f^{k} \cdot (f^{2}- {\tt b} ) \sum_{i=0}^{\frac{k}{2}-1} \binom{k-1-i}{i}  {\tt b} ^{i}  {\tt a} ^{k-2-2i} + {\tt b} ^{k} \\
 & = f^{k+2}\sum_{i=0}^{\frac{k}{2}-1} \left\{  \binom{k-2-i}{i}+ \binom{k-2-i}{i-1} \right\}  {\tt b} ^{i} {\tt a} ^{k-2-2i} \\
 & \;\;- f^{k} \sum_{i=0}^{\frac{k}{2}-1} \binom{k-1-i}{i} {\tt b} ^{i+1}  {\tt a} ^{k-2-2i} + {\tt b} ^{k} \\
 &=(\ref{trueformula}) + f^{k+2} \sum_{i=0}^{\frac{k}{2}-1}\binom{k-2-i}{i-1} {\tt b} ^{i} {\tt a} ^{k-2-2i} \\
 & \;\;- f^{k} \sum_{i=0}^{\frac{k}{2}-1} \binom{k-1-i}{i} {\tt b} ^{i+1}  {\tt a} ^{k-2-2i} +f^{2}  {\tt b} ^{k-1} 
\end{align*}
Multiplying equation (\ref{even}) by $f {\tt b} ^{-1}$, we see that the last line is (\ref{trueformula}) $+$ 0, and we are done.
Now assume (\ref{FormulaForPhi}) is true for $k$ odd.  Then, 
we wish to deduce (\ref{FormulaForPhi}) for $k+1$, which takes the form
\begin{align}\label{FormulaForPhikEven}    \Uptheta_{k+1} = f^{k+1} \sum_{i=0} ^{ \frac{k-1}{2}} \binom{k-1-i}{i}  {\tt b} ^{i}  {\tt a} ^{k-1-2i} -  {\tt b} ^{k}  .  \end{align}
Again we use the recursion to obtain the true formula
\begin{align}\label{trueformulakodd}   \Uptheta_{k+1} & =  f^{k+2} \sum_{i=0} ^{ \frac{k-3}{2} } \binom{k-2-i}{i}  {\tt b} ^{i}  {\tt a} ^{k-2-2i}   +f^{2}  {\tt b} ^{k-1} - {\tt b} ^{k}   .\end{align}
We rewrite (\ref{FormulaForPhikEven})
as
\begin{align*}   \Uptheta_{k+1}  & =f^{k} (f^{2}- {\tt b} ) \sum_{i=0} ^{ \frac{k-3}{2} } \binom{k-1-i}{i}  {\tt b} ^{i}  {\tt a} ^{k-2-2i} + f^{k+1}  {\tt b} ^{\frac{k-1}{2}} - {\tt b} ^{k} \\
&= (\ref{trueformulakodd}) + f^{k+2} \sum_{i=0} ^{ \frac{k-3}{2} } \binom{k-2-i}{i-1}  {\tt b} ^{i}  {\tt a} ^{k-2-2i}  \\
& \;\; - f^{k}\sum_{i=0} ^{ \frac{k-3}{2} } \binom{k-1-i}{i}  {\tt b} ^{i+1}  {\tt a} ^{k-2-2i} + f^{k+1}  {\tt b} ^{\frac{k-1}{2}} -f^{2} {\tt b} ^{k-1}.\\
\end{align*}
Multiplying (\ref{odd}) by $f {\tt b} ^{-1}$, the last line is (\ref{trueformulakodd}), and we are done.
      \end{proof}

\end{document}